\documentclass[11pt]{amsart}
\usepackage{amsmath,amsfonts,amssymb,mathrsfs}

\usepackage[margin=1in]{geometry}
\usepackage{url}
\usepackage[shortlabels]{enumitem}

\usepackage[unicode,bookmarks]{hyperref}
\usepackage[dvipsnames]{xcolor}
\hypersetup{colorlinks=true,citecolor=NavyBlue,linkcolor=Brown,urlcolor=Orange}

\usepackage{tikz-cd,adjustbox}
\usepackage[arrow,curve,matrix]{xy}

\newif\ifdraft
\draftfalse

\newcommand{\tensor}{\otimes}
\newcommand{\isom}{\simeq}

\newcommand{\C}{\mathbb{C}}

\newcommand{\Q}{\mathbb{Q}}
\newcommand{\Z}{\mathbb{Z}}
\newcommand{\A}{\mathbb{A}}
\renewcommand{\P}{\mathbb{P}}
\newcommand{\DD}{\mathbb{D}}

\newcommand{\M}{\mathcal{M}}
\newcommand{\N}{\mathcal{N}}
\newcommand{\D}{\mathcal{D}}

\newcommand{\G}{\mathcal{G}}
\newcommand{\HH}{\mathbb{H}}
\newcommand{\I}{\mathcal{I}}
\renewcommand{\O}{\mathcal{O}}
\newcommand{\DR}{\mathrm{DR}}

\renewcommand{\L}{\mathcal{L}}
\newcommand{\codim}{\mathrm{codim}}
\newcommand{\w}{\omega}

\newcommand{\K}{\mathcal{K}}
\newcommand{\gr}{\mathrm{Gr}}
\newcommand{\Hom}{\mathcal{H}om}
\newcommand{\h}{\underline h}

\renewcommand{\sp}{\mathrm{sp}}

\newcommand{\DB}{\underline{\Omega}} 
\newcommand{\Sing}{\mathrm{Sing}}
\newcommand{\dual}{\mathbf{D}}
\newcommand{\Supp}{\mathrm{Supp}}

\newcommand{\sm}{\smallsetminus}

\newcommand{\IC}{\mathrm{IC}}
\newcommand{\wt}{\widetilde}
\newcommand{\lct}{\mathrm{lct}}
\newcommand{\Core}{\mathrm{Core}}
\newcommand{\IH}{\mathit{IH}}
\newcommand{\X}{\mathcal{X}}

\counterwithin{equation}{subsection}
\counterwithout{subsection}{section}
\counterwithin{figure}{subsection}

\newtheorem{thm}[equation]{Theorem}
\newtheorem{cor}[equation]{Corollary}
\newtheorem{lem}[equation]{Lemma}
\newtheorem{prop}[equation]{Proposition}

\theoremstyle{definition}
\newtheorem{defn}[equation]{Definition}
\newtheorem{question}[equation]{Question}

\theoremstyle{remark}
\newtheorem{rmk}[equation]{Remark}
\newtheorem{ex}[equation]{Example}

\theoremstyle{plain}

\newcommand{\theoremref}[1]{\hyperref[#1]{Theorem~\ref*{#1}}}
\newcommand{\lemmaref}[1]{\hyperref[#1]{Lemma~\ref*{#1}}}
\newcommand{\definitionref}[1]{\hyperref[#1]{Definition~\ref*{#1}}}
\newcommand{\propositionref}[1]{\hyperref[#1]{Proposition~\ref*{#1}}}
\newcommand{\conjectureref}[1]{\hyperref[#1]{Conjecture~\ref*{#1}}}
\newcommand{\corollaryref}[1]{\hyperref[#1]{Corollary~\ref*{#1}}}
\newcommand{\exampleref}[1]{\hyperref[#1]{Example~\ref*{#1}}}

\makeatletter
\let\old@caption\caption
\renewcommand*{\caption}[1]{%
	\setcounter{figure}{\value{equation}}%
	\stepcounter{equation}%
	\old@caption{#1}\relax%
}
\makeatother

\newcounter{intro}

\newtheorem{intro-conjecture}[intro]{Conjecture}
\newtheorem{intro-corollary}[intro]{Corollary}
\newtheorem{intro-theorem}[intro]{Theorem}

\makeatletter
\def\Ddots{\mathinner{\mkern1mu\raise\p@
\vbox{\kern7\p@\hbox{.}}\mkern2mu
\raise4\p@\hbox{.}\mkern2mu\raise7\p@\hbox{.}\mkern1mu}}
\makeatother

\begin{document}

\title[GIT stability and Hodge structures of hypersurfaces]{The GIT stability and Hodge structures of hypersurfaces via minimal exponent}

\author{Sung Gi Park}
\address{Department of Mathematics, Princeton University, and School of Mathematics, Institute for Advanced Study, Princeton, NJ, USA}
\email{sp6631@princeton.edu \,\,\,\,\,\,sgpark@ias.edu}



\thanks{}


\date{\today}

\begin{abstract}
Let $X\subset \mathbb P^n$ be a degree $d$ hypersurface. We prove that $X$ is GIT stable if the minimal exponent $\widetilde \alpha(X)>\frac{n+1}{d}$ and GIT semistable if $\widetilde \alpha(X)=\frac{n+1}{d}$, resolving a question of Laza. Conversely, for GIT semistable cubic hypersurfaces, we prove a uniform lower bound for the minimal exponent, which implies that every such cubic has canonical singularities (and is terminal for $n\ge 6$), answering a question of Spotti-Sun. In the classical cases $(n,d)=(2,4),(2,6),(3,3),(4,3),(5,3)$, the period map from the GIT moduli is an open embedding over the stable locus with $\widetilde \alpha(X)>\frac{n+1}{d}$ and extends regularly to the Baily-Borel compactification precisely along the boundary where $\widetilde \alpha(X)=\frac{n+1}{d}$.

To generalize this period map behavior in the Calabi-Yau type case $\frac{n+1}{d}=m+1\in \mathbb Z$, we introduce $m$-liminal sources and $m$-liminal centers, refining the theory of sources and log canonical centers. For an $m$-Du Bois hypersurface, we prove that the core of the limit mixed Hodge structure of any one-parameter smoothing is completely determined by the $m$-liminal source. In particular, maximal unipotent degeneration is detected by the local singularity type of the special fiber.

\end{abstract}

\maketitle

\makeatletter
\newcommand\@dotsep{4.5}
\def\@tocline#1#2#3#4#5#6#7{\relax
  \ifnum #1>\c@tocdepth 
  \else
    \par \addpenalty\@secpenalty\addvspace{#2}%
    \begingroup \hyphenpenalty\@M
    \@ifempty{#4}{%
      \@tempdima\csname r@tocindent\number#1\endcsname\relax
    }{%
      \@tempdima#4\relax
    }%
    \parindent\z@ \leftskip#3\relax
    \advance\leftskip\@tempdima\relax
    \rightskip\@pnumwidth plus1em \parfillskip-\@pnumwidth
    #5\leavevmode\hskip-\@tempdima #6\relax
    \leaders\hbox{$\m@th
      \mkern \@dotsep mu\hbox{.}\mkern \@dotsep mu$}\hfill
    \hbox to\@pnumwidth{\@tocpagenum{#7}}\par
    \nobreak
    \endgroup
  \fi}
\def\l@section{\@tocline{1}{0pt}{1pc}{}{\bfseries}}
\def\l@subsection{\@tocline{2}{0pt}{25pt}{5pc}{}}
\makeatother


\tableofcontents

\section{Introduction}\label{scn:intro}

This paper studies GIT stability and Hodge theory of degenerations of hypersurfaces using the minimal exponent and recently developed notions of higher singularities. In the literature, the GIT moduli spaces and their period maps have been studied, mainly for low-dimensional and low-degree hypersurfaces \cite{Shah80, Shah81, ACT02, LS07, Artebani09, Laza09, Looijenga09, Laza10, ACT11, LOG18}, based on the explicit classification of GIT stable (and semistable) hypersurfaces and the study of limit mixed Hodge structures of one-parameter degenerations.

Beyond low-dimensional and low-degree hypersurfaces, the explicit analysis of the GIT moduli space for hypersurfaces is hardly known. One of the major difficulties lies in determining which hypersurfaces are GIT (semi)stable, and the other lies in analyzing the Hodge structures when hypersurfaces degenerate. We aim to overcome these difficulties and establish new results.

Throughout the text, a variety is a reduced separated scheme of finite type over $\C$. All pure Hodge structures are assumed polarizable, and all mixed Hodge structures are graded polarizable. We use a decreasing filtration $F^\bullet$ and an increasing filtration $F_\bullet$, related by $F^p=F_{-p}$; we freely pass between these conventions.

\smallskip
\noindent
{\bf The GIT stability and extension of period map via minimal exponent.} The GIT stability of a hypersurface $X\subset \P^n$ of degree $d$ has been studied for decades from the perspective of singularities of $X$. While providing a complete classification of GIT (semi)stable hypersurfaces is essentially impossible aside from special cases with small $n$ and $d$, many results were established on sufficient conditions for (semi)stability. Here is a list of results from the literature:

\begin{itemize}
\item $d\ge 3$: $X$ is stable if $X$ is smooth (\cite{Mumford94}).
\item $d\ge n+1$: $X$ is stable (resp. semistable) if $\lct(\P^n,X)>\frac{n+1}{d}$ (resp. $\ge$) (\cite{Hacking04, KL04} for plane curves and \cite{Lee08} for hypersurfaces).
\item $d= n+1$: $X$ is stable if $X$ has canonical singularities (\cite{Tian94}).
\end{itemize}

For Fano hypersurfaces ($d\le n$), not much was known about the explicit stability of singular hypersurfaces, even with mild singularities. Recently, sufficient conditions in terms of the multiplicity and the dimension of the singular locus were established in \cite{Mordant24, He25}.

We prove a new criterion for GIT (semi)stability that recovers and generalizes the above classical results in terms of a singularity invariant, namely the \emph{minimal exponent} $\wt\alpha(X)$. This invariant is the negative of the greatest root of the reduced Bernstein-Sato polynomial of hypersurface singularities, defined by Saito \cite{Saito93}.

\begin{intro-theorem}\label{thm:GIT stability via minimal exponent}
A hypersurface $X\subset \P^n$ of degree $d\ge 3$ is GIT stable (resp. semistable) if $\wt\alpha(X)>\frac{n+1}{d}$ (resp. $\ge$).
\end{intro-theorem}

For example, every nodal hypersurface of degree $\ge3$ is GIT stable. Note that the minimal exponent $\wt\alpha(X)$ refines the log canonical threshold:
$$
\lct(\P^n,X)=\min\left\{\wt\alpha(X),1\right\}.
$$
Moreover, $X$ has canonical singularities if and only if $\wt\alpha(X)>1$ \cite{Saito93}. With the convention $\wt\alpha(X)=\infty$ for smooth $X$, Theorem \ref{thm:GIT stability via minimal exponent} immediately recovers the results mentioned above.

Conversely, for GIT semistable cubic hypersurfaces, we prove a uniform lower bound on the minimal exponent. In particular, every such $X$ has canonical singularities, thereby answering Question 5.8 of Spotti-Sun \cite{SS17} in the affirmative. Previously, this was known for cubic surfaces, threefolds, and fourfolds ($n\le 5$), via explicit GIT analyses \cite{Allcock03, Yokoyama02,Yokoyama08, Laza09, Laza10} or through the equivalence between GIT stability and K-stability \cite{OSS16,LX19,Liu22}. By contrast, no analogous bound exists for degree $d\ge 4$: products of GIT semistable quadrics and cubics remain GIT semistable, so reducible GIT semistable hypersurfaces occur.

\begin{intro-theorem}\label{thm:cubic converse via minimal exponent}
For $n\ge 3$, every GIT semistable cubic hypersurface $X\subset \P^n$ satisfies
$$
\wt\alpha(X)\ge \max\left\{\frac{4}{3},\frac{n+1}{9}\right\}.
$$
Furthermore, if $n\ge 6$, then $\wt\alpha(X)\ge\frac{5}{3}$ and $X$ has terminal singularities.
\end{intro-theorem}

For $n\le 5$, there exist GIT semistable cubics with canonical but non-terminal singularities; the bound $n\ge 6$ is sharp for terminality. For explicit examples and for sharp bounds on the minimal exponent when $n\le 5$, see Remark \ref{rmk:boundary cubics}. Notably, the minimal log discrepancy of a hypersurface singularity is greater than $k$ if its minimal exponent is greater than $1+\frac{k}{2}$; see Proposition \ref{prop:minimal exponent vs mld}.

For small $n$ and $d$, explicit GIT classifications are available and key to understanding both the birational geometry of the GIT moduli space and the relevant period map; depending on $(n,d)$, the latter is induced by the Hodge structure of a suitable Calabi-Yau type cyclic cover.

Although the detailed study of period maps is highly case-dependent, the minimal exponent provides a uniform description of the indeterminacy locus. For each classical pair $(n,d)\in\left\{(2,4),(2,6),(3,3),(4,3),(5,3)\right\}$, let
$$
\Phi:\P^{\binom{n+d}{d}-1} \dashrightarrow (\Gamma\backslash D)^*
$$
be the period map from the projective parameter space of degree $d$ hypersurfaces in $\P^n$ to the Baily-Borel compactification of the respective period domain $\Gamma\backslash D$. This period map descends to the GIT moduli space
$$
\mathcal P:\overline\M^{\rm GIT} \dashrightarrow (\Gamma\backslash D)^*.
$$
We summarize the results from the literature with an input from Theorem \ref{thm:GIT stability via minimal exponent}.

\begin{intro-corollary}[{\cite{Shah80, Kondo00, ACT02, LS07, Artebani09, Looijenga09, Laza10, ACT11}}]
\label{cor:classical extension results}
Fix $(n,d)$ in the above list, and denote by the open sets
$$
U :=\left\{[X]\in \P^{\binom{n+d}{d}-1} \mid \wt\alpha(X)\ge \frac{n+1}{d} \right\},\quad V:= \left\{[X]\in \P^{\binom{n+d}{d}-1} \mid \wt\alpha(X)> \frac{n+1}{d} \right\}.
$$

\noindent 
(1) The period map $\Phi$ extends to a regular morphism on $U$.

\noindent
(2) The inverse image $\Phi^{-1}(\Gamma\backslash D)$ in $U$ is $V$.

\noindent
(3) Let $\pi:\left(\P^{\binom{n+d}{d}-1}\right)^{ss}\to\overline\M^{\rm GIT}$ be the GIT quotient. Then $\mathcal P|_{\pi(V)}:\pi(V)\to \Gamma\backslash D$ is an open embedding and the indeterminacy locus of $\mathcal P$ is the closed set $Z$ of GIT polystable hypersurfaces with $\wt\alpha(X)<\frac{n+1}{d}$:
$$
Z:=\pi\left(\left(\P^{\binom{n+d}{d}-1}\right)^{ss}\sm U\right)\subset \overline\M^{\rm GIT}.
$$
\end{intro-corollary}

Using the description of limit mixed Hodge structures provided in Theorem \ref{thm:core of liminal sources}, we give direct proofs of (1) and (2) (except for cubic surfaces). Statement (3), on the other hand, follows from the cited literature, where it relies on a subtle and technical analysis of the lattice theory and the period map. Aside from the computation of minimal exponents, the only genuinely new result is for cubic fourfolds: answering a question posed by Laza at the 2024 AIM workshop ``Higher Du Bois and Higher Rational Singularities," we show that $1$-Du Bois cubic fourfolds are GIT semistable, and that the period map extends over their parameter space $U$ as stated.

Beyond the classical cases, Bakker-Filipazzi-Mauri-Tsimerman \cite{BFMT25} have recently announced a construction of Baily-Borel compactifications. In particular, for the moduli stack $\mathcal Y$ of polarized klt log Calabi-Yau pairs with coarse space $Y$, they obtain a unique normal compactification $Y^{\rm BBH}$ such that the Hodge bundle extends amply and satisfies the natural extension property along normal crossing boundaries. Along the same line, we conjecture an analogous Baily-Borel compactification for Calabi-Yau type hypersurfaces, accompanied by an extension statement parallel to Corollary \ref{cor:classical extension results}. For any pair $(n,d)$, we expect a reduction to this conjecture after an appropriate cyclic cover, as in the classical cases; see Question \ref{question:BB compactification and period map}.

\begin{intro-conjecture}
\label{conj:BB compactification for Calabi-Yau type}
Let $(n,d)$ be a pair with $\frac{n+1}{d}=m+1\in \Z$. Then the GIT moduli space $M$ of degree $d$ hypersurfaces in $\P^n$ with $m$-rational singularities admits the Hodge-theoretic compactification $M^{\rm BBH}$. Moreover, the rational map
$$
\Phi:\P^{\binom{n+d}{d}-1}\dashrightarrow M^{\rm BBH}
$$
is a regular morphism on the locus parameterizing $m$-Du Bois hypersurfaces.
\end{intro-conjecture}

The boundary behavior of the Hodge structures at the threshold $\frac{n+1}{d}$ in Theorem \ref{thm:GIT stability via minimal exponent} and Corollary \ref{cor:classical extension results} is captured by the Hodge-theoretic objects, namely \emph{liminal sources}. Building on recent advances in higher Du Bois and higher rational singularities, we describe limit mixed Hodge structures via liminal sources, in parallel with the classical description via sources of log canonical centers. This provides a first step toward Conjecture \ref{conj:BB compactification for Calabi-Yau type} and toward understanding when (and how) the period map extends to the conjectural Baily-Borel compactification. We focus on the integral case $\frac{n+1}{d}\in \Z$, since other cases are expected to reduce to this one.

\smallskip
\noindent
{\bf Liminal sources of Calabi-Yau type hypersurfaces and Hodge structures.} A degree $d$ hypersurface $X\subset \P^n$ is called \emph{Calabi-Yau type} if $\frac{n+1}{d}=m+1\in \Z$. This term stems from the fact that the middle cohomology $H^{n-1}(X,\Q)$ resembles that of Calabi-Yau varieties of dimension $n-2m-1$: when $X$ is smooth, the Hodge numbers are
$$
h^{m,n-1-m}(X)=h^{n-1-m,m}(X)=1\quad\textrm{and}\quad h^{i,n-1-i}(X)=h^{n-1-i,i}(X)=0\quad \forall\; i<m.
$$
In terms of higher singularities, Theorem \ref{thm:GIT stability via minimal exponent} states that $X$ is GIT stable if $X$ is $m$-rational and $X$ is GIT semistable if $X$ is $m$-liminal -- that is, $m$-Du Bois but not $m$-rational (see \cite{FL24l} for definitions).

The difference between $m$-liminal singularities and $m$-rational singularities is encoded in certain simple perverse subquotients -- called $m$-liminal sources -- of the constant perverse sheaf $\Q_X[\dim X]$. By Saito's theory of mixed Hodge modules \cite{Saito88, Saito90}, every simple subquotient of $\Q_X[\dim X]$ is the minimal extension of a polarizable variation of $\Q$-Hodge structure. More precisely, if a simple subquotient has support $Z\subset X$, then its underlying pure Hodge module is $\IC^H_Z(\mathbb V)$, where $(\mathbb V, F^\bullet)$ is an irreducible polarizable variation of $\Q$-Hodge structure defined on a smooth Zariski dense open subset of $Z$.

\smallskip
\noindent
{\bf Definition.} Let $X$ be a variety with $m$-Du Bois hypersurface singularities. A simple pure Hodge module $\IC^H_Z(\mathbb V)$ with strict support $Z\subset X$ is an \emph{$m$-liminal source} of $X$ if it is a simple subquotient of the mixed Hodge module $\Q^H_X[\dim X]$ and the underlying polarizable variation of $\Q$-Hodge structure $(\mathbb V, F^\bullet)$ satisfies $F^{m+1}\mathbb V_\C\neq\mathbb V_\C$. An \emph{$m$-liminal center} of $X$ is any $Z$ that appears as the support of an $m$-liminal source.

\smallskip

The notions of $m$-liminal sources and $m$-liminal centers play a crucial role in the Hodge theory of Calabi-Yau type hypersurfaces. Notably, they exhibit close analogies with sources and log canonical centers. We begin with the properties of $m$-liminal centers.

\begin{intro-theorem}\label{thm:liminal centers}
Let $X$ be a variety with $m$-Du Bois hypersurface singularities. Then:

\noindent
(1) An intersection of two $m$-liminal centers is a union of $m$-liminal centers.

\noindent
(2) There is a unique $m$-liminal source for each $m$-liminal center.

\noindent 
(3) Any union of $m$-liminal centers has Du Bois singularities, and every minimal (with respect to inclusion) $m$-liminal center has rational singularities.

In particular, the $m$-liminal locus (i.e. the complement of the locus of $m$-rational singularities) has Du Bois singularities.
\end{intro-theorem}

When $X$ has log canonical singularities, the analogous statements are proven for sources and log canonical centers: $(1)$ by Ambro \cite{Ambro03,Ambro11}, $(2)$ by Kollár \cite{Kollar16} up to a crepant birational equivalence, and $(3)$ by Kollár-Kovacs \cite{KK10} and Kawamata \cite{Kawamata98}. While the source of a log canonical center is generically a klt (log) Calabi-Yau fibration over the center, unique up to a crepant birational equivalence, the $m$-liminal source is generically a variation of Calabi-Yau type Hodge structure over the $m$-liminal center; see Section \ref{sec:liminal centers and liminal sources}. At this moment, we lack a birational geometric description of the $m$-liminal source for $m\ge 1$.

In fact, the minimal $m$-liminal center is ``the center of minimal exponent" defined in Schnell-Yang \cite{SY23}. Hence, the second part of $(3)$ follows from \textit{loc. cit.}

For Calabi-Yau type hypersurfaces, $m$-liminal sources exhibit $\Q$-Hodge structural uniqueness, analogous to the crepant birational uniqueness of sources of a projective log canonical Calabi-Yau variety. More precisely, the \emph{cores}, defined below by Laza, of the middle cohomologies of liminal sources are unique.

\smallskip
\noindent
{\bf Definition (Laza).} Let $H=(V_\Q,F^\bullet, W_\bullet)$ be a mixed $\Q$-Hodge structure with Hodge filtration $F^\bullet$ on $V_\C:=V_\Q\tensor_\Q\C$ and weight filtration $W_\bullet$ on $V_\Q$. We say that $H$ is of \emph{Calabi-Yau type} if $F^mV_\C=V_\C$ and $\dim_\C\gr_F^{m}V_\C=1$ for some $m$. The \emph{core} of $H$, denoted $\Core(H)$, is the simple subquotient $H'=(V'_\Q,F^\bullet)$ of $H$ such that $\dim_\C\gr_F^{m}V'_\C=1$.

\smallskip

Informally, the core of a Calabi–Yau type mixed Hodge structure is the simple subquotient that contains the outermost ``1," hence pure of Calabi-Yau type. For a Calabi-Yau type hypersurface $X$ with $m$-Du Bois singularities, the cores of (i) the middle cohomology for $X$, (ii) the limit mixed Hodge structures for one-parameter smoothings, and (iii) the middle cohomologies of minimal $m$-liminal sources, all coincide.

\begin{intro-theorem}\label{thm:core of liminal sources}
Let $X\subset \P^n$ be an $m$-Du Bois hypersurface of degree $d$, with $\frac{n+1}{d}=m+1\in \Z$. For any $m$-liminal source $\IC^H_Z(\mathbb V)$ supported on a minimal $m$-liminal center $Z\subset X$ (minimal by inclusion), and for any one-parameter smoothing $f:\X\to \Delta$ of $X$, the mixed Hodge structures $\HH^0(Z,\IC^H_Z(\mathbb V)), H^{n-1}(X,\Q), H^{n-1}(\X_\infty,\Q)$ are of Calabi-Yau type and their cores are isomorphic:
$$
\Core\left(\HH^0(Z,\IC^H_Z(\mathbb V))\right)\isom\Core\left(H^{n-1}(X,\Q)\right)\isom\Core\left(H^{n-1}(\X_\infty,\Q)\right).
$$
\end{intro-theorem}

Here, $H^{n-1}(\X_\infty,\Q)$ denotes the limit mixed Hodge structure of the degeneration $f:\X\to \Delta$. The fact that $H^{n-1}(X,\Q)$ is of Calabi-Yau type follows from Friedman-Laza's constancy of the Hodge-Du Bois numbers $\h^{p,q}$ for $0\le p\le m$ in families with $m$-Du Bois singularities \cite{FL24}. 

The middle cohomologies of fibers of a one-parameter smoothing of Calabi-Yau type hypersurfaces induce a variation of Calabi-Yau type Hodge structure over the punctured disk. The associated nilpotent operator $N$ of the limit mixed Hodge structure satisfies $N^{n-2m}=0$. We say that the degeneration is \emph{maximal} in the sense of Kontsevich-Soibelman \cite{KS01}, if $N^{n-2m-1}\neq 0$. Equivalently, the monodromy operator $T$ is \emph{maximally unipotent}, meaning that for sufficiently divisible $s$, we have $(T^s-1)^{n-2m}=0$ but $(T^s-1)^{n-2m-1}\neq 0$.

Using Theorem \ref{thm:core of liminal sources}, we obtain that any one-parameter smoothing of a Calabi-Yau type hypersurface with $m$-Du Bois singularities is maximally degenerate if and only if there exists a point on the special fiber with a specific type of singularities.

\begin{intro-corollary}\label{cor:maximal degeneration}
Let $X\subset \P^n$ be an $m$-Du Bois hypersurface of degree $d$, with $\frac{n+1}{d}=m+1\in \Z$. Then the following are equivalent:

\noindent
(1) $\Core\left(H^{n-1}(X,\Q)\right)=\Q^H(-m)$.

\noindent
(2) $\Q_{\left\{x\right\}}^H(-m)$ is an $m$-liminal source of $X$ for some $x\in X$.

\noindent 
(3) Any one-parameter smoothing of $X$ is a maximal degeneration.
\end{intro-corollary}

Here, $\Q^H$ denotes the trivial $\Q$-Hodge structure, and $\Q^H(-m)$ is its Tate twist. By Davis-L\H{o}rincz-Yang \cite{DLY24} with some additional input, one can prove that condition (2) is equivalent to the statement that the multiplicity of the root $-m-1$ of the reduced local Bernstein-Sato polynomial at the point $x\in X$ is $n-2m-1$.

Theorem \ref{thm:core of liminal sources} has another consequence regarding Hodge-Du Bois numbers 
$$
\h^{p,q}(X):=\dim_\C\gr_F^pH^{p+q}(X,\C).
$$
If $X$ has $m$-rational singularities, then the $m$-liminal locus $S$ of $X$ is empty and $\IC_X^H$ is the only $m$-liminal source of $X$. Hence, the core of $H^{n-1}(X,\Q)$ has weight $n-1$, or equivalently $\h^{n-1-m,m}(X)=1$. If $X$ is $m$-liminal, then $S$ is nonempty. Hence, the core has weight $<n-1$ and $\h^{n-1-m,m}(X)=0$. In fact, the numbers $\h^{n-1-m,\bullet}(X)$ are determined by $\h^{0,\bullet}(S)$:

\begin{intro-theorem}\label{thm:Hodge numbers}
Let $X\subset \P^n$ be an $m$-liminal hypersurface of degree $d$, with $\frac{n+1}{d}=m+1\in \Z$. Denote by $S$, the $m$-liminal locus of $X$  (i.e. union of every $m$-liminal center $\subsetneq X$). Then
$$
\h^{n-1-m,i}(X)-\h^{0,i-1-m}(S)=\begin{cases}
-1 & \text{if }\; i=m+1\\
\phantom{-}1 & \text{if }\; i=n-1-m\\
\phantom{-}0 & \text{otherwise.}
\end{cases}
$$
Furthermore, if the core of $H^{n-1}(X,\Q)$ has weight $\le n-3$, then $S$ is connected.
\end{intro-theorem}

Note that $S$ has Du Bois singularities by Theorem \ref{thm:liminal centers}, hence $\h^{0,i}(S)=h^i(S,\O_S)$.

In particular, if $X$ is a semi-log canonical Calabi-Yau hypersurface $(m=0)$, then $S$ is the non-klt locus and the Hodge-Du Bois numbers $\h^{0,i}(S)$ are birational invariants of $X$. Moreover, for a minimal $m$-liminal source $\IC^H_Z(\mathbb V)$, the quantity $n-1-\mathrm{weight}(\IC^H_Z(\mathbb V))$ is equal to the dimension of the dual complex of a minimal dlt model of any one-parameter smoothing; this follows directly from Theorem \ref{thm:core of liminal sources} and Nicaise-Xu \cite{NX16}. If this quantity is at least $2$, then $S$ is connected. See Example \ref{ex:K3 surface} for K3 surfaces.

In the follow-up work, we provide a detailed analysis of $1$-liminal sources and centers for GIT polystable cubic fourfolds and compute the full Hodge diamond using the techniques developed in this paper.

\smallskip
\noindent 
{\bf Example: Degenerations of cubic sevenfolds.} A cubic sevenfold $X\subset \P^8$ is a Calabi-Yau type hypersurface. Theorem \ref{thm:GIT stability via minimal exponent} says $X$ is GIT stable if it has $2$-rational singularities, and GIT semistable if it is $2$-liminal. Suppose $X$ has $2$-Du Bois singularities. Then, $X$ has $2$-rational (resp. $2$-liminal) singularities if and only if $\Core\left(H^{7}(X,\Q)\right)$ has weight $7$ (resp. $<7$). Moreover, the core of the limit Hodge structure of any one-parameter smoothing of $X$ is independent of the choice of smoothing. More explicitly, we analyze the following four cases.

\textit{Case 1.} $X=\left\{x_0^3+x_1^3+\cdots+x_8^3=0\right\}\subset \P^8$. Then $X$ is smooth, GIT stable, and the nilpotent operator $N$ of the limit mixed Hodge structure of any one-parameter smoothing satisfies $N=0$. Note that the Hodge numbers of $H^7(X,\Q)$ are:
$$
0\quad0\quad1\quad84\quad84\quad1\quad0\quad0,
$$
with $h^{7,0}(X)$ on the left and $h^{0,7}(X)$ on the right.

\textit{Case 2.} $X=\left\{x_0^3+x_1^3+\cdots+x_5^3+x_6x_7x_8=0\right\}\subset \P^8$. Then $X$ is $2$-liminal, GIT semistable, and the nilpotent operator $N$ satisfies $N\neq 0$, $N^2=0$. Additionally,
$$
\Core\left(H^7(X,\Q)\right)=\Core\left(H^4(Y,\Q)\right)(-1)
$$
where $Y=\left\{x_0^3+x_1^3+\cdots+x_5^3=0\right\}\subset \P^5$ is the Fermat cubic fourfold.

\textit{Case 3.} $X=\left\{x_0^3+x_1^3+x_2^3+x_3x_4x_5+x_6x_7x_8=0\right\}\subset \P^8$. Then $X$ is $2$-liminal, GIT semistable, and the nilpotent operator $N$ satisfies $N^2\neq0$, $N^3=0$. Additionally,
$$
\Core\left(H^7(X,\Q)\right)=\Core\left(H^1(C,\Q)\right)(-2)
$$
where $C=\left\{x_0^3+x_1^3+x_2^3=0\right\}\subset \P^2$ is the Fermat cubic curve.

\textit{Case 4.} $X=\left\{x_0x_1x_2+x_3x_4x_5+x_6x_7x_8=0\right\}\subset \P^8$. Then $X$ is $2$-liminal, GIT semistable, and the nilpotent operator $N$ satisfies $N^3\neq0$, $N^4=0$. We have
$$
\Core\left(H^7(X,\Q)\right)=\Q^H(-2),
$$
and every one-parameter smoothing of $X$ is a maximal degeneration.

\smallskip
\noindent
{\bf Acknowledgements.} I am grateful to Radu Laza and Brad Dirks for organizing the AIM workshop on higher Du Bois and higher rational singularities, and for asking questions that motivated this paper. I thank Kenny Ascher, Robert Friedman, Matt Kerr, Hyunsuk Kim, János Kollár, Yongnam Lee, Jennifer Li, Yuchen Liu, Lisa Marquand, Lauren\c tiu Maxim, Mircea Musta\c t\u a, Mihnea Popa, Sasha Viktorova, Chenyang Xu, and Ruijie Yang for valuable discussions. Part of this work was completed during my visit to KIAS as a June E Huh Visiting Fellow. This research was supported by the Oswald Veblen Fund at IAS.

\section{Preliminaries}
\label{sec:preliminaries}

\subsection{Du Bois complexes, intersection Du Bois complexes, and condition $D_m$}\label{sec:DB}
A complex variety $X$ has the associated \emph{filtered de Rham complex} $(\DB_X^\bullet, F^\bullet)$ in the bounded derived category of filtered differential complexes on $X$. This object was initially defined and studied by Du Bois \cite{DB81} and Deligne \cite{Deligne74} as a generalization of the de Rham complex for smooth varieties. The complex $\DB_X^\bullet$ is quasi-isomorphic to the constant sheaf $\C_X$, and its filtration $F^\bullet$ induces the Hodge filtration of the mixed Hodge structure on the singular cohomology $H^\bullet(X,\Q)$ of $X$ upon taking the hypercohomologies when $X$ is proper. The \emph{$p$-th Du Bois complex} $\DB^p_X$ of $X$ is the shifted graded piece
$$
\DB^p_X : = \gr^p_F\, \DB^\bullet_X[p],
$$
which is an object in $D^b_{\rm coh}(X,\O_X)$. For a detailed treatment of Du Bois complexes, see \cite[Chapter V]{GNPP} or \cite[Chapter 7.3]{PS08}.

The Du Bois complex admits an interpretation via Saito's theory of mixed Hodge modules. Locally (after a closed embedding $X\hookrightarrow Y$ with $Y$ smooth), an object of ${\rm MHM}(X)$ is given by a quadruple
$$
\M:=(M,F_\bullet, K;W_\bullet),
$$
where $M$ is a regular holonomic (right) $\D_Y$-module with a good increasing Hodge filtration $F_\bullet M$, $K$ is a $\Q$-perverse sheaf on $X$ equipped with a comparison isomorphism $\DR(M)\isom K\tensor_\Q\C$, and $W_\bullet$ is the weight filtration; these data satisfy Saito's axioms \cite{Saito88, Saito90}. The category ${\rm MHM}(X)$ is abelian, and the derived category $D^b{\rm MHM}(X)$ carries the full six-functor formalism; see details in \emph{loc. cit.}

In \cite{Saito00}, Saito proved that Du Bois complexes are naturally isomorphic to the graded de Rham complexes of the trivial object $\Q_X^H[\dim X]\in D^b{\rm MHM}(X)$, up to a shift:
$$
\DB^p_X\isom\gr^F_{-p}\DR(\Q^H_X[\dim X])[p-\dim X].
$$
When $X$ is an equidimensional variety, we replace the trivial object $\Q_X^H[\dim X]$ with the pure Hodge module $\IC^H_X$ of weight $\dim X$ associated to the intersection complex. The \emph{$p$-th intersection Du Bois complex} $I\DB^p_X$ of $X$ is the shifted graded piece
$$
I\DB^p_X := \gr^F_{-p}\DR(\IC_X^H)[p-\dim X],
$$
which is an object in $D^b_{\rm coh}(X,\O_X)$. Applying the graded de Rham functor $\gr^F_{-p} \DR(\,\cdot\,)$ to the natural morphism
$$
\gamma_X: \Q^H_X[\dim X]\to \IC^H_X
$$
in the derived category $D^b{\rm MHM}(X)$ of mixed Hodge modules on $X$, we obtain a natural morphism from the Du Bois complex to the intersection Du Bois complex:
$$
\gamma_p: \DB^p_X \to I\DB^p_X.
$$
In \cite{PP25a}, Popa and the author extensively studied this morphism and introduced a notation $D_m$ when $\gamma_p$ are isomorphic for all $p\le m$:

\begin{defn}
Condition $D_m$ is said to hold for an equidimensional variety $X$ if the morphisms $\gamma_p: \DB^p_X \to I\DB^p_X$ are isomorphisms for all $0 \le p \le m$. For brevity, we will sometimes write $\DB^p_X = I\DB^p_X$ to indicate this isomorphism. 
\end{defn}

See \cite{PP25a} for further discussions on Hodge-Du Bois numbers and intersection Hodge numbers,
$$
\h^{p,q}(X):=\dim_\C\gr_F^pH^{p+q}(X,\C),\quad I\h^{p,q}(X):=\dim_\C\gr_F^p\mathit{IH}^{p+q}(X,\C),
$$
when $X$ is projective and satisfies condition $D_m$. We note that condition $D_m$ is also studied in \cite{DOR25}, under different notation, where it is referred to as $X$ being a \emph{rational homology manifold up to Hodge degree $m$}.

We record a duality formula for the graded de Rham functor of mixed Hodge modules, which induces a duality on intersection Du Bois complexes (see e.g. \cite[Lemma 3.2]{Park23}).

\begin{prop}
\label{prop:duality}
Let $X$ be a quasi-projective variety and $\M^\bullet\in D^b{\rm MHM}(X)$. Then, for every integer $p$, we have an isomorphism
$$
R\Hom_{\O_X}\left(\gr^F_p\DR(\M^\bullet),\omega_X^\bullet\right)\isom\gr^F_{-p}\DR(\mathbb \dual(\M^\bullet))
$$
in $D^b_{\rm coh}(X,\O_X)$, where $\omega^\bullet_X$ is the dualizing complex of $X$.
\end{prop}

Here, $\dual:{\rm MHM}(X)\to {\rm MHM}(X)$ is the Saito-Verdier dualizing functor for mixed Hodge modules. When $X$ is equidimensional, the polarization $\IC_X^H(\dim X)\isom \dual \IC_X^H$ of the intersection complex yields the duality
$$
I\DB_X^p\isom R\Hom_{\O_X}\left(I\DB_X^{\dim X-p}, \omega_X^\bullet[-\dim X]\right).
$$
For any integer $k$ and $\M=(M,F_\bullet, K;W_\bullet)\in {\rm MHM}(X)$, the Tate twist is
$$
\M(k):=(M,F_{\bullet-k},K(k); W_{\bullet+2k}),
$$
where $K(k)=K\tensor_\Q\Q(k)$ and $\Q(k)=(2\pi i)^k\Q\subset \C$.

\subsection{Higher singularities and minimal exponents}\label{sec:higher singularities}

The Du Bois complexes have been used to define \emph{higher Du Bois} and \emph{higher rational} singularities, which refine the classical notions of Du Bois and rational singularities. These refinements have been developed for varieties with local complete intersection (lci) singularities; see, for example, \cite{JKSY, MOPW, MP22, FL24}.

\begin{defn}\label{defn:higher singularities lci}
Let $m\ge 0$ be an integer and let $X$ be a variety with lci singularities.
\begin{enumerate}
\item $X$ has \emph{$m$-Du Bois singularities} if the natural morphisms $\Omega^p_X \to \DB^p_X$ are isomorphisms for all $0\le p\le m$.
\item $X$ has \emph{$m$-rational singularities} if the morphisms $\Omega^p_X \to \DD_X(\DB^{n-p}_X)$  are isomorphisms for all $0\le p\le m$, where $\DD_X (\,\cdot\,) : = R\Hom (\, \cdot \, , \omega_X)$ denotes the (shifted) Grothendieck dual.
\item We say $X$ is \emph{$m$-liminal} if $X$ is $m$-Du Bois but not $m$-rational.
\end{enumerate}
\end{defn}

When $X$ has hypersurface singularities, one can define another numerical invariant -- the \emph{minimal exponent} -- for $X$. Introduced by \cite{Saito93} through the Bernstein-Sato polynomial of a local defining equation $f$, this invariant provides a precise criterion for higher Du Bois and higher rational singularities. We briefly explain this.

For a non-invertible regular function $f$ on a germ of a smooth complex variety $Y$ at $y\in Y$, there exists a nonzero polynomial $b(s)\in \C[s]$ and a differential operator $P(s)\in \D_Y[s]$ such that
$$
P(s)f^{s+1}=b(s)f^s
$$
formally in $\O_Y[\frac{1}{f},s]\cdot f^s$, where $\D_Y$ is the ring of differential operators on $Y$. The set of all polynomials $b(s)$ satisfying this equation is an ideal of the polynomial ring $\C[s]$; its monic generator is the \emph{Bernstein-Sato polynomial} of $f$, denoted $b_f(s)$. It is easy to see from the construction that $b_f(-1)=0$. The \emph{minimal exponent} $\wt \alpha_{y}(f)$ is defined to be the negative of the greatest root of the reduced Bernstein-Sato polynomial $\wt b_f(s):=b_f(s)/(s+1)$. Note that $\left\{f=0\right\}$ is smooth if and only if $\wt b_f(s)=1$, in which case $\wt \alpha_y(f)=+\infty$ by convention. By Kashiwara's rationality theorem \cite{Kashiwara76}, $\wt \alpha_{y}(f)$ is a positive rational number.

For a variety $X$ with hypersurface singularities, the local minimal exponent at $x\in X$ is defined by
$$
\wt\alpha_x(X):=\wt\alpha_x(f_x),
$$
where $f_x$ is a local defining equation for $X$ in a smooth ambient variety near the point $x$. This invariant is a well-defined positive rational number: $\wt\alpha_{x}(f_x)$ is a positive rational number, independent of the choice of the embedding (see e.g. \cite[Proposition 4.14]{CDMO24}).

The global minimal exponent of $X$ is defined by
$$
\wt \alpha(X):=\min_{x\in X} \wt\alpha_x(X),
$$
which is again a positive rational number; it is well known that the minimum is attained at some point $x\in X$. When $X$ is a divisor in a smooth variety $Y$, the log canonical threshold satisfies
$$
\lct(Y,X)=\min\left\{\wt\alpha(X),1\right\}.
$$
This implies that $X$ has Du Bois singularities if and only if $\wt \alpha(X)\ge 1$. Furthermore, Saito \cite{Saito93} proved that $X$ has rational singularities if and only if $\wt \alpha(X)> 1$. These criteria were generalized for higher Du Bois and higher rational singularities:

\begin{thm}
\label{thm:higher singularities vs minimal exponent}
Let $X$ be a variety with hypersurface singularities. Then

\noindent 
(1) $X$ has $m$-Du Bois singularities if and only if $\wt \alpha(X)\ge m+1$.

\noindent 
(2) $X$ has $m$-rational singularities if and only if $\wt \alpha(X)> m+1$.
\end{thm}

Note that (1) was proved in \cite[Theorem 1]{JKSY} and \cite[Theorem 1.1]{MOPW} and $(2)$ was proved in \cite[Appendix]{FL24} and \cite[Theorem E]{MP25} (see \cite{CDM24, CDMO24} for the generalization of this result to lci singularities). For our later use, we record here the Thom-Sebastiani theorem for minimal exponents:

\begin{thm}[{\cite[Theorem 0.8]{Saito94}}]
\label{thm:Thom-Sebastiani for minimal exponents}
Let $Y_1$ and $Y_2$ be smooth varieties and $f_1\in \O_{Y_1}(Y_1)$ and $f_2\in \O_{Y_2}(Y_2)$ be nonzero regular functions. For points $y_1\in Y_1$ and $y_2\in Y_2$, assume $f_1(y_1)=0$ and $f_2(y_2)=0$. Then
$$
\wt\alpha_{(y_1,y_2)}(f_1\oplus f_2)=\wt\alpha_{y_1}(f_1)+\wt\alpha_{y_2}(f_2)
$$
where $f_1\oplus f_2\in \O_{Y_1\times Y_2}(Y_1\times Y_2)$ and $(y_1,y_2)\in Y_1\times Y_2$.
\end{thm}

See also \cite[Example 6.7]{MP20} for an alternative explanation.

More recently, the notions of higher Du Bois and higher rational singularities were generalized for arbitrary varieties, not necessarily with local complete intersection singularities. We record the following definition from \cite[Definitions 1.2 and 1.3]{SVV23}. When $X$ has lci singularities, these notions agree with Definition \ref{defn:higher singularities lci}; see \cite[Propositions 5.5 and 5.6]{SVV23} for more details.

\begin{defn}
\label{defn:higher singularities general}
Let $X$ be a variety. We say that $X$ has \textit{$m$-Du Bois singularities} if it is seminormal, and 
\begin{enumerate}
	\item $\codim_X\; \Sing(X) \ge 2m+1$;
	\item $\mathcal H^{>0}(\DB_X^p)=0$ for all $0\le p \le m$;
	\item $\mathcal H^0(\DB_X^p)$ is reflexive, for all $0\le p \le m$.
\end{enumerate}
We say that $X$ has \textit{$m$-rational singularities} if it is normal, and
\begin{enumerate}
	\item $\codim_X\;\Sing(X) > 2m+1$;
	\item $\mathcal H^{>0}(\DD_X(\DB_X^{n-p}))=0$ for all $0\le p \le m$.
\end{enumerate}
\end{defn}

Condition (2) of $m$-Du Bois (resp. $m$-rational) singularities is referred to as pre-$m$-Du Bois (resp. pre-$m$-rational) singularities. For a normal variety, pre-$m$-rational singularities are equivalent to pre-$m$-Du Bois singularities with $D_m$, which follows from \cite[Theorem B]{SVV23}, \cite[Proposition 9.4]{PSV25}, and \cite[Remark 5.2]{DOR25}:

\begin{prop}
\label{prop:pre-m-rational is pre-m-Du Bois with D_m}
Let $X$ be a normal variety. Then the following are equivalent:

\noindent
(1) $\mathcal H^{>0}(\DD_X(\DB_X^{n-p}))=0$ for all $0\le p \le m$.

\noindent
(2) $\mathcal H^{>0}(\DB_X^p)=0$ for all $0\le p \le m$, and $X$ satisfies condition $D_m$.
\end{prop}

\subsection{The RHM defect objects and liminal sources}\label{sec:RHM defect}
\label{subsec:RHM defect objects and liminal sources}

For a variety $X$ with hypersurface singularities, the difference between $m$-Du Bois and $m$-rational singularities is encoded in the RHM-defect object, defined and studied in \cite{PP25a}. This object is important both for understanding the Hodge structure of $m$-liminal varieties and for analyzing limit mixed Hodge structures of one-parameter degenerations.

\begin{defn}[{\cite[Definition 6.1]{PP25a}}]
\label{defn:RHM defect object}
The \emph{RHM-defect object} of an equidimensional variety $X$ of dimension $n$ is the object $\K_X^\bullet\in D^b{\rm MHM}(X)$ sitting in the distinguished triangle:
\begin{equation}
\label{eqn:Q to IC triangle}
\K_X^\bullet \longrightarrow \Q_X^H[\dim X]\xrightarrow{\gamma_X}\IC_X^H\xrightarrow{+1}.    
\end{equation}
\end{defn}

By definition, condition $D_m$ is equivalent to the vanishing $\gr^F_{-p}\DR(\K^\bullet_X)=0$ for $0\le p\le m$.

When $X$ has hypersurface singularities, the sheaf $\Q_X[\dim X]$ is perverse and the RHM-defect object $\K^\bullet_X$ is a single mixed Hodge module. In this situation, it is proven in \cite[Theorem 3.1]{CDM24} that $X$ is $m$-rational if and only if it is $m$-Du Bois and satisfies condition $D_m$. Additionally, $m$-Du Bois implies $(m-1)$-rational by Theorem \ref{thm:higher singularities vs minimal exponent}, hence condition $D_{m-1}$. In summary,
\begin{center}
$X$ is $m$-rational $\Longleftrightarrow$ $X$ is $m$-Du Bois and $\gr^F_{-m}\DR(\K^\bullet_X)= 0$.
\end{center}
We use this characterization to present an equivalent definition of $m$-liminal source (equivalent to the one in the introduction).

\begin{defn}
\label{defn:liminal sources and centers}
Let $X$ be a variety with $m$-Du Bois hypersurface singularities. A pure Hodge module $\M$ is an \emph{$m$-liminal source} of $X$ if either $\M=\IC_X^H$ or $\M$ is a simple subquotient of $\K_X^\bullet$ such that
\begin{equation}
\label{eqn:nonvanishing grDR-m}
\gr^F_{-m}\DR(\M)\neq 0.    
\end{equation}
An \emph{$m$-liminal center} is the strict support $\Supp(\M)\subset X$ of an $m$-liminal source $\M$.
\end{defn}

Note that when $X$ has $m$-rational singularities, we have $\gr^F_{-m}\DR(\K^\bullet_X)=0$ and $\IC^H_X$ is the only $m$-liminal source, with $X$ as the only $m$-liminal center. Since $\K_X^\bullet$ is of weight $\le \dim X-1$ by \cite[Proposition 6.4]{PP25a}, an $m$-liminal source $\M$ is of weight $\le \dim X-1$ if $\M\neq \IC^H_X$.

By the structure theorem \cite[Theorem 3.21]{Saito90} of pure Hodge modules, a simple Hodge module $\M$ is the minimal extension of an irreducible polarizable variation of Hodge structure $(\mathbb V, F^\bullet)$ on a smooth open subvariety of an irreducible subvariety $Z\subset X$. Following the convention for the minimal extension, we denote by
$$
\M=\IC^H_Z(\mathbb V).
$$
If $\M$ is an $m$-liminal source of an $m$-Du Bois variety $X$ not $\IC_X^H$, then \eqref{eqn:nonvanishing grDR-m} is equivalent to $F^m\mathbb V_\C\neq F^{m+1}\mathbb V_\C$ (equivalently, $\gr_F^m\mathbb V_\C\neq 0$). Therefore, this definition of $m$-liminal source agrees with the one in the introduction; details follow in the next paragraph.

Indeed, over the locus where $(\mathbb V,F^\bullet)$ is a variation of Hodge structure, from the definition of the graded de Rham functor, we have
$$
\gr^F_{-m}\DR(\M)\isom\gr_F^{m}\mathcal V[\dim Z]
$$
where $\mathcal V$ is a vector bundle with flat connection associated to $\mathbb V$. Hence,
\begin{equation}
\label{eqn:Hodge piece nonvanishing for VHS}
F^m\mathbb V_\C\neq F^{m+1}\mathbb V_\C\Longleftrightarrow \gr^F_{-m}\DR(\M)\neq 0     
\end{equation}
on an open set of $Z$. By dualizing, \eqref{eqn:nonvanishing grDR-m} is equivalent to 
$$
F_m\dual\M:=\gr^F_{m}\DR(\dual\M)\neq 0,
$$
that is, the index of the first nonzero Hodge filtration of $\dual \M$ is $m$ (as a right D-module). The first nonzero Hodge filtration is a torsion-free $\O_Z$-module from Saito's theory, so \eqref{eqn:nonvanishing grDR-m} on the open set implies the same everywhere.

\smallskip
\noindent
\textbf{Log rational pairs.} As names suggest, $m$-liminal sources and $m$-liminal centers satisfy analogous properties of sources and log canonical centers. For instance, any union of $m$-liminal centers has Du Bois singularities. We recall a key notion used in \cite{Park23} to give an alternative proof of the theorem of Kollár and Kovács \cite{KK10}, that a union of log canonical centers is Du Bois.

\begin{defn}
\label{defn:log rational pair}
Let $X$ be a variety and $Z\subset X$ a reduced closed subscheme. We call $(X,Z)$ a \emph{log rational pair} if
\begin{enumerate}
\item the natural morphism $\I_{X,Z}\to \DB^0_{X,Z}$ is a quasi-isomorphism, where $\I_{X,Z}$ is the ideal sheaf of $Z$ in $X$; and
\item the open complement $X\sm Z$ has rational singularities.
\end{enumerate}
\end{defn}

Note that the Du Bois complex $\DB^0_{X,Z}$ of a pair $(X,Z)$ is an object in $ {\rm coh}(X, \O_X)$, sitting in a distinguished triangle:
$$
\DB^0_{X,Z}\rightarrow\DB^0_X\xrightarrow{\rho}\DB^0_Z\xrightarrow{+1},
$$
hence admits a natural morphism $\I_{X,Z}\to \DB^0_{X,Z}$. If condition (2) is omitted, $(X,Z)$ is a \emph{Du Bois pair} in the sense of \cite[Definition 3.13]{Kovacs11}.

Combining Kovács' criteria \cite[Theorem 1]{Kovacs00} for rational singularities and \cite[Theorem 5.4]{Kovacs11} for Du Bois pairs, we obtain a criterion for a log rational pair (see \cite[Corollary 1.10]{Park23}): 

\begin{prop}
\label{prop:criterion for log rational pair}
Let $(X',Z')$ be a log rational pair and $\mu:X'\to X$ a proper morphism with $\mu(Z')\subset Z$. Then, $(X,Z)$ is a log rational pair if there exists a left quasi-inverse of the natural morphism $\I_{X,Z}\to R\mu_*\I_{X',Z'}$.
\end{prop}

Here, a left quasi-inverse refers to a morphism $R\mu_*\I_{X',Z'}\to \I_{X,Z}$ such that the composition is a quasi-isomorphism of $\I_{X,Z}$ to itself.

\subsection{The Hilbert-Mumford criterion and minimal exponent of affine cone}

The Hilbert-Mumford criterion provides a standard method to check the GIT (semi)stability of a hypersurface $X\subset \P^n$. We briefly review this criterion in a form that is compatible with a particular bound for the minimal exponent.

\begin{defn}
Let $w=(w_0,...,w_n)\in \Q^{n+1}$ be the rational weight system. For a nonzero polynomial $f\in \C[x_0,\dots,x_n]$, the weight $\mathrm{wt}_w(f)$ of a polynomial $f$ is the minimum of $\sum_{i=0}^nw_ie_i$ for all monomials $x_0^{e_0}\cdots x_n^{e_n}$ appearing in $f$ with nonzero coefficients.
\end{defn}

By the diagonalizability of a one-parameter subgroup of the special linear group $SL(n+1)$, the Hilbert-Mumford criterion for the GIT stability of hypersurfaces can be stated as follows:

\begin{prop}[Hilbert-Mumford numerical criterion \cite{Mumford94}]
\label{prop:Hilbert-Mumford criterion}
Let $X\subset \P^n$ be a hypersurface defined by a degree $d$ homogeneous polynomial $f(x_0,\dots,x_n)=0$. Then $X$ is GIT stable (resp. semistable) if and only if for every nontrivial (rational) weight system $w=(w_0,\dots,w_n)$ and $g\in SL(n+1)$, we have $\mathrm{wt}_w(f\circ g)< \frac{d}{n+1}\sum_{i=0}^n w_i$ (resp. $\mathrm{wt}_w(f\circ g)\le \frac{d}{n+1}\sum_{i=0}^n w_i$).
\end{prop}

Here, a nontrivial weight system refers to $w=(w_0,\dots,w_n)$ such that $w_i$ are not all equal. Note that for a homogeneous polynomial $f$ and a weight system $w'=w+\alpha:=(w_0+\alpha,\dots,w_n+\alpha)$, we have
$$
\mathrm{wt}_{w'}(f)=\mathrm{wt}_w(f)+d\alpha\quad\mathrm{and}\quad \frac{d}{n+1}\sum_{i=0}^n (w_i+\alpha)=\frac{d}{n+1}\sum_{i=0}^n w_i+d\alpha.
$$
This implies that the Hilbert-Mumford numerical criterion is sufficient to check for only nonnegative weight systems.

On the other hand, a weight system gives an upper bound for the minimal exponent, which refines \cite[Proposition 8.13]{Kollar97} for log canonical thresholds.

\begin{prop}[{\cite[Proposition 2.1]{CDM25}}]
\label{prop:weight-minimal exponent inequality}
Let $f\in \C[x_0,\dots,x_n]$ be a nonzero polynomial with $f(0)=0$, and let $w=(w_0,\dots,w_n)$ be a nonnegative nonzero weight system. If the hypersurface $\{f=0\}\subset \C^{n+1}$ is singular at $0$, then the minimal exponent of $f$ at $0$ satisfies
$$
\wt \alpha_0(f)\le\frac{w_0+\cdots+w_n}{\mathrm{wt}_{w}(f)}.
$$
\end{prop}

Unlike minimal exponents, the analogous bound for the log canonical threshold does not require the hypersurface to be singular at the origin. This is essentially why the methods of \cite{Hacking04,KL04,Lee08}, which prove GIT stability of hypersurfaces with bounds on log canonical thresholds, cannot be directly adapted to the case of minimal exponents.

Note that the minimal exponent of a homogeneous degree $d\ge 2$ polynomial $f$ at the origin satisfies
\begin{equation}
\label{eqn:minimal exponent multiplicity bound}
\wt \alpha_0(f)\le \frac{n+1}{\mathrm{mult}_0(f)}=\frac{n+1}{d}.
\end{equation} 
This inequality follows either from \cite[Theorem E(3)]{MP20} or from the above weighted bound applied with $w=(1,\dots,1)$. By combining Propositions \ref{prop:Hilbert-Mumford criterion} and \ref{prop:weight-minimal exponent inequality}, we conclude that a hypersurface $X\subset \P^n$ is GIT semistable when its defining equation $f=0$ satisfies
$$
\wt \alpha_0(f)= \frac{n+1}{d}.
$$
Recall that $\wt \alpha_0(f)=\wt \alpha_0(f\circ g)$ for every $g\in SL(n+1)$, since the minimal exponent is an invariant of the hypersurface singularity. Therefore, the semistability part of Theorem \ref{thm:GIT stability via minimal exponent} is settled, if the following implication is true: $\wt\alpha(X)\ge \frac{n+1}{d}$ $\Rightarrow$ $\wt \alpha_0(f)= \frac{n+1}{d}$. This is verified in Theorem \ref{thm:minimal exponent of cone}.

\section{The GIT stability via minimal exponent}

\subsection{Higher singularities of affine cones}

Let $X$ be a projective scheme with an ample line bundle $\L$. Following \cite[Section 3.8]{Kollar13}, the \emph{affine cone} over $X$ with conormal bundle $\L$ is
$$
C(X,\L):=\text{Spec}\bigoplus_{k\ge 0}H^0(X,\L^k).
$$
In particular, when $X\subset \P^n$ is a hypersurface defined by a homogeneous polynomial $f=0$, the affine cone over $X$ with conormal bundle $\O_X(1)$ is the classical affine cone $\mathrm{Cone}(X)\subset \A^{n+1}$ over $X$:
$$
C(X,\O_X(1))\isom\mathrm{Cone}(X):=\{f=0\}\subset \A^{n+1}.
$$

From the viewpoint of singularities of the minimal model program, it is known that semi-log canonical or klt singularities are preserved under taking an affine cone in certain settings, such as cones over Calabi-Yau or Fano varieties. We prove a refinement of this fact for hypersurfaces: if the degree and the dimension satisfy a specific numerical inequality, then the affine cone with any conormal bundle $\O_X(r)$ (for all $r\ge 1$) preserves higher Du Bois or higher rational singularities.

\begin{thm}
\label{thm:higher singularities of cone}
Let $X\subset \P^n$ be a hypersurface of degree $d\ge 2$, and let $r$ be any positive integer. For an integer $m\ge 0$,

\noindent
(1) if $\frac{n+1}{d}\ge m+1$, then $X$ has $m$-Du Bois singularities if and only if the affine cone $C(X,\O_X(r))$ has $m$-Du Bois singularities.

\noindent
(2) if $\frac{n+1}{d}> m+1$, then $X$ has $m$-rational singularities if and only if the affine cone $C(X,\O_X(r))$ has $m$-rational singularities.
\end{thm}

When $m=0$, this statement follows from the classical result about semi-log canonical and klt singularities of affine cones (see, for example, \cite[Lemma 3.1]{Kollar13}).

Before proving Theorem \ref{thm:higher singularities of cone}, we provide two basic lemmas which are well known to experts, and include short proofs for completeness. The first one provides a natural resolution of the sheaf of Kähler differentials, which is used repetitively throughout the text:

\begin{lem}
\label{lem:Koszul resolution of differentials}
Let $X\subset Y$ be a Cartier divisor in a smooth variety $Y$, and
$$
\O_Y(-X)|_X\xrightarrow{\phi} \Omega^1_Y|_X\to \Omega_X\to 0.
$$
be the associated conormal exact sequence. For any integer $p\ge 1$, if $\codim_X\Sing(X)\ge p$, then the Koszul complex
$$
K_p^\bullet(\phi):\O_Y(-pX)|_X\to\Omega_Y(-(p-1)X)|_X \to\cdots\to\Omega_Y^p|_X,
$$
associated to the morphism $\phi:\O_Y(-X)|_X\to \Omega^1_Y|_X$ is naturally quasi-isomorphic to $\Omega_X^p$.
\end{lem}

We consider $K_p^\bullet(\phi)$ as a complex supported on degrees $[-p,0]$, so that there exists a natural map
\begin{equation}
\label{eqn:resolution of differentials}
K_p^\bullet(\phi)\to \Omega_X^p[0].    
\end{equation}

\begin{proof}
It suffices to show that \eqref{eqn:resolution of differentials} is locally a quasi-isomorphism. Let $f=0$ be the local defining equation of $X$ in $Y$, and $y_0,\dots,y_n$ be the system of local coordinates of $Y$. Upon trivialization, $\phi:\O_Y(-X)|_X\to \Omega^1_Y|_X$ is locally represented by the matrix
$$
\left[\frac{\partial f}{\partial y_0},\dots, \frac{\partial f}{\partial y_n}\right].
$$
Note that the ideal generated by these partial derivatives defines $\Sing(X)$. Thus, the depth of this ideal is $\codim_X\Sing(X)$. The depth sensitivity of the Koszul complex (see \cite[Theorem 16.8]{Matsumura89}) implies the cohomology vanishing $\mathcal H^{<0}(K_p^\bullet(\phi))=0$, or equivalently
$$
K_p^\bullet(\phi)\isom \mathcal H^0(K_p^\bullet(\phi))[0]=\Omega^p_X[0],
$$
for $p\le \codim_X\Sing(X)$.
\end{proof}

Note that the section $\phi(X):\O_Y|_X\to \Omega^1_Y(X)|_X$ is the logarithmic differential $d\log f$ restricted to $X$, where $f:\O_Y\to \O_Y(X)$ is the natural section. Each differential in the Koszul complex $K_p^\bullet(\phi)$ is the wedge product map $\cdot \wedge d\log f$, restricted to $X$. Now, consider the natural composition of maps
\begin{equation}
\label{eqn:p-atiyah class}
\Omega_Y^p\to \Omega_Y^p|_X\xrightarrow{\wedge\; d \log f}\Omega_Y^{p+1}(X)|_X\to\Omega_Y^{p+1}[1],
\end{equation}
where the first map is the restriction map and the last map is the connecting morphism arising from the short exact sequence 
$$
0\to\Omega_Y^{p+1}\to\Omega_Y^{p+1}(X)\to\Omega^{p+1}_Y(X)|_X\to0.
$$
This composition has the following cohomological interpretation:

\begin{lem}
\label{lem:atiyah class}
Let $Y$ be a smooth projective variety, and $X\subset Y$ be a Cartier divisor with the associated line bundle $\L=\O_Y(X)$. For every $p\ge 0$, the cup product map with the first Chern class
$$
\cdot\cup c_1(\L): H^q(Y,\Omega_Y^p)\to H^{q+1}(Y,\Omega_Y^{p+1})
$$
coincides with $\frac{1}{2\pi i}$ times the map induced on cohomology by the composition \eqref{eqn:p-atiyah class}.
\end{lem}

\begin{proof}
For $p=0$, the composition
\begin{equation}
\label{eqn:atiyah class}
\O_Y\to \O_Y|_X\xrightarrow{d \log f}\Omega_Y^1(X)|_X\to\Omega_Y^1[1]  
\end{equation}
is the Atiyah class $a(\L)\in H^1(Y,\Omega_Y)$. Indeed, for a trivializing chart $\{U_i\}$ of $\O_Y(X)$ with $f_i=f|_{U_i}$, the \v Cech cocycle $\{d \log f_i-d \log f_j\}=\{d\log g_{ij}\}$ represents the extension class associated to the composition \eqref{eqn:atiyah class}, where $g_{ij}=\frac{f_i}{f_j}$ is the transition function of the line bundle $\L=\O_Y(X)$. This \v Cech cocyle represents the Atiyah class $a(\L)$, which is equal to $2\pi i c_1(\L)$ by \cite[Proposition 12]{Atiyah57}.

The composition map \eqref{eqn:p-atiyah class} for general $p\ge 0$ is the wedge product of $\Omega_Y^p$ with \eqref{eqn:atiyah class}. Hence, the cohomology map of this composition is the cup product map with the Atiyah class $a(\L)$.
\end{proof}

We now prove Theorem \ref{thm:higher singularities of cone}. For notational convenience, set $C(X,r):=C(X,\O_X(r))$, and denote its cone point by $0\in C(X,r)$.

\begin{proof}[Proof of Theorem \ref{thm:higher singularities of cone}]
Since $C(X,r)\sm\{0\}$ is a $\mathbb G_m$-torsor over $X$, $C(X,r)\sm\{0\}$ has hypersurface singularities. Then by Theorem \ref{thm:higher singularities vs minimal exponent}, if $C(X,r)\sm\{0\}$ has $m$-Du Bois (resp. $m$-rational) singularities, then $X$ has $m$-Du Bois (resp. $m$-rational) singularities. Hence, it suffices to prove the forward implications of this theorem when $X$ has $m$-Du Bois (resp. $m$-rational) singularities. The case $m=0$ follows from classical results on semi-log canonical and klt singularities (see \cite[Lemma 3.1]{Kollar13}); note that semi-log canonical implies Du Bois \cite{KK10} and klt implies rational \cite{Elkik81}. From now on, we assume $m\ge 1$, and thus $C(X,r)$ has rational singularities.

\textbf{Proof of (1).}
Suppose $X$ has $m$-Du Bois singularities and $\frac{n+1}{d}\ge m+1$. We check the conditions in Definition \ref{defn:higher singularities general}. Since $\codim\;\Sing(X)\ge 2m+1$, we have
$$
\codim\;\Sing(C(X,r))\ge 2m+1.
$$
Additionally since $C(X,r)$ has rational singularities, \cite[Corollary 1.11]{KS21} and \cite[Theorem 7.12]{HJ14} implies that $\mathcal H^0(\DB_{C(X,r)}^p)$ is reflexive for all $p$.

Hence, $C(X,r)$ has $m$-Du Bois singularities if and only if the higher cohomologies of Du Bois complexes vanish,
$$
\mathcal H^{>0}(\DB_{C(X,r)}^p)=0 \quad\text{for all}\quad 0\le p\le m.
$$
In other words, $C(X,r)$ has pre-$m$-Du Bois singularities as defined in \cite{SVV23}. By \cite[Corollary 7.1]{PS25}, this condition is equivalent to
$$
H^i(X,\Omega_X^p(kr))=0 \quad \text{for all}\quad i,k\ge 1,\; 0\le p\le m.
$$
It suffices to prove for $r=1$. By Lemma \ref{lem:Koszul resolution of differentials}, the $p$-th sheaf of differentials twisted by $\O_{\P^n}(k)$, $\Omega_X^p(k)$, is quasi-isomorphic to
$$
K_p^\bullet(\phi)(k):\O_{\P^n}(-pd+k)|_X\to\Omega_{\P^n}(-(p-1)d+k)|_X \to\cdots\to\Omega_{\P^n}^p(k)|_X.
$$
Consequently, we have the spectral sequence induced by the stupid filtration
$$
E_1^{i,j}\Longrightarrow H^{i+j}(X,\Omega_X^p(k)),
$$
where $E_1^{i,j}=H^j(X,\Omega_{\P^n}^{p+i}(id+k)|_X)$ for $i\le 0$ and $E_1^{i,j}=0$ otherwise. It suffices to prove $E_2^{i,j}=0$ for all $j>0$.

With a fixed $j>0$, we have the following complex $E_1^{\bullet,j}$:
\begin{equation}
\label{eqn:E_1 row}
H^j(X,\O_{\P^n}(-pd+k)|_X)\xrightarrow{d_1} H^j(X,\Omega_{\P^n}(-(p-1)d+k)|_X)\xrightarrow{d_1} \cdots\xrightarrow{d_1} H^j(X,\Omega_{\P^n}^p(k)|_X).
\end{equation}
For $-p\le i\le 0$, consider the short exact sequence
$$
0\to\Omega_{\P^n}^{p+i}((i-1)d+k)\to\Omega_{\P^n}^{p+i}(id+k)\to\Omega_{\P^n}^{p+i}(id+k)|_X\to0.
$$
From the Bott vanishing theorem for projective spaces (see e.g. \cite[Theorem 7.2.3]{CMSP17}), the cohomology $H^j(\P^n,\Omega^{e}_{\P^n}(l))$ vanishes for all $j>0$, except in the two cases:
$$
(i)\;j=e \text{ and } l=0,\quad (ii)\;j=n \text{ and } l<-n+e.
$$
Applying this with $e=p+i$ and $l=(i-1)d+k$, we have
$$
(i-1)d+k\ge -n+m+i\ge -n+p+i,
$$
for $k\ge 1$ under the hypothesis $\frac{n+1}{d}\ge m+1$; the case $(ii)$ does not occur. From the long exact sequence of cohomology, this implies that
$$
H^j(X,\Omega_{\P^n}^{p+i}(id+k)|_X)= \left\{ \begin{array}{lcl}
H^j(\P^n,\Omega_{\P^n}^j) & \text{for } i=j-p \text{ and } k=(p-j)d,\\
H^{j+1}(\P^n,\Omega_{\P^n}^{j+1}) & \text{for } i=j+1-p \text{ and } k=(p-j)d,\\
0 &  \text{otherwise}.
\end{array}\right.
$$
Therefore, if \eqref{eqn:E_1 row} is not zero, then $k=(p-j)d$, in which case \eqref{eqn:E_1 row} is
$$
0\to \cdots\to 0\to H^j(X,\Omega_{\P^n}^{j}|_X)\xrightarrow{d_1} H^j(X,\Omega_{\P^n}^{j+1}(d)|_X)\to0\to\cdots\to 0.
$$
The map $d_1:H^j(X,\Omega_{\P^n}^{j}|_X)\to H^j(X,\Omega_{\P^n}^{j+1}(d)|_X)$ is naturally isomorphic to $2\pi i$ times the cup product map
$$
\cdot\cup c_1(\O_{\P^n}(X)): H^j(\P^n,\Omega_{\P^n}^j)\to H^{j+1}(\P^n,\Omega_{\P^n}^{j+1})
$$
by Lemma \ref{lem:atiyah class}, and this map is an isomorphism. Hence, $E_2^{i,j}=0$ for all $j>0$ as desired.

\textbf{Proof of (2).} Suppose $X$ has $m$-rational singularities and $\frac{n+1}{d}> m+1$. By (1), we already know that $C(X,r)$ has $m$-Du Bois singularities, and
$$
\codim\;\Sing(C(X,r))> 2m+1.
$$
Therefore by Proposition \ref{prop:pre-m-rational is pre-m-Du Bois with D_m}, it suffices to prove condition $D_m$ for $C(X,r)$, that is
$$
\DB_{C(X,r)}^p\isom I\DB_{C(X,r)}^p\quad \text{for all}\quad 0\le p\le m.
$$
In terms of the RHM defect object $\K_{C(X,r)}^\bullet$, this condition is equivalent to
$$
\gr^F_{-p}\DR(\K_{C(X,r)}^\bullet)=0 \quad \text{for all}\quad p\le m.
$$
Since $C(X,r)\sm\{0\}$ is a $\mathbb G_m$-torsor over $X$, we know that $C(X,r)\sm\{0\}$ has $m$-rational singularities. Hence, the above vanishing holds over $C(X,r)\sm\{0\}$.

Denote by $j:C(X,r)\sm\{0\}\hookrightarrow C(X,r)$ the open embedding and $\iota:\{0\}\hookrightarrow C(X,r)$ the closed embedding. From the standard fact about the graded de Rham functor under pushforward (see e.g. \cite[Lemma 3.4]{Park23} and its dual statement), we have
$$
\gr^F_{-p}\DR(j_!\K_{C(X,r)\sm\{0\}}^\bullet)=0 \quad \text{for all}\quad p\le m.
$$
From the distinguished triangle
$$
j_!\K_{C(X,r)\sm\{0\}}^\bullet\to \K_{C(X,r)}^\bullet\to \iota_*\iota^*\K_{C(X,r)}^\bullet\xrightarrow{+1},
$$
it suffices to prove that
$$
\gr^F_{-p}\DR(\iota^*\K_{C(X,r)}^\bullet)=0 \quad \text{for all}\quad p\le m.
$$
Applying the pullback $\iota^*$ to the distinguished triangle \eqref{eqn:Q to IC triangle} for $C(X,r)$, we have
$$
\iota^*\K_{C(X,r)}^\bullet\to \Q_{\{0\}}^H[n]\to \iota^*\IC_{C(X,r)}^H\xrightarrow{+1}.
$$
Recall that the case $m=0$ is proven in \cite[Lemma 3.1]{Kollar13}. It remains to prove
$$
\gr^F_{-p}\DR(\iota^*\IC_{C(X,r)}^H)=0 \quad \text{for all}\quad 1\le p\le m.
$$

Consider the blow up $\mu:\wt C\to C(X,r)$ of $C(X,r)$ at the cone point, and the associated Cartesian diagram:
\begin{equation*}
\xymatrix{
{X}\ar[r]^-{\iota}\ar[d]_{\mu}& {\wt C}\ar[d]^{\mu}\\
{\{0\}}\ar[r]^-{\iota}&{C(X,r)}
}
\end{equation*}
Note that $\wt C$ is an $\A^1$-bundle over $X$, and the exceptional divisor of $\mu$ is its zero section (see \cite[Section 3.8]{Kollar13} for a general discussion of affine cones). In particular, we have 
$$
\iota^*\IC^H_{\wt C}\isom\IC^H_X[1].
$$
By Saito's Decomposition Theorem \cite[Théorème 5.3.1]{Saito88}, we have
$$
\mu_*\IC^H_{\wt C}\isom \IC^H_{C(X,r)}\oplus\M^\bullet.
$$
Since $\mu$ is an isomorphism away from the cone point $\{0\}$, $\M^\bullet$ is supported on $\{0\}$ and satisfies the hard Lefschetz property. By the proper base change theorem \cite[(4.4.3)]{Saito90}, we have
$$
\mu_*\IC_X^H[1]\isom\iota^*\mu_*\IC^H_{\wt C}\isom \iota^*\IC^H_{C(X,r)}\oplus\iota^*\M^\bullet.
$$
Since the constructible cohomologies of the intersection complex are supported in degrees $<0$, we have $H^{\ge0}(\iota^*\IC^H_{C(X,r)})=0$, which induces an isomorphism
$$
H^i(\iota^*\M^\bullet)\isom\IH^{n+i}(X,\Q),\quad H^{-i}(\iota^*\IC^H_{C(X,r)})\oplus H^{-i}(\iota^*\M^\bullet)\isom \IH^{n-i}(X,\Q)
$$
for all $i\ge 0$. By the hard Lefschetz on $\M^\bullet$ and $\IH^\bullet(X,\Q)$, we have $H^i(\M^\bullet)\isom H^{-i}(\M^\bullet)(-i)$ and $\IH^{n+i}(X,\Q)\isom \IH^{n-i-2}(X,\Q)(-i-1)$. Therefore,
$$
H^{-i}(\iota^*\IC^H_{C(X,r)})\oplus \IH^{n-i-2}(X,\Q)(-1)\isom \IH^{n-i}(X,\Q).
$$
for $i\ge 0$. For the graded pieces of Hodge filtration, we have
\begin{equation}
\label{eqn:pullback IC of cone}
\gr^F_{-p}H^{-i}(\iota^*\IC^H_{C(X,r)})\oplus \gr^F_{-p+1}\IH^{n-i-2}(X,\C)\isom \gr^F_{-p}\IH^{n-i}(X,\C)
\end{equation}
(we implicitly treat $H^{-i}(\iota^*\IC^H_{C(X,r)})$ as the associated complex Hodge structure).
Since $X$ has $m$-rational singularities, we have the following equality of Hodge-Du Bois numbers and intersection Hodge numbers (see \cite[Section 4]{PP25a} for definitions)
$$
\h^{p,q}(X)=I\h^{p,q}(X)=h^{p,q}(X')
$$
for all $0\le p\le m$, where $X'$ is a smooth hypersurface of degree $d$; the first equality is \cite[Theorem 7.1]{PP25a} and the second equality is \cite[Corollary 1.4]{FL24}. Under the hypothesis $\frac{n+1}{d}>m+1$, we have the vanishing of Hodge numbers
$$
h^{p,q}(X')=0 \quad\text{for all}\quad q\neq p,\; 0\le p\le m
$$
and $H^{p,p}(X')=1$ for all $0\le p\le m$, which easily follows from the weak Lefschetz theorem and Griffiths' description of the middle primitive cohomology (see, for example, \cite[Corollary 6.12]{Voisin03}). In particular,
$$
I\h^{p-1,q-1}(X)=I\h^{p,q}(X) \quad\text{for all}\quad 1\le p\le m
$$
and applying this to \eqref{eqn:pullback IC of cone}, we deduce $\gr^F_{-p}H^{-i}(\iota^*\IC^H_{C(X,r)})=0$ for all $i$ and $1\le p\le m$, as desired.
\end{proof}

\subsection{From minimal exponent to GIT stability}

The global minimal exponent of a hypersurface determines the local minimal exponent of its classical affine cone at the cone point in a precise formula below. This is anticipated by Theorem \ref{thm:higher singularities of cone} in the case $r=1$, which is a key ingredient of the proof. The formula provides a direct bridge between the global minimal exponent of a hypersurface and the Hilbert-Mumford numerical criterion.

\begin{thm}
\label{thm:minimal exponent of cone}
For a hypersurface $X\subset \P^n$ of degree $d\ge 2$, defined by a homogeneous polynomial $f=0$, we have
$$
\wt\alpha_0(f)=\min \left\{\wt\alpha(X),\frac{n+1}{d}\right\}.
$$
\end{thm}

In particular, if $\wt\alpha(X)\ge\frac{n+1}{d}$, then $\wt\alpha_0(f)=\frac{n+1}{d}$.

\begin{proof}[Proof of Theorem \ref{thm:minimal exponent of cone}]
Recall that the affine cone $C(X,\O_X(1))\isom\mathrm{Cone}(X)$ is the hypersurface $\{f=0\}\subset\C^{n+1}$, and we have the equality
\begin{equation}
\label{eqn:minimal exponent of cone-origin}
\wt\alpha(\mathrm{Cone}(X)\sm\{0\})=\wt\alpha(X)
\end{equation}
since $\mathrm{Cone}(X)\sm\{0\}$ is a $\mathbb G_m$-torsor over $X$. From the lower semicontinuity of the minimal exponent (see \cite[Theorem E(2)]{MP20}), we have $\wt\alpha_0(f)\le\wt\alpha(X)$. Combined with \eqref{eqn:minimal exponent multiplicity bound}, we have
$$
\wt\alpha_0(f)\le\min \left\{\wt\alpha(X),\frac{n+1}{d}\right\}.
$$

When $\wt\alpha_0(f)=\frac{n+1}{d}$, the equality holds. Hence, it suffices to prove $\wt\alpha_0(f)=\wt\alpha(X)$ when $\wt\alpha_0(f)<\frac{n+1}{d}$. Recall that $\wt\alpha_0(f)\in \Q$ by Kashiwara's rationality theorem \cite{Kashiwara76}, so there exists a sufficiently divisible positive integer $N$ such that $N\wt\alpha_0(f)\in \Z$.

Consider homogeneous polynomials $f_1,\dots,f_N$ defined by
$$
f_i(x_{i,0},\dots,x_{i,n})=f(x_{i,0},\dots,x_{i,n}).
$$
Denote by
$$
F=f_1\oplus \cdots\oplus f_N\in \C[x_{i,j}]_{1\le i\le N,\; 0\le j\le n},
$$
and $X_N\subset \P^{N(n+1)-1}$ be the degree $d$ hypersurface, defined by $F=0$. Consider a singular point $p_N=[p_{i,j}]\in X_N$. Without loss of generality, assume that $p_{1,0}\neq 0$. In the affine chart $\{x_{1,0}\neq 0\}\subset \P^{N(n+1)-1}$ which contains $p_N$, the equation of $X_N$ is given by
$$
F(1,x_{i,j})_{(i,j)\neq(1,0)}=f_1(1,x_{1,1},\dots,x_{1,n})\oplus f_2\oplus\cdots\oplus f_N=0.
$$
Hence, by Thom-Sebastiani Theorem \ref{thm:Thom-Sebastiani for minimal exponents} and \eqref{eqn:minimal exponent of cone-origin}, the local minimal exponent of $X_N$ at $p_N$ is at least $\wt\alpha(X)+(N-1)\wt\alpha_0(f)$. Furthermore, this lower bound is attained at the point
$$
[p_{1,0}:\dots:p_{1,n}:0:\dots:0]\in \P^{N(n+1)-1}
$$
when $p:=[p_{1,0}:\dots:p_{1,n}]\in X$ is the point satisfying $\wt\alpha_p(X)=\wt\alpha(X)$. Therefore,
$$
\wt\alpha(X_N)=\wt\alpha(X)+(N-1)\wt\alpha_0(f).
$$

On the other hand, by Thom-Sebastiani Theorem \ref{thm:Thom-Sebastiani for minimal exponents}, we have
$$
\wt\alpha_0(F)=N\wt\alpha_0(f)<\frac{N(n+1)}{d}.
$$
Denote by $m:=\wt\alpha_0(F)-1\in \Z$. Then, Theorem \ref{thm:higher singularities vs minimal exponent} implies that $\mathrm{Cone}(X_N)$ has $m$-Du Bois singularities, but not $m$-rational singularities (in other words, has $m$-liminal singularities).
Hence, Theorem \ref{thm:higher singularities of cone} implies that $X_N$ has $m$-liminal singularities, or equivalently,
$$
\wt\alpha(X_N)=m+1=N\wt\alpha_0(f).
$$
Therefore, $\wt\alpha_0(f)=\wt\alpha(X)$.
\end{proof}

Next, we prove Theorem \ref{thm:GIT stability via minimal exponent}. Earlier works \cite{Hacking04,KL04,Lee08} prove GIT semistability (resp. stability) of non-Fano hypersurfaces with log canonical threshold $\ge \frac{n+1}{d}$ (resp. $>$), by recasting the Hilbert-Mumford criterion as pointwise inequalities for local log canonical thresholds on $X$.

A naïve substitution of the minimal exponent for the log canonical threshold fails: at a smooth point $x\in X$, Proposition \ref{prop:weight-minimal exponent inequality} cannot be applied, so one cannot verify the Hilbert–Mumford criterion from the local data on $X$ alone. We overcome this by considering both the cone point of the affine cone and the points of $X$.

\begin{proof}[Proof of Theorem \ref{thm:GIT stability via minimal exponent}]
Denote by $f=0$ the defining degree $d$ homogeneous polynomial of $X$. By Theorem \ref{thm:minimal exponent of cone}, if $\wt\alpha(X)\ge \frac{n+1}{d}$, then $\wt\alpha_0(f)=\frac{n+1}{d}$. Applying Propositions \ref{prop:Hilbert-Mumford criterion} and \ref{prop:weight-minimal exponent inequality}, this implies that $X$ is GIT semistable.

Now, assume $\wt\alpha(X)>\frac{n+1}{d}$. It remains to prove that $X$ is GIT stable. We argue by contradiction: suppose that $X$ is not GIT stable. Then, by the Hilbert-Mumford numerical criterion, there exist a nontrivial rational weight system $w=(w_0,...,w_n)$ and $g\in SL(n+1)$ such that
$$
\mathrm{wt}_w(f\circ g) \ge \frac{d}{n+1}\sum_{i=0}^n w_i.  
$$
Without loss of generality, we replace $f\circ g$ with $f$ and assume that $w_0\le w_1\le\cdots\le w_n$, not all equal.

Let $k\ge0$ be an integer such that $w_0=\dots=w_k<w_{k+1}$. Then, for any linear change of coordinates $h\in SL(k+1)\subset SL(n+1)$ within $x_0,\dots,x_k$, we have
\begin{equation}
\label{eqn:Hilbert-Mumford 0 to k}
\mathrm{wt}_w(f\circ h)=\mathrm{wt}_w(f) \ge \frac{d}{n+1}\sum_{i=0}^n w_i.  
\end{equation}

\smallskip
\noindent
\textit{Claim 1.}  $\P^{k}=\{x_{k+1}=\dots=x_n=0\}\subset X\subset \P^n$ and $X$ is smooth along $\P^k$.

For any point $x\in \P^k$, there exists $h\in SL(k+1)$ satisfying $h(x)=[1:0:\dots:0]$. We replace $f\circ h$ with $f$, and $x$ with $[1:0:\dots:0]$. By \eqref{eqn:Hilbert-Mumford 0 to k}, $x_0^d$ does not appear in $f$ with nonzero coefficients, and thus, $[1:0:\dots:0]\in X$.

We argue by contradiction: suppose $[1:0:\dots:0]\in X$ is a singular point of $X$. Equivalently, the hypersurface $\{f(1,x_1,\dots,x_n)=0\}\subset\C^n$ is singular at $(x_1,\dots,x_n)=0$. For a weight system
$$
\wt w=(\wt w_1,\dots, \wt w_n):=(w_1-w_0,\dots,w_n-w_0),
$$
we have
$$
\mathrm{wt}_{\wt w}(f(1,x_1,\dots,x_n))=\mathrm{wt}_{w}(f)-dw_0\ge \frac{d}{n+1}\sum_{i=1}^{n} \wt w_i.
$$
This implies
$$
\wt\alpha(X)\le \wt\alpha_0(f(1,x_1,\dots,x_n))\le \frac{n+1}{d}
$$
by Proposition \ref{prop:weight-minimal exponent inequality}, which is a contradiction.

\smallskip
\noindent
\textit{Claim 2.} Let $e$ be a positive integer. Denote by
$$
f_{n-e}(x_0,\dots,x_{n-e}):=f(x_0,\dots,x_{n-e},0,\dots,0)\in \C[x_0,\dots,x_{n-e}].
$$
If $(e-1)(d-1)-1\le k$, then $f_{n-e}\neq 0$, $\wt\alpha_0(f_{n-e})=\frac{n+1}{d}-e$, and $w_n=\dots=w_{n-e+1}$.

We proceed by induction on $e$. Suppose $e=1$. By Claim 1,  $X$ is smooth at $[1:0:\dots:0]$, so $x_0^{d-1}x_l$ appears in $f$ with nonzero coefficients for some $l\neq0$. Therefore,
\begin{equation}
\label{eqn:monomial d-1,1}
(d-1)w_0+w_n\ge(d-1)w_0+w_l\ge \frac{d}{n+1}\sum_{i=0}^n w_i.
\end{equation}
In particular, $d\le n+1$. Write
$$
f(x_0,\dots,x_n)=x_nF_n(x_0,\dots,x_n)+f_{n-1}(x_0,\dots,x_{n-1}).
$$
If $f_{n-1}=0$, then $f=x_nF_n$ which implies $\frac{n+1}{d}<\wt\alpha(X)\le 1$; this cannot happen, so $f_{n-1}\neq0$.

Consider a weight system
$$
w'=(w_0',\dots, w_{n-1}'):=(0,w_1-w_0,\dots,w_{n-1}-w_0).
$$
If this weight system is zero, then $w_0=\cdots=w_{n-1}$. By \eqref{eqn:Hilbert-Mumford 0 to k}, this implies that $f_{n-1}=0$. Hence, $w'$ is a nonnegative nontrivial weight system, and \eqref{eqn:monomial d-1,1} implies $d<n+1$.

Denote by $w_n':=w_n-w_0$. From \eqref{eqn:Hilbert-Mumford 0 to k} and \eqref{eqn:monomial d-1,1}, we have
$$
\mathrm{wt}_{w'}(f_{n-1})\ge\frac{d}{n+1}\sum_{i=0}^{n} w_i' \quad \text{and} \quad w_n'\ge \frac{d}{n+1}\sum_{i=0}^{n} w_i',
$$
from which we deduce
$$
\mathrm{wt}_{w'}(f_{n-1})\ge \frac{d}{n+1-d}\sum_{i=0}^{n-1} w_i'.
$$
By Proposition \ref{prop:weight-minimal exponent inequality}, this implies that $\wt \alpha_0(f_{n-1})\le\frac{n+1}{d}-1$.

On the other hand, we have
$$
\frac{n+1}{d}=\wt\alpha_0(f)\le\wt\alpha_0(x_nF_n)+\wt\alpha_0(f_{n-1})\le 1+\wt\alpha_0(f_{n-1})
$$
where the first inequality follows from \cite[Proposition 6.6]{MP20}. Therefore, $\wt \alpha_0(f_{n-1})=\frac{n+1}{d}-1$. This completes the case $e=1$.

Suppose the claim is true for $e\ge 1$. We prove for $e+1$: if $e(d-1)-1\le k$, then $f_{n-e-1}\neq 0$, $\wt\alpha_0(f_{n-e-1})=\frac{n+1}{d}-e-1$, and $w_n=\dots=w_{n-e}$. Write
$$
f=x_nF_n+\dots+x_{n-e+1}F_{n-e+1}+f_{n-e}, \quad f_{n-e}=x_{n-e}F_{n-e}+f_{n-e-1}.
$$
for some $F_{n-i}\in \C[x_0,\dots,x_{n-i}]$ for all $0\le i\le e$.

Note that $k\le n-e$, since by the induction hypothesis, we have $w_n=\dots=w_{n-e+1}$ and $w_0=\dots=w_k$ while $w$ is a nontrivial weight system. If the hypersurface $\{f_{n-e}=0\}\subset\P^{n-e}$ is singular along every point of
$$
\P^k=\{x_{k+1}=\dots=x_{n-e}=0\}\subset\P^{n-e},
$$
then $X$ is singular along
$$
\P^k\cap\{F_n=\dots=F_{n-e+1}=0\}\subset X,
$$
which is nonempty because $k\ge e$. Hence, the hypersurface $\{f_{n-e}=0\}\subset\P^{n-e}$ contains a smooth point in $\P^k$. By an appropriate linear change of coordinates within $x_0,\dots,x_k$, the smooth point becomes $[1:0:\dots:0]$. Note that this does not affect the statement of Claim 2.

As a consequence, $x_0^{d-1}x_l$ appears in $f_{n-e}$ with nonzero coefficients for some $0<l\le n-e$. Therefore, we have
$$
(d-1)w_0+w_{n-e}\ge(d-1)w_0+w_l\ge \frac{d}{n+1}\sum_{i=0}^n w_i,
$$
which implies
\begin{equation}
\label{eqn:monomial d-1,1 (2)}
w_{n-e}'\ge \frac{d}{n+1}\sum_{i=0}^n w_i',
\end{equation}
using the notation $w_i'=w_i-w_0$, as before. In particular, this implies $w_{n-e}'$ is positive and $d(e+1)\le n+1$. If $f_{n-e-1}=0$, then
$$
f=x_nF_n+\dots+x_{n-e+1}F_{n-e+1}+x_{n-e}F_{n-e}.
$$
For a singular point $p\in \{x_{n-e}=\dots=x_{n}=F_{n-e}=\dots=F_n=0\}\subset X$ of $X$, we have
$$
\frac{n+1}{d}<\wt\alpha(X)\le\wt\alpha_p(X)\le e+1.
$$
The last inequality is due to \cite[Proposition 6.6]{MP20}; at $p$, the local minimal exponent for each term $x_{n-i}F_{n-i}$ is at most $1$, for all $0\le i\le e$. This cannot happen, since $d(e+1)\le n+1$. Thus, $f_{n-e+1}\neq 0$.

Consider a weight system
$$
w^{(e+1)}=(w_0',\dots, w_{n-e-1}')=(0,w_1-w_0,\dots,w_{n-e-1}-w_0).
$$
This is a nontrivial weight system; if not, \eqref{eqn:Hilbert-Mumford 0 to k} implies that $f_{n-e-1}=0$. From this nontriviality and \eqref{eqn:monomial d-1,1 (2)}, we have $d(e+1)<n+1$.

From \eqref{eqn:Hilbert-Mumford 0 to k}, we have
$$
\mathrm{wt}_{w^{(e+1)}}(f_{n-e-1})\ge\frac{d}{n+1}\sum_{i=0}^{n} w_i'.
$$
Then, \eqref{eqn:monomial d-1,1 (2)} induces
$$
\mathrm{wt}_{w^{(e+1)}}(f_{n-e-1})\ge \frac{d}{n+1-(e+1)d}\sum_{i=0}^{n-e-1} w_i',
$$
and the equality holds only if $w_n=\dots=w_{n-e}$. By Proposition \ref{prop:weight-minimal exponent inequality}, this implies that $\wt\alpha_0(f_{n-e-1})\le\frac{n+1}{d}-e-1$.

On the other hand, we have
$$
\frac{n+1}{d}=\wt\alpha_0(f)\le\wt\alpha_0(x_nF_n)+\dots+\wt\alpha_0(x_{n-e}F_{n-e})+\wt\alpha_0(f_{n-e-1})\le e+1+\wt\alpha_0(f_{n-e-1})
$$
as in the case $e=1$. Therefore, $\wt \alpha_0(f_{n-e-1})=\frac{n+1}{d}-e-1$ and all the above inequalities are equalities. Hence, $w_n=\dots=w_{n-e}$.

\smallskip
\noindent
\textit{Claim 3.} If $(e-1)(d-1)-1\le k < e(d-1)-1$, then $\wt\alpha_0(f_{n-e})<\frac{n+1}{d}-e$.

Let $\epsilon\ll 1$ be a small positive rational number, and consider a nonnegative weight system
$$
w_\epsilon''=(w_0'',\dots, w_{n-e}''):=(0,\dots,0,w_{k+1}'-\epsilon,\dots, w_{n-e}'-\epsilon).
$$
This is nontrivial from the proof of Claim 2, and we have
$$
\mathrm{wt}_{w_\epsilon''}(f_{n-e})\ge \mathrm{wt}_{w^{(e)}}(f_{n-e})-d\epsilon\ge\frac{d}{n+1-ed}\sum_{i=0}^{n-e} w_i'-d\epsilon>\frac{d}{n+1-ed}\sum_{i=0}^{n-e} w_i''
$$
where the last inequality follows from the assumption $k< e(d-1)-1$. By Proposition \ref{prop:weight-minimal exponent inequality}, we obtain $\wt\alpha_0(f_{n-e})<\frac{n+1}{d}-e$.

Take $e=\left[\frac{k+1}{d-1}\right]+1$. Then Claim 3 contradicts Claim 2.
\end{proof}

As a byproduct, we prove a precise inversion of adjunction formula for homogeneous polynomials, which may be of independent interest.

\begin{prop}
\label{prop:adjunction minimal exponent general hyperplane}
Let $f\in \C[x_0,\dots,x_n]$ be a nonzero homogeneous polynomial of degree $d\ge 2$. For a general hyperplane $H\subset \C^{n+1}=\mathrm{Spec}\; \C[x_0,\dots,x_n]$ through the origin, we have
$$
\wt\alpha_0(f|_H)=\min\left\{\wt\alpha_0(f),\frac{n}{d}\right\}.
$$
\end{prop}

\begin{proof}
Let $X\subset \P^n$ be a hypersurface defined by $f=0$. Denote by $\P H\subset \P^n$ a projective hyperplane associated to a general hyperplane $H\subset \C^{n+1}$ through the origin. Then by \cite[Lemma 7.5]{MP20a}, we have
$$
\wt\alpha(X\cap\P H)\ge\wt\alpha(X).
$$
Applying Theorem \ref{thm:minimal exponent of cone}, we obtain
$$
\wt\alpha_0(f|_H)=\min\left\{\wt\alpha(X\cap\P H),\frac{n}{d}\right\}\ge \min\left\{\wt\alpha(X),\frac{n}{d}\right\}=\min\left\{\wt\alpha_0(f),\frac{n}{d}\right\}.
$$
On the other hand, $\wt\alpha_0(f|_H)\le\wt\alpha_0(f)$ by \cite[Theorem E(1)]{MP20} and $\wt\alpha_0(f|_H)\le\frac{n}{d}$ by \eqref{eqn:minimal exponent multiplicity bound}, which complete the proof.
\end{proof}

\subsection{From GIT stability of cubic hypersurfaces to minimal exponent}

Motivated by the conjectural equivalence between GIT stability and K-stability for cubic hypersurfaces, one expects that every GIT semistable cubic has canonical singularities -- formalized as \cite[Question 5.8]{SS17}. This is known in dimension $\le 4$, where explicit classifications of GIT polystable cubics are available. For dimension $\ge 5$, however, the complexity of the GIT problem has made it very difficult to classify GIT (semi)stable cubics or to understand their singularities. In this section, we resolve \cite[Question 5.8]{SS17} by establishing a lower bound for the minimal exponent of GIT semistable cubic hypersurfaces, which in particular implies that they have canonical singularities.

We begin by relating the minimal exponent of a hypersurface to that of the complement of a subvariety. As an application, we obtain a lower bound for the global minimal exponent of a hypersurface in terms of the dimension of its singular locus.

\begin{thm}
\label{thm:minimal exponent of complement}
Let $X\subset \P^n$ be a hypersurface of degree $d\ge 2$. For any closed subset $Z\subset X$, we have
$$
\wt\alpha(X)\ge \min\left\{\wt\alpha(X\sm Z),\frac{n-\dim Z}{d}\right\}.
$$
In particular,
$$
\wt\alpha(X)\ge\frac{n-\dim\Sing(X)}{d}.
$$
\end{thm}

Here, $\Sing(X)$ denotes the singular locus of $X$. By Theorem \ref{thm:minimal exponent of cone} and semicontinuity of the minimal exponent, the last inequality is equivalent to the lower bound in \cite[Theorem E(3)]{MP20}. Moreover, combined with Theorem \ref{thm:minimal exponent of cone}, the first inequality recovers the bound on log canonical thresholds in \cite[Theorem 0.2]{dFEM03} for homogeneous hypersurfaces in $\A^{n+1}$.

\begin{proof}
Denote by $s=\dim Z$. Consider the affine cone $\{f=0\}\subset\C^{n+1}$, and let $L\subset \C^{n+1}$ be the intersection of $(s+1)$-general hyperplanes through the origin. Then $(X\sm Z)\cap \P L\subset \P L$ is a projective hypersurface whose affine cone is $\{f|_L=0\}\subset L$. Additionally,
$$
\wt\alpha((X\sm Z)\cap \P L)\ge \wt\alpha(X\sm Z)
$$
by \cite[Lemma 7.5]{MP20a}. Hence,
$$
\wt\alpha_0(f|_L)\ge \min\left\{\wt\alpha(X\sm Z),\frac{n-s}{d}\right\}
$$
by Theorem \ref{thm:minimal exponent of cone}. Moreover, we have $\wt\alpha_0(f)\ge\wt\alpha_0(f|_L)$ by \cite[Theorem E(1)]{MP20}, which implies that
$$
\wt\alpha(X)\ge\wt\alpha_0(f)\ge\min\left\{\wt\alpha(X\sm Z),\frac{n-s}{d}\right\}.
$$
This completes the proof.
\end{proof}

By Theorem \ref{thm:minimal exponent of complement}, the upper bound on the dimension of the singular locus yields the lower bound on the minimal exponent. In the case of a cubic hypersurface, the secant variety of its singular locus is contained in the hypersurface. We utilize this fact to obtain a bound.

\begin{lem}
\label{lem:dim sing bound for cubic}
Let $X\subset \P^n$ be a GIT semistable cubic hypersurface. Then
$$
\dim \Sing(X)\le\frac{2n-1}{3}.
$$
\end{lem}

\begin{proof}
Note that a cubic hypersurface contains the secant variety of its singular locus. Let $s:=\dim \Sing(X)$, and pick a smooth point $x\in \Sing(X)$ with the reduced structure. After a linear change of coordinates, we may assume that the tangent plane to $\Sing(X)$ at $x$ is 
$$
\P^s=\{x_{s+1}=\dots=x_n=0\}\subset \P^n.
$$
Since $\P^s\subset X$, the defining cubic polynomial $f$ of $X$ is contained in the ideal $(x_{s+1},\dots,x_n)$.

Consider the weight system $w=(w_0,\dots,w_n)$ defined by
$$
w_0=\dots=w_s=-1,\quad w_{s+1}=\dots=w_n=2+\epsilon,
$$
for some $\epsilon \in \Q$ such that $\sum_iw_i=0$. Equivalently, this condition is $(n-s)(2+\epsilon)=s+1$. If $\epsilon>0$, then $\mathrm{wt}_w(f)>0$, which contradicts GIT semistability of $X$ by the Hilbert-Mumford numerical criterion. Hence, we must have
$$
(n-s)\epsilon=s+1-2(n-s)\le0,
$$
which completes the proof.
\end{proof}

Next, we prove geometric obstructions for the GIT semistability of cubics, crucial for the proof of Theorem \ref{thm:cubic converse via minimal exponent}.

\begin{lem}
\label{lem:geometric obstruction to cubic GIT}
Let $n\ge 6$, and let $X\subset \P^n$ be a cubic hypersurface. If either of the following holds:
\begin{enumerate}
    \item $X$ contains an $(n-2)$-plane, i.e. $\P^{n-2} \subset X$; or
    \item $\Sing(X)$ contains a hypersurface $Z\subset \P^{n-3}\subset \P^n$ of an $(n-3)$-plane,
\end{enumerate}
then $X$ is not GIT semistable.
\end{lem}

\begin{proof} If $X$ contains an $(n-2)$-plane, then after a linear change of coordinates, we may assume that the defining cubic $f$ of $X$ is contained in the ideal $(x_{n-1},x_n)$. Then the weight system $w$ with $s=n-2$ given in the proof of Lemma \ref{lem:dim sing bound for cubic} shows that $X$ is not GIT semistable.

Assume condition (2). After a linear change of coordinates, the $(n-3)$-plane becomes
$$
\{x_{n-2}=x_{n-1}=x_n=0\}\subset \P^n.
$$
Let $g(x_0,\dots,x_{n-3})$ be the defining homogeneous polynomial of $Z$ in $\P^{n-3}$.

If $Z$ is a hyperplane, then after a linear change of coordinates, we may assume $g=x_{n-3}$. Hence, $X$ is singular along
$$
\{x_{n-3}=x_{n-2}=x_{n-1}=x_n=0\}\subset \P^n
$$
which implies that $f\in (x_{n-3},x_{n-2},x_{n-1},x_n)^2$. Consider the weight system $w=(w_0,\dots,w_n)$ defined by
$$
w_0=\dots=w_{n-4}=-2,\quad w_{n-3}=\dots=w_n=1+\epsilon,
$$
for some $\epsilon \in \Q$ such that $\sum_iw_i=0$. Since $n\ge 6$, we have $\epsilon>0$ and $\mathrm{wt}_w(f)>0$. Therefore, $X$ is not GIT semistable.

If $Z$ is not a hyperplane, then the secant variety of $Z$ is $\P^{n-3}$. This implies that $\P^{n-3}\subset X$, or equivalently
$$
f\in (x_{n-2},x_{n-1},x_n).
$$
Then,
$$
f=x_{n-2}g_{n-2}+x_{n-1}g_{n-1}+x_ng_n+h
$$
for some polynomials $g_{n-2},g_{n-1},g_n\in \C[x_0,\dots,x_{n-3}]$ and $h\in (x_{n-2},x_{n-1},x_n)^2$. Since $Z\subset \Sing(X)$, we have
$$
\frac{\partial f}{\partial x_{n-2}}=\frac{\partial f}{\partial x_{n-1}}=\frac{\partial f}{\partial x_{n}}=0\quad \mathrm{on}\quad Z.
$$
This implies that $g_{n-2}=g_{n-1}=g_{n}=0$ on $Z\subset \P^{n-3}$. Note that $g$ is not linear and $g_{n-2}$, $g_{n-1}$, and $g_{n}$ are quadratic. Hence, $g_{n-2}$, $g_{n-1}$, and $g_{n}$ are constant multiples of $g$. After a linear coordinate change within $x_{n-2},x_{n-1},x_n$, we have
$$
f=cx_ng+h,\quad h\in (x_{n-2},x_{n-1},x_n)^2
$$
for some $c\in \C$. Consider the weight system $w=(w_0,\dots,w_n)$ defined by
$$
w_0=\dots=w_{n-3}=-1,\quad w_{n-2}=w_{n-1}=\frac{1}{2}+\epsilon,\quad w_n=2+\epsilon,
$$
for some $\epsilon \in \Q$ such that $\sum_iw_i=0$. Since $n\ge 6$, we have $\epsilon>0$ and $\mathrm{wt}_w(f)>0$. Therefore, $X$ is not GIT semistable.
\end{proof}

Recall that the rank of the Hessian matrix -- equivalently, the rank of the quadratic part of the defining equation -- of a hypersurface singularity is independent of the chosen local equation. If this rank is $r$, then the singularity (of dimension $n-1$) is, up to analytic change of coordinates, locally equivalent to the hypersurface
$$
\{x_1^2+\dots+x_r^2+g(x_{r+1},\dots,x_n)=0\}\subset \C^{n}
$$
near the origin, where $g$ has no quadratic terms. By Thom-Sebastiani Theorem \ref{thm:Thom-Sebastiani for minimal exponents}, the minimal exponent is at least $\frac{r}{2}$. In particular, if the minimal exponent of the singularity is $<\frac{7}{4}$, then $r\le 3$. We further prove that, under the assumption of GIT semistability for cubic fivefolds and higher, this rank is in fact at most $2$.

\begin{lem}
\label{lem:cubic Hessian rank}
Let $n\ge 6$ and let $X\subset\P^n$ be a GIT semistable cubic hypersurface. For a singular point $x\in X$ such that 
$$
\wt\alpha_x(X)<\frac{7}{4},
$$
the rank of the Hessian matrix of the singularity at $x\in X$ is at most $2$.
\end{lem}

\begin{proof}
We argue by contradiction: suppose the rank is $3$. Without loss of generality, we assume $x=[0:\dots:0:1]$, and $f(x_0,\dots,x_{n-1},1)$ has quadratic part of rank $3$, where $f$ is the defining homogeneous cubic polynomial of $X$. After a linear change of coordinates, we have the following local defining polynomial $P$ at $x$:
$$
P:=f(x_0,\dots,x_{n-1},1)=x_{0}^2+x_{1}^2+x_{2}^2+f(x_0,\dots,x_{n-1},0).
$$
Then there exist homogeneous quadratics $q_0,q_{1},q_{2}$ and a homogeneous cubic $g$ in $\C[x_3,\dots,x_{n-1}]$ such that
$$
F(x_0,\dots,x_{n-1}):=f(x_0\dots,x_{n-1},0)-x_{0}q_0-x_{1}q_1-x_{2}q_2-g\in (x_{0},x_{1},x_{2})^2.
$$
Denote by $y_j:=x_{j}+\frac{q_j}{2}$ for $0\le j\le 2$. With this change of coordinates, we have
$$
P=y_0^2+y_1^2+y_2^2-\frac{q_0^2}{4}-\frac{q_1^2}{4}-\frac{q_2^2}{4}+g+F(y_0-\frac{q_0}{2},y_1-\frac{q_1}{2},y_2-\frac{q_2}{2},x_3,\dots,x_{n-1}).
$$

Define a $\mathbb G_m$-action on the coordinates as
$$
t\cdot x_i=t^2x_i\;\;(3\le i\le n-1),\quad t\cdot y_j=t^3y_j\;\;(0\le j\le 2).
$$
Taking the limit $t\to 0$, we obtain a $\mathbb G_m$-equivariant degeneration of $P$ to
$$
y_0^2+y_1^2+y_2^2+g.
$$
If $g\neq 0$, then by Thom-Sebastiani Theorem \ref{thm:Thom-Sebastiani for minimal exponents} and the lower semicontinuity of the minimal exponent \cite[Theorem E(2)]{MP20}, we have 
$$
\wt\alpha_0(P)\ge\wt\alpha_0(y_0^2+y_1^2+y_2^2+g)=\frac{3}{2}+\wt\alpha_0(g)\ge\frac{3}{2}+\frac{1}{3}.
$$
The last inequality follows from \cite[Proposition 9.5.13]{LazarsfeldII}, that the log canonical threshold of the divisor $\{g=0\}$ at the origin is at least $\frac{1}{\mathrm{mult}_0(g)}$. This contradicts $\wt\alpha_0(P)=\wt\alpha_x(X)<\frac{7}{4}$, and thus, $g=0$.

Next, define a $\mathbb G_m$-action on the coordinates as
$$
t\cdot x_i=tx_i\;\;(3\le i\le n-1),\quad t\cdot y_j=t^2y_j\;\;(0\le j\le 2).
$$
Taking the limit $t\to 0$, we obtain a $\mathbb G_m$-equivariant degeneration of $P$ to
$$
y_0^2+y_1^2+y_2^2-\frac{q_0^2}{4}-\frac{q_1^2}{4}-\frac{q_2^2}{4}.
$$
As above, if $q_0^2+q_1^2+q_2^2\neq0$, then we have
$$
\wt\alpha_0(P)\ge \frac{3}{2}+\wt\alpha_0(q_0^2+q_1^2+q_2^2) \ge \frac{3}{2}+\frac{1}{4}.
$$
Thus, $q_0^2+q_1^2+q_2^2=0$.

In summary, we have
$$
f=x_{0}q_0+x_{1}q_1+x_{2}q_2+h, \quad h\in (x_{0},x_{1},x_{2})^2
$$
such that $q_0^2+q_1^2+q_2^2=0$.

If $q_0,q_1,q_2$ are proportional, that is, constant multiples of a quadratic $q\in \C[x_3,\dots,x_{n-1}]$, then $\Sing(X)$ contains
$$
\{x_0=x_1=x_2=q=0\}\subset \P^n.
$$
This is condition (2) in Lemma \ref{lem:geometric obstruction to cubic GIT}, hence a contradiction.

From $(q_1+\sqrt{-1}q_2)(q_1-\sqrt{-1}q_2)=-q_0^2$, the rank of the quadratic polynomial $q_0$ is $2$. Indeed, if $q_0=0$, then $q_1$ is a constant multiple of $q_2$. If $q_0$ has rank greater than $2$, then $q_0$ is an irreducible polynomial. This implies that $q_1$ and $q_2$ are constant multiples of $q_0$. If $q_0$ has rank $1$, then by a linear change of coordinates, we may assume $q_0=x_3^2$, which implies that $q_1$ and $q_2$ are constant multiples of $x_3^2$.

After a linear change of coordinates within $x_3,\dots,x_{n-1}$, we may assume $q_0=x_3x_4$ and
$$
(q_1+\sqrt{-1}q_2)(q_1-\sqrt{-1}q_2)=-x_3^2x_4^2.
$$
This implies that $q_1,q_2\in \C[x_3,x_4]$, and thus, $q_0,q_1,q_2$ are quadratic polynomials in variables $x_3$ and $x_4$. Consider the weight system $w=(w_0,\dots,w_n)$ defined by
$$
w_0=w_1=w_2=1+\epsilon,\quad w_3=w_4=-\frac{1}{2},\quad w_5=\dots=w_n=-2,
$$
for some $\epsilon \in \Q$ such that $\sum_iw_i=0$. Since $n\ge 6$, we have $\epsilon>0$ and $\mathrm{wt}_w(f)>0$. Therefore, $X$ is not GIT semistable, which is a contradiction.
\end{proof}

Lastly, before proving Theorem \ref{thm:cubic converse via minimal exponent}, we relate the minimal exponent with the minimal log discrepancy. Following \cite[Section 2.1]{Kollar13}, the minimal log discrepancy of a pair $(X,\Delta)$ over an irreducible subvariety $W\subset X$ is defined as
$$
\mathrm{mld}(W;X,\Delta):= \min_E\{1+a(E;X,\Delta) : \mathrm{center}_X(E)= W\}
$$
where the minimum is taken every irreducible divisor $E$ of $Y$ for every birational morphism $\mu:Y\to X$ with $\mu(E)=W$; here, $a(E;X,\Delta)$ is the discrepancy of $E$ with respect to the pair $(X,\Delta)$. If $\Delta$ is empty, we omit this.

\begin{prop}
\label{prop:minimal exponent vs mld}
Let $X$ be a variety with hypersurface singularities and $W\subset \Sing(X)$ be an irreducible subvariety in the singular locus. If $\wt\alpha(X)> 1+\frac{k}{2}$ for a nonnegative integer $k$, then the minimal log discrepancy of $X$ over $W$ satisfies
$$
\mathrm{mld}(W;X)\ge k+1.
$$
In particular, if $\wt\alpha(X)>\frac{3}{2}$, then $X$ has terminal singularities.
\end{prop}

\begin{proof}
We may assume $X$ is quasi-projective. When the dimension of $W$ is positive, for a general hyperplane section $H\subset X$, we have
$$
\mathrm{mld}(W;X)=\mathrm{mld}(W\cap H;X\cap H).
$$
This follows from the Bertini theorem. Additionally, by \cite[Lemma 7.5]{MP20a}, we have
$$
\wt\alpha(X\cap H)\ge\wt\alpha(X)>1+\frac{k}{2}.
$$
Therefore, it suffices to prove when $W$ is a closed point $x\in X$. Let $L$ be a general hyperplane section through $x\in X$. Then by \cite[Theorem 1.5]{DM23}, we have
$$
\wt\alpha(X\cap L)\ge\wt\alpha(X)-\frac{1}{2}>1+\frac{k-1}{2},
$$
and by the inversion of adjunction \cite[Theorem 1.1]{EM04}, we have
$$
\mathrm{mld}(x;X)\ge \mathrm{mld}(x;X\cap L)+1.
$$
Recall that $X$ has canonical singularities when $\wt\alpha(X)>1$. Additionally, $X$ is smooth when $\wt\alpha(X)>\frac{\dim X+1}{2}$ (see e.g. \cite[Theorem E(3)]{MP20}); this implies that $X$ is smooth at the generic point of $W$ if
$$
\wt\alpha(X)>\frac{\codim_X W+1}{2},
$$
and thus, $k\le\codim_XW-2$ from the assumption $W\subset \Sing(X)$. Hence, applying the hyperplane section argument $k$-times completes the proof.
\end{proof}

We finally prove Theorem \ref{thm:cubic converse via minimal exponent}, using the materials developed in this section.

\begin{proof}[Proof of Theorem \ref{thm:cubic converse via minimal exponent}]
By Theorem \ref{thm:minimal exponent of complement} and Lemma \ref{lem:dim sing bound for cubic}, we have
$$
\wt\alpha(X)\ge\frac{n+1}{9}.
$$
From the explicit GIT analysis in \cite{Allcock03, Yokoyama02,Yokoyama08, Laza09, Laza10}, we have
$$
\wt\alpha(X)\ge\frac{4}{3} \quad \mathrm{when} \quad 3\le n\le 5;
$$
see Remark \ref{rmk:boundary cubics} below. Therefore, it remains to prove
$$
\wt\alpha(X)\ge\frac{5}{3} \quad \mathrm{when} \quad n\ge 6.
$$
In this case, Proposition \ref{prop:minimal exponent vs mld} implies that $X$ has terminal singularities.

Denote by
$$
S:=\left\{x\in X:\wt\alpha_x(X)<\frac{7}{4}\right\}.
$$
This is a closed subset of $X$ by the discreteness and upper semicontinuity of the minimal exponent. If $\dim S\le n-5$, then Theorem \ref{thm:minimal exponent of complement} implies $\wt\alpha(X)\ge \frac{5}{3}$. Hence, the following claim completes the proof.

\smallskip
\noindent
\textit{Claim.} $\dim S\le n-5$.

We argue by contradiction: suppose $\dim S\ge n-4$. Pick an irreducible component $Z$ of $S$ with $\dim Z\ge n-4$. Let $\P^e\subset \P^n$ be the minimal projective subspace containing $Z$:
$$
Z\subset \P^e\subset\P^n.
$$
In other words, $Z$ is a nondegenerate subvariety of $\P^e$. By the assumption, $e\ge n-4$.

If $e\le n-3$, then this contradicts Lemma \ref{lem:geometric obstruction to cubic GIT} (2).

If $e=n-2$, then the secant variety of $Z$ is $\P^e$. Indeed, the secant variety of a nondegenerate curve in $\P^3$ is $\P^3$. Since nondegeneracy is preserved by taking general hyperplane sections, this implies that the secant variety of the intersection $Z\cap \P^3$ is $\P^3$ for a general $3$-plane $\P^3\subset \P^e$. Therefore, $\P^e\subset X$, which contradicts Lemma \ref{lem:geometric obstruction to cubic GIT} (1).

Assume $e\ge n-1$. By Lemma \ref{lem:cubic Hessian rank}, the rank of the Hessian matrix at $x\in Z$ is at most $2$. We choose $n$-general points of $Z$ that spans $\P^{n-1}$. After a linear change of coordinates, we may assume these points are the coordinate points
$$
p_1:=[0:1:0:\dots:0],\dots,p_n:=[0:0:\dots:0:1]
$$
excluding $[1:0:\dots:0]$. The local equation of $X$ at each $p_i$ is
$$
f(x_0,\dots,x_{i-1},1,x_{i+1},\dots,x_n)=0.
$$
Denote by $q_i$, the quadratic term of this local equation at $p_i$. The vanishing locus $\{q_i=0\}\subset \P^n$ contains the singular locus of $X$, and thus, contains $Z$. Since $q_i$ has rank $\le 2$, the vanishing locus is the product of two hyperplanes, and one of the two contains $Z$. This implies that $e=n-1$ and $Z\subset\{x_0=0\}=\P^{n-1}$. Hence, every $q_i$ is divisible by $x_0$.

Express $f=x_0q+h$ where $h\in \C[x_1,\dots,x_n]$. The discussion above implies that the quadratic term of $h|_{x_i=1}$ should be zero for all $1\le i\le n$. This means $h=0$ and $f=x_0q$, which contradicts the GIT semistability of $X$.
\end{proof}

\begin{rmk}[Sharp bound for $n\le 5$]
\label{rmk:boundary cubics}
For a cubic hypersurface $X\subset\P^n$, we summarize below results from the literature on explicit GIT analyses, together with computations of minimal exponents (see Lemma \ref{lem:ADE and minimal expoenent}):

\noindent
(1) $n=3$: cubic surfaces \cite{Hilbert}.
\begin{center}
stable $\Longleftrightarrow$ at worst $A_1$-singularities $\Longleftrightarrow$ $\wt\alpha(X)>\frac{4}{3}$\\
semistable $\Longleftrightarrow$ at worst $A_1, A_2$-singularities $\Longleftrightarrow$ $\wt\alpha(X)\ge\frac{4}{3}$
\end{center}

\noindent
(2) $n=4$: cubic threefolds \cite{Allcock03, Yokoyama02}. Denote by $T$ the chordal cubic.
\begin{center}
stable $\Longleftrightarrow$ at worst $A_1, A_2, A_3, A_4$-singularities $\Longleftrightarrow$ $\wt\alpha(X)>\frac{5}{3}$\\
semistable $\nsim_{GIT} T$ $\Longleftrightarrow$ at worst $A_1, A_2, A_3, A_4, A_5, D_4$-singularities $\Longleftrightarrow$ $\wt\alpha(X)\ge\frac{5}{3}$
\end{center}
Here, ``semistable $\nsim_{GIT} T$" means ``semistable but not GIT equivalent to $T$." The chordal cubic $T$ is GIT polystable, and the minimal exponent is $\wt\alpha(T)=\frac{3}{2}$. Additionally, $T$ is not terminal.

\noindent
(3) $n=5$: cubic fourfolds \cite{Yokoyama08, Laza09, Laza10}. Denote by $\chi$, the one-parameter family of GIT polystable cubic fourfolds defined in \cite[Theorem 2.6]{Laza10}.
\begin{center}
stable, isolated singularities $\Longleftrightarrow$ simple $ADE$-singularities $\Longleftrightarrow$ $\wt\alpha(X)>2$\\
semistable $\nsim_{GIT} \chi$ $\Longrightarrow$ $\wt\alpha(X)\ge2$ $\Longrightarrow$ semistable
\end{center}
Here, ``semistable $\nsim_{GIT} \chi$" means ``semistable but not GIT equivalent to a cubic in $\chi$." The question of whether $\wt\alpha(X)\ge2$ implies GIT semistability was raised by Laza, and is answered here as a special case of Theorem \ref{thm:GIT stability via minimal exponent}. For the secant to the Veronese surface in $\P^5$, denoted $\omega\in \chi$, the minimal exponent is $\wt\alpha(\omega)=\frac{3}{2}$. For all other $X\in \chi\sm\{\omega\}$, the minimal exponent is $\wt\alpha(X)=\frac{11}{6}$. Additionally, $w$ is not terminal.

Determining the precise sharp bound for the minimal exponent in dimension five and higher remains an interesting open question.
\end{rmk}

\subsection{Extendability of period map to Baily-Borel compactification}

From the parameter space of degree $d$ hypersurfaces in $\P^n$, there exists a period map $\Phi_0$ to the period domain $\Gamma\backslash D$ for the primitive $\Z$-Hodge structure of middle cohomology, defined over the smooth locus:
$$
\Phi_0:\P^{\binom{n+d}{d}-1} \dashrightarrow \Gamma\backslash D.
$$
This descends to the map from the GIT moduli space
$$
\mathcal P_0:\overline\M^{\rm GIT} \dashrightarrow \Gamma\backslash D,
$$
and this provides a natural source to study the birational geometry of the moduli space when the generic Torelli theorem holds.

By Theorem \ref{thm:GIT stability via minimal exponent}, nodal hypersurfaces are GIT stable for $n\ge 3$, $d\ge 3$; the nodal singularity of dimension $n-1$ has the minimal exponent equal to $\frac{n}{2}$. Whether the period map $\mathcal P_0$ extends regularly over the locus of nodal hypersurfaces depends on the parity of $n$.

\begin{cor}
\label{cor:period map extension middle cohomology}
Let $(n,d)\neq(3,3)$, where $n\ge 3$ and $d\ge 3$. 

\noindent
(1) When $n$ is even, the GIT moduli space of smooth hypersurfaces is the domain of definition of the period map $\mathcal P_0$.

\noindent
(2) When $n$ is odd, the period map $\mathcal P_0$ extends regularly to the GIT moduli space of hypersurfaces with simple $ADE$-singularities
\end{cor}

We say that a germ of a hypersurface has \emph{simple $ADE$-singularities} if it is locally analytically isomorphic to the hypersurface
$$
\left\{f(x_1,x_2,x_3)+x_3^2+\cdots+x_{n}^2=0\right\}\subset \C^n
$$
near the origin, where $\{f(x_1,x_2,x_3)=0\}$ is a surface $ADE$-singularity. By Theorem \ref{thm:GIT stability via minimal exponent} and Lemma \ref{lem:ADE and minimal expoenent}, every hypersurface with simple $ADE$-singularities is GIT stable when $n\ge 5, d\ge 3$ or $n\ge 3, d\ge 4$.

\begin{lem}
\label{lem:ADE and minimal expoenent}
Let $x\in X$ be the germ of a hypersurface singularity of dimension $e\ge 2$. Then
$$
\wt\alpha_x(X)>\frac{e}{2}\Longleftrightarrow \text{simple } ADE\text{-singularity}.
$$
\end{lem}

\begin{proof}
We proceed by induction on $e$. When $e=2$, we have $\wt\alpha_x(X)>1$ if and only if $x\in X$ is a rational singularity by \cite{Saito93}, which is equivalent to an $ADE$-singularity.

Suppose the claim is true for $e\ge 2$. We prove when the dimension of $X$ is $e+1$. Since
$$
\wt\alpha_x(X)\le\frac{e+2}{\mathrm{mult}_x(X)}
$$
by \cite[Theorem E(3)]{MP20}, we have $\mathrm{mult}_x(X)=2$. In particular, this implies that the rank of the Hessian matrix is at least $1$, and thus, $x\in X$ is locally analytically isomorphic to the hypersurface
$$
\{F(x_1,\dots,x_{e+1})+x_{e+2}^2=0\}\subset\C^{e+2}
$$
near the origin. By Thom-Sebastiani Theorem \ref{thm:Thom-Sebastiani for minimal exponents}, this reduces to the case when the dimension of $X$ is $e$.
\end{proof}

\begin{proof}[Proof of Corollary \ref{cor:period map extension middle cohomology}]
Recall Griffiths' Removable Singularity Theorem (see \cite[Theorem 9.5]{Griffiths70} or \cite[Application 16]{GT84}): the period map $\mathcal P_0$ is regular over an open set $U\subset \overline\M^{\rm GIT}$ if and only if the local monodromy for $U_{\rm sm}$ around each point $[X]\in U\sm U_{\rm sm}$ is finite. Here, $U_{\rm sm}$ denotes the locus of smooth hypersurfaces. Additionally, for an arbitrary one-parameter smoothing $\X\to \Delta$ of a hypersurface $X\subset\P^n$, the local monodromy is finite if and only if the limit mixed Hodge structure $H^{n-1}(\X_\infty,\Q)$ is pure of weight $n-1$.

Suppose $n$ is even. Let $X\subset \P^n$ be a nodal hypersurface with exactly one node. It is well known that the limit mixed Hodge structure for any one-parameter smoothing of $X$ is never pure. Therefore, the period map $\mathcal P_0$ does not extend to the neighborhood of $[X]\in \overline\M^{\rm GIT}$. Furthermore, any singular hypersurface is a limit of nodal hypersurfaces with exactly one node. Indeed, for a hypersurface $\{f=0\}$ singular at $[1:0\dots:0]$, one may consider a degeneration $\{f+tg=0\}$ such that $\{g=0\}$ is a nodal hypersurface with exactly one node at $[1:0\dots:0]$. This proves (1).

Suppose $n$ is odd. By Lemma \ref{lem:ADE and minimal expoenent}, any hypersurface $X$ with simple $ADE$-singularities satisfies $\wt\alpha(X)>\frac{n-1}{2}$, and thus has $\frac{n-3}{2}$-rational singularities. In particular, $X$ is a rational homology manifold by \cite[Theorem A]{PP25a}, and $H^{n-1}(X,\Q)$ is pure of weight $n-1$. Consider any one-parameter smoothing $\X\to \Delta$ of $X$. By Theorem \ref{thm:cohomologically insignificant}, the cokernel of the specialization map
$$
\mathrm{sp}^{n-1}:H^{n-1}(X,\Q)\to H^{n-1}(\X_\infty,\Q)
$$
is a direct sum of the trivial Hodge structure $\Q^H(-\frac{n-1}{2})$ with a Tate twist. In particular, $H^{n-1}(\X_\infty,\Q)$ is pure, and thus, the local monodromy is finite. This proves (2).
\end{proof}

Next, we consider the period map for classical pairs $(n,d)\in\left\{(2,4),(2,6),(3,3),(4,3),(5,3)\right\}$. In proving the corollary, we disregard the lattice structure of the limit mixed Hodge structure, as this suffices for establishing the extension results in parts (1) and (2) (see Remark \ref{rmk:extension using Q-LMHS}).

\noindent
\textbf{Proof of Corollary \ref{cor:classical extension results}.}
In the classical cases, the literature works with suitable period domains so that the period map satisfies the global Torelli theorem over the smooth locus. For each case, we briefly describe the period map and give a streamlined proof of parts (1) and (2) of Corollary \ref{cor:classical extension results}. Except in the case of cubic fourfolds, this is a reinterpretation of known results from the literature using the minimal exponent.

For sextic plane curves $(n,d)=(2,6)$, Shah \cite{Shah80} considers the double cover of $\P^2$ branched along the sextic. This is a degree $2$ K3 surface, and thus one obtains a period map from the GIT moduli space to the Baily-Borel compactification of the period domain for degree $2$ K3 surfaces. For a plane sextic $C\subset \P^2$, denote by $S$ the double cover. Then $S$ has hypersurface singularities and
$$
\wt\alpha(C)>\frac{1}{2} \;(\mathrm{resp.}\ge) \Longleftrightarrow \wt\alpha(S)>1 \;(\mathrm{resp.}\ge).
$$
Indeed, if the local defining equation of $C$ is $f=0$, then the local defining equation of $S$ is $z^2-f=0$ for an independent variable $z$, and we apply Thom-Sebastiani Theorem \ref{thm:Thom-Sebastiani for minimal exponents}. If $\wt\alpha(C)>\frac{1}{2}$, then $S$ has $ADE$-singularities. Hence, for any one-parameter smoothing of $C$, the resulting one-parameter smoothing of $S$ induced by the double cover has locally finite monodromy. In contrast, if $\wt\alpha(C)=\frac{1}{2}$, then $S$ has Du Bois (but not rational) singularities. Hence, the local monodromy is not finite, but the direct sum of graded weight pieces of the limit mixed Hodge structure
$$
\bigoplus_w\gr^W_wH^2(\mathcal S_\infty,\Q)
$$
is independent of the one-parameter smoothing $\mathcal S\to \Delta$ of $S$. In terms of Theorem \ref{thm:core of liminal sources}, the core is invariant. Hence, parts (1) and (2) follow. Part (3) follows from \textit{loc. cit.}, which proves that the only GIT polystable sextic $C$ with $\wt\alpha(C)<\frac{1}{2}$ is the triple conic, corresponding precisely to the indeterminacy locus (see also \cite{Laza16}).

For quartic plane curves $(n,d)=(2,4)$, Kond\={o} \cite{Kondo00} and Artebani \cite{Artebani09} consider the quartic cover of $\P^2$ branched along the quartic. This is a quartic K3 surface, and thus one obtains a period map from the GIT moduli space to the Baily-Borel compactification of quartic K3 surfaces with a non-symplectic automorphism of order $4$. For a plane quartic $C\subset \P^2$, denote by $S$ the quartic cover. Then $S$ has hypersurface singularities and
$$
\wt\alpha(C)>\frac{3}{4} \;(\mathrm{resp.}\ge) \Longleftrightarrow \wt\alpha(S)>1 \;(\mathrm{resp.}\ge).
$$
As in the sextic case, this yields parts (1) and (2). Part (3) follows from \cite{Artebani09}, which proves that the indeterminacy locus is the double conic -- the only GIT polystable quartic $C$ with $\wt\alpha(C)<\frac{3}{4}$.

For cubic surfaces $(n,d)=(3,3)$, Allcock-Carlson-Toledo \cite{ACT02} considers the triple cover of $\P^3$ branched along the cubic. This is a cubic threefold, and one obtains a period map from the GIT moduli space to the Satake compactification of a quotient of the complex hyperbolic 4-space. Recall from Remark \ref{rmk:boundary cubics}(1) that the semistable locus in the parameter space is $U$ and the stable locus is $V$. Therefore, \cite[Theorem 3.17]{ACT02} yields all parts (1), (2), and (3); the indeterminacy locus is empty and the period map $\mathcal P$ is an isomorphism.

For cubic threefolds $(n,d)=(4,3)$, Allcock-Carlson-Toledo \cite{ACT11} and Looijenga-Swierstra \cite{LS07} consider the triple cover of $\P^4$ branched along the cubic, which is a cubic fourfold. This yields a period map from the GIT moduli space to the Baily-Borel compactification of a quotient of the $10$-dimensional subspace inside the $20$-dimensional Type IV bounded symmetric domain associated to cubic fourfolds. For a cubic threefold $X\subset \P^4$, denote by $Y$ the triple cover. Then
$$
\wt\alpha(X)>\frac{5}{3} \;(\mathrm{resp.}\ge) \Longleftrightarrow \wt\alpha(Y)>2 \;(\mathrm{resp.}\ge).
$$
By Theorem \ref{thm:core of liminal sources}, if $\wt\alpha(X)\ge\frac{5}{3}$, then the core of $H^4(Y,\Q)$ determines the core of the limit mixed Hodge structure $H^4(\mathcal Y_\infty,\Q)$ of any one-parameter smoothing $\mathcal Y\to \Delta$. Since the core for a K3-type limit mixed Hodge structure determines the direct sum of graded weight pieces, the direct sum
$$
\bigoplus_w\gr^W_wH^4(\mathcal Y_\infty,\Q)
$$
is independent of the smoothing. Theorem \ref{thm:core of liminal sources} additionally implies that $H^4(\mathcal Y_\infty,\Q)$ is pure of weight $4$ if and only if $Y$ has $1$-rational singularities (equivalently, $\wt\alpha(Y)>2$). Hence, parts (1) and (2) follow. Part (3) follows from \cite[Theorem 3.1]{LS07}, which proves $\mathcal P|_{\pi(V)}$ is an open embedding, and from \cite{ACT11}, which identifies the chordal cubic $T$ as the indeterminacy locus (see Remark \ref{rmk:boundary cubics}(2)).

For cubic fourfolds $(n,d)=(5,3)$, Looijenga \cite{Looijenga09} and Laza \cite{Laza10} studied the period map to the Baily-Borel compactification of a quotient of the $20$-dimensional Type IV bounded symmetric domain. Let $X\subset \P^5$ be a cubic fourfold. If $\wt\alpha(X)\ge 2$, then by Theorem \ref{thm:core of liminal sources}, the direct sum of weight graded pieces
$$
\bigoplus_w\gr^W_wH^4(\mathcal X_\infty,\Q)
$$
is independent of the one-parameter smoothing $\X\to \Delta$ of $X$. As in the case of cubic threefolds, $H^4(\mathcal X_\infty,\Q)$ is pure of weight $4$ if and only if $X$ has $1$-rational singularities. Hence, parts (1) and (2) follow. Part (3) follows from \cite[Theorem 4.1]{Looijenga09}, which proves $\mathcal P|_{\pi(V)}$ is an open embedding, and from \cite{Laza10}, which identifies the one-parameter family $\chi$ as the indeterminacy locus (see Remark \ref{rmk:boundary cubics}(3)).

\begin{rmk}
\label{rmk:extension using Q-LMHS}
In the proof above, we used the fact that the invariance of the graded weight pieces of the limit $\Q$-mixed Hodge structure at $X\subset \P^n$ is sufficient to extend the period map across $X$. Indeed, resolve the indeterminacy of
$$
\Phi:\P^{\binom{n+d}{d}-1} \dashrightarrow (\Gamma\backslash D)^*
$$
by a blow up $\mu:\wt\P\to \P^{\binom{n+d}{d}-1}$. If, for every one-parameter smoothing of $X$, the weight graded pieces of the limit $\Q$-mixed Hodge structure are invariant, then the image $\Phi(\mu^{-1}([X]))$ is supported on a countable subset of $(\Gamma\backslash D)^*$. It follows that $\Phi(\mu^{-1}([X]))$ must be a closed point, and hence $\Phi$ extends in a neighborhood of $[X]$.
\end{rmk}

Motivated by Theorem \ref{thm:GIT stability via minimal exponent}, Corollary \ref{cor:classical extension results}, and the recent construction of Bakker-Filipazzi-Mauri-Tsimerman on the Baily-Borel compactification of Calabi-Yau varieties \cite{BFMT25}, we propose the following question, as a generalized version of Conjecture \ref{conj:BB compactification for Calabi-Yau type}.

\begin{question}
\label{question:BB compactification and period map}
For which pairs $(n,d)$ with $n\ge 3$ and $d\ge 3$ does there exist a Hodge-theoretic compactification $M^{\rm BBH}$ of the GIT moduli space $M$ of degree $d$ hypersurfaces $X\subset\P^n$ with $\wt\alpha(X)>\frac{n+1}{d}$? Moreover, is the indeterminacy locus of the rational map
$$
\mathcal P:\overline\M^{\rm GIT} \dashrightarrow M^{\rm BBH}
$$
equal to the locus of GIT polystable hypersurfaces $X\subset \P^n$ with $\wt\alpha(X)<\frac{n+1}{d}$?
\end{question}

Corollary \ref{cor:classical extension results} answers this question for the classical pairs $\{(2,4),(2,6),(3,3),(4,3),(5,3)\}$. For quartic K3 surfaces (i.e. $n=3, d=4$), the Baily-Borel compactification exists, and the indeterminacy locus of the period map is predicted by Laza-O'Grady \cite{LOG18}. The K-moduli theoretic resolution of the period map in Ascher-DeVleming-Liu \cite{ADL23} should determine -- and appears to determine -- this indeterminacy locus, although we do not verify this here.

\begin{rmk}
Unlike Conjecture \ref{conj:BB compactification for Calabi-Yau type}, we do not specify the precise meaning of a ``Hodge-theoretic compactification," partly because there is no canonical choice of Hodge line bundle to begin with. For Calabi-Yau type hypersurfaces, however, one has a natural Hodge line bundle associated to the first nonzero graded piece of the Hodge filtration. Following \cite{BFMT25}, the existence of a compactification $M^{\rm BBH}$ amounts to proving the integrability and the torsion combinatorial monodromy conditions for this Hodge bundle. As seen from Corollary \ref{cor:classical extension results}, we expect Conjecture \ref{conj:BB compactification for Calabi-Yau type} to provide an answer to Question \ref{question:BB compactification and period map} by considering appropriate cycle covers, which reduces to the Calabi-Yau type case.  
\end{rmk}

\section{Hodge theory of Calabi-Yau type hypersurfaces}
\label{sec:Hodge theory of Calabi-Yau type hypersurfaces}

\subsection{Liminal centers and liminal sources}
\label{sec:liminal centers and liminal sources}

Liminal centers of hypersurface singularities satisfy analogous properties of log canonical centers as in Theorem \ref{thm:liminal centers}. In fact, they are obtained from general properties of mixed Hodge modules satisfying a certain condition: the first nonzero Hodge filtration is a line bundle. We start by defining generalized notions of liminal sources and liminal centers of a mixed Hodge module.

\begin{defn}
\label{defn:liminal sources and centers of MHM}
Let $X$ be a variety and $\K\in {\rm MHM}(X)$ be a mixed Hodge module. Let $m$ be the index of the first nonzero Hodge filtration of $\K$. A pure Hodge module $\M$ is a \emph{liminal source} of $\K$ if $\M$ is a simple subquotient of $\K$ such that $F_m\M\neq 0$. A \emph{liminal center} of $\K$ is the strict support $\Supp(\M)\subset X$ of a liminal source $\M$ of $\K$.
\end{defn}

Recall that the index of the first nonzero Hodge filtration of $\K$ is independent of the choice of local embedding of $X\hookrightarrow Y$ with $Y$ smooth and of the filtered (right) $\D_Y$-module presentation $(K,F_\bullet)$. Concretely, let
$$
m:=\min\left\{p:\gr^F_p\DR(\K)\neq0\right\}.
$$
Then the first nonzero Hodge filtration
$$
F_m\K:=\gr^F_m\DR(\K)
$$
is independent of the embedding (see \cite[Proposition 2.33]{Saito90}). When $F_m\K$ is a line bundle, we prove an analogous statement of Theorem \ref{thm:liminal centers} for liminal centers of $\K$.

\begin{prop}
\label{prop:liminal sources and centers of MHM}
In the setting of Definition \ref{defn:liminal sources and centers of MHM}, suppose $F_m\K$ is a line bundle on a closed subscheme $S\subset X$.
Then:

\noindent
(1) An intersection of two liminal centers of $\K$ is a union of liminal centers of $\K$.

\noindent
(2) There is a unique liminal source of $\K$ for each liminal center of $\K$.

\noindent 
(3) Any union of liminal centers of $\K$ has Du Bois singularities, and every minimal (with respect to inclusion) liminal center of $\K$ has rational singularities.

In particular, $S$ has Du Bois singularities.
\end{prop}

We begin with a lemma, necessary for the proof of this proposition.

\begin{lem}
\label{lem:reducedness of lowest Hodge filtration}
Let $\M\in {\rm MHM}(X)$ be a mixed Hodge module, and let $m$ be the index of the first nonzero Hodge filtration of $\M$. If $F_m\M\isom \O_S$ for some closed subscheme $S\subset X$, then $S$ is reduced.
\end{lem}

\begin{proof}
Since the assertions are Zariski local on $X$, we may assume $X$ is smooth by Kashiwara’s equivalence for mixed Hodge modules \cite[4.2.10]{Saito90}. Denote by $j:X\sm S\hookrightarrow X$ the open embedding. Consider the adjunction morphism of mixed Hodge modules \cite[4.4.1]{Saito90},
$$
\M\to \mathcal H^0(j_*j^*\M),
$$
and denote by $\N$ the kernel of this morphism. Passing to the Hodge filtration at level $m$, the map of $\O_X$-modules
$$
F_m\M\isom\O_S\to F_m\mathcal H^0(j_*j^*\M)
$$
is the zero map. Indeed, this map lifts to a morphism $\O_S\to j_*j^*\M$, which, by adjunction, is the zero map. As a consequence, we have
$$
F_m\M\isom \O_S\isom F_m\N.
$$
It is clear from the construction that $\Supp(\N)\subset S_{\rm red}$, where $S_{\rm red}$ is a reduced scheme of $S$. Therefore, $F_m\N$ is a $\O_{S_{\rm red}}$-module by \cite[Lemme 3.2.6]{Saito88}, and $S=S_{\rm red}$.
\end{proof}

\begin{proof}[Proof of Proposition \ref{prop:liminal sources and centers of MHM}]
After replacing $\K$ by $\Gamma_S(\K)$, we may assume $S=X$ (see Definition \ref{defn:two pullbacks of MHM}). Indeed, we have
$$
\Gamma_S(\K)\subset \K\quad\mathrm{and}\quad F_m\Gamma_S(\K)=F_m\K,
$$
and thus the set of liminal sources of $\Gamma_S(\K)$ is exactly the set of liminal sources of $\K$. Since the assertions are Zariski local on $X$, after shrinking we may further assume $F_m\K=\O_X$.

We argue by induction on $s$, the number of simple factors of $\K$. For the base case $s=1$, $\K$ is simple, so statements (1) and (2) are vacuous. In particular, $\K$ is the only liminal source with the only liminal center $X$. By \cite[Proposition 7.36]{SY23}, $X$ has rational singularities and statement (3) follows. For completeness: take a resolution $\mu:\wt X\to X$ and let $\wt\K$ be the simple Hodge module on $\wt X$ that agrees with $\K$ over the isomorphic locus. Then, Saito's Decomposition Theorem \cite[Théorème 5.3.1]{Saito88} yields a splitting
$$
\K\to \mu_*\wt\K\to \K,
$$
in $D^b{\rm MHM}(X)$ and passing to the Hodge filtration at level $m$ gives a splitting
$$
F_m\K\to R\mu_*F_m\wt\K\to F_m\K.
$$
This induces a left inverse of the natural morphism $\O_X\to R\mu_*\O_{\wt X}$. By Kovács' criterion \cite[Theorem 1]{Kovacs00}, $X$ has rational singularities.

Let $\K_0\in {\rm MHM}(X)$ be a simple factor of the first nonzero weight filtration of $\K$. Consider the following short exact sequence
$$
0\to \K_0\to \K\to \K/\K_0\to 0.
$$
Passing to the Hodge filtration at level $m$, we obtain
$$
0\to F_m\K_0\to F_m\K\to F_m(\K/\K_0)\to 0.
$$
Since $F_m\K=\O_X$, we have $F_m(\K/\K_0)=\O_Y$ for a closed subvariety $Y\subset X$ by Lemma \ref{lem:reducedness of lowest Hodge filtration}.

If $X=Y$, then the induction hypothesis applies to $\K/\K_0$, which completes the proof.

If $X\neq Y$, then $F_m\K_0=\I_{Y\subset X}$ is the ideal sheaf of $Y\subsetneq X$, and $\K_0$ is a simple Hodge module strict supported on the closure $Z:=\overline{X\sm Y}$. Every other liminal source of $\K$ is a liminal source of $\K/\K_0$ and is supported inside $Y$. Hence, the induction hypothesis on $\K/\K_0$ implies statement (2).

Next, we prove statements (1) and (3). Consider the resolution square
\begin{equation*}
\xymatrix{
{E}\ar[d]_{}\ar[r]^-{}& {\wt Z}\ar[d]^{\mu}&{}\\
{Z\cap Y}\ar[r]^-{}&{Z}\ar[r]&{X}
}
\end{equation*}
where $\mu$ is a resolution of singularities and $E:=\mu^{-1}(Z\cap Y)$. Let $W\subset Z$ be a closed subset containing $Z\cap Y$ such that $\K_0$ is a polarizable variation of Hodge structure on $Z\sm W$. The existence of $W$ follows from the structure theorem \cite[Theorem 3.21]{Saito90} of pure Hodge modules. Taking a further resolution, we may assume that $D:=\mathrm{Exc}(\mu)\cup\mu^{-1}(W)$, the union of the exceptional locus $\mathrm{Exc}(\mu)$ and $\mu^{-1}(W)$, is a simple normal crossing divisor.

Denote by $\tilde j:\wt Z\sm D\to \wt Z$ and $j: \wt Z\sm D\to Z$ open embeddings. Note that $j=\mu\circ\tilde j$.

Let $\mathcal V:=j^*\K_0$ be the polarizable variation of Hodge structure (as a filtered right D-module), and let $\widetilde \K_0:=\tilde j_{!*}\mathcal V$ be the pure Hodge module on $\wt Z$ associated to the minimal extension of $\mathcal V$. By \cite[Proposition 2.6]{Saito91}, we have
\begin{equation}
\label{eqn:pre-log rational}
R\mu_*F_m\wt\K_0\isom F_m\K_0=\I_{Y\subset X}.
\end{equation}
Recall that $D\subset \wt Z$ is a simple normal crossing divisor, so we have an inclusion of mixed Hodge modules and the inclusion of the first nonzero Hodge filtrations:
$$
\wt\K_0\subset \tilde j_*\mathcal V,\quad F_m\wt\K_0\subset F_m\tilde j_*\mathcal V.
$$
The latter is an inclusion of line bundles associated to Deligne's canonical extension with eigenvalues of residues in $(-1,0]$ and $[-1,0)$, respectively. Thus, we have
\begin{equation}
\label{eqn:comparison of lowest Hodge filtration}
F_m\tilde j_*\mathcal V(-D)\subset F_m\wt\K_0\subset F_m\tilde j_*\mathcal V.    
\end{equation}

Denote by $\L:=F_m\tilde j_*\mathcal V$ the line bundle. By the commutativity of the graded de Rham functor with proper pushforward, we have
$$
\mu_*\L=F_m\mathcal H^0(j_*\mathcal V).
$$
From the adjunction map $\K\to \mathcal H^0(j_*j^*\K)$, we have the map of the first nonzero Hodge filtrations
\begin{equation}
\label{eqn:O_X to mu_*L}
F_m\K=\O_X\to F_m\mathcal H^0(j_*j^*\K)=F_m\mathcal H^0(j_*\mathcal V)=\mu_*\L.    
\end{equation}
Here, we treat $j$ as the open embedding $\wt Z\sm D\subset X$. The equality $F_m\mathcal H^0(j_*j^*\K)=F_m\mathcal H^0(j_*\mathcal V)$ holds from the following distinguished triangle
$$
\gr^F_{m}\DR(j_*j^*\K_0)\to \gr^F_{m}\DR(j_*j^*\K) \to \gr^F_{m}\DR(j_*j^*(\K/\K_0))\xrightarrow{+1}
$$
and the vanishing $\gr^F_{m}\DR(j_*j^*(\K/\K_0))=0$, which follows from the vanishing 
$$
\gr^F_{\le m}\DR(j^*(\K/\K_0))=0
$$
and \cite[Lemma 3.4]{Park23}.

Note that the map \eqref{eqn:O_X to mu_*L} is an isomorphism on $\wt Z\sm D$, and $\mu_*\L$ is a torsion-free sheaf supported on $Z$. Denote by $\K_1$, the image of the adjunction map $\K\to \mathcal H^0(j_*j^*\K)$. Then the composition map
$$
\K_0\to\K\to \K_1
$$
has the associated map of the first nonzero Hodge filtrations
$$
\I_{Y\subset X}\to \O_X\to \O_Z,
$$
where the identification $F_m\K_1=\O_Z$ follows from Lemma \ref{lem:reducedness of lowest Hodge filtration}. Since $\K_0$ is simple, the map $\K_0\to \K_1$ is injective, and its quotient $\K_1/\K_0$ has the first nonzero Hodge filtration
$$
F_m(\K_1/\K_0)=\O_{Z\cap Y}.
$$
Therefore, $Z\cap Y$ is a union of liminal centers of $\K/\K_0$. Applying the induction hypothesis to $\K/\K_0$, we obtain statement (1).

Continuing with the map \eqref{eqn:O_X to mu_*L}, we have a section
$$
\O_{\wt Z}\to \L
$$
induced by adjunction. Denote by $D_0$ the associated effective Cartier divisor, so that $\L\isom\O_{\wt Z}(D_0)$. Recall that \eqref{eqn:O_X to mu_*L} is an isomorphism on $\wt Z\sm D$. This implies $\Supp(D_0)\subset D$. By \eqref{eqn:comparison of lowest Hodge filtration}, we obtain
$$
F_m\wt\K_0=\O_{\wt Z}(D_0-B)
$$
for some reduced normal crossing divisor $B\le D$.

We first prove that $-E\le D_0-B$. It is clear that $-D\le D_0-B$. Suppose a divisor $F$ with $\mu(F)\nsubseteq Z\cap Y$ had a negative coefficient in $D_0-B$. Then
$$
\left(\mu_*\O_{\wt Z}(D_0-B)\right)|_{Z\sm Y}\neq \O_{Z\sm Y},
$$
whereas $(\I_{Y\subset X})|_{Z\sm Y}=\O_{Z\sm Y}$. This contradicts \eqref{eqn:pre-log rational}. Hence, we have $-E\le D_0-B$.

As a consequence, we obtain a sequence of morphisms
$$
\I_{Y\subset X}\to R\mu_*\O_{\wt Z}(-E)\to R\mu_*\O_{\wt Z}(D_0-B)=R\mu_*F_m\wt\K_0=\I_{Y\subset X}
$$
This composition is an isomorphism of subsheaves in $\mu_*\L$. Note that $\I_{Y\subset X}=\I_{Z\cap Y\subset Z}$, induced by the exact sequence
$$
0\to F_m\K_0\to F_m\K_1\to F_m(\K_1/\K_0)\to 0 \;\;\Longleftrightarrow\;\; 0\to \I_{Y\subset X}\to \O_Z\to \O_{Z\cap Y}\to 0.
$$
Hence, we deduce from Proposition \ref{prop:criterion for log rational pair} that $(Z,Z\cap Y)$ is a log rational pair, that is, $Z\sm Y$ has rational singularities and $(Z,Z\cap Y)$ is a Du Bois pair. 

Applying the induction hypothesis to $\K/\K_0$, any union of liminal centers of $\K/\K_0$ has Du Bois singularities. Additionally, any union of $Z$ and liminal centers of $\K/\K_0$ has Du Bois singularities. Indeed, denote by $T$ a union of liminal centers of $\K/\K_0$. Note that
$$
Z\cup T=Z\cup ((Z\cap Y)\cup T),
$$
and $(Z\cap Y)\cup T$ is a union of liminal centers of $\K/\K_0$. From the basic property of the Du Bois pair (see, for example, \cite[Proposition 5.1]{Kovacs11}), $Z\cup T$ is Du Bois since $(Z,Z\cap Y)$ is a Du Bois pair and $(Z\cap Y)\cup T$ is Du Bois. Minimal liminal centers of $\K$ should also have rational singularities by the induction hypothesis. This completes the proof of statement (3).
\end{proof}

Applying Proposition \ref{prop:liminal sources and centers of MHM} to the dual of the RHM-defect object $\K_X^\bullet$, we deduce Theorem \ref{thm:liminal centers} as shown below.

\begin{proof}[Proof of Theorem \ref{thm:liminal centers}]
Dualizing \eqref{eqn:nonvanishing grDR-m} in Definition \ref{defn:liminal sources and centers}, we have
$$
\gr^F_m\DR(\dual \M)=F_m\dual \M\neq 0,
$$
and $\dual \M$ is a simple subquotient of $\dual \K_X^\bullet$. Note that $m$ is the index of the first nonzero Hodge filtration of $\dual \K_X^\bullet$, which implies the equality $\gr^F_m\DR(\dual \M)=F_m\dual \M$. It suffices to prove that $F_m\dual \K_X^\bullet$ is a line bundle supported precisely on the $m$-liminal locus of $X$ with reduced scheme structure. Proposition \ref{prop:liminal sources and centers of MHM} then applies, yielding Theorem \ref{thm:liminal centers}.

Since the assertions are local on $X$, we assume that $X$ is a Cartier divisor in a smooth variety $Y$ of dimension $n$. From the assumption that $X$ has $m$-Du Bois singularities, we have $\Omega_X^m=\DB_X^m$. Applying the graded de Rham functor $\gr^F_{-m}\DR(\,\cdot\,)$ to \eqref{eqn:Q to IC triangle}, we have the distinguished triangle
$$
\gr^F_{-m}\DR(\K_X^\bullet)\to \Omega_X^m[n-1-m]\to I\DB_X^m[n-1-m]\xrightarrow{+1}.
$$
Applying the Grothendieck duality $R\Hom_{\O_X}(\,\cdot\,,\w_X[n-1])$ with the duality formula (see Proposition \ref{prop:duality}), we have
$$
I\DB_X^{n-1-m}[m]\to R\Hom_{\O_X}(\Omega_X^m,\w_X)[m]\to F_m\dual\K_X^\bullet\xrightarrow{+1}.
$$
and the associated long exact sequence of $\O_X$-modules:
$$
\cdots\to \mathcal Ext^m_{\O_X}(\Omega_X^m,\w_X)\to F_m\dual\K_X^\bullet\to \mathcal H^{m+1}(I\DB_X^{n-1-m})\to\cdots.
$$
Since a graded de Rham complex of a Hodge module lives in nonpositive degrees, we have the vanishing $\mathcal H^{m+1}(I\DB_X^{n-1-m})=0$. This implies the surjection
$$
\mathcal Ext^m_{\O_X}(\Omega_X^m,\w_X)\to F_m\dual\K_X^\bullet
$$
of $\O_X$-modules.

On the other hand, by Lemma \ref{lem:Koszul resolution of differentials}, the Koszul complex
$$
K_m^\bullet(\phi):\O_Y(-mX)|_X\to\Omega_Y(-(m-1)X)|_X \to\cdots\to\Omega_Y^m|_X,
$$
is a locally free resolution of $\Omega_X^m$. This implies that
$$
\mathcal Ext^m_{\O_X}(\Omega_X^m,\w_X)\isom\mathrm{coker}\left(\Hom_{\O_X}(\Omega_Y(-(m-1)X)|_X,\w_X)\to \w_X(mX)\right).
$$
In particular, we have a surjection
$$
\w_X(mX)\to F_m\dual\K_X^\bullet.
$$
Note that $F_m\dual\K_X^\bullet$ is an $\O_X$-module precisely supported on the $m$-liminal locus of $X$ (see the discussion following Definition \ref{defn:RHM defect object}). Therefore, by Lemma \ref{lem:reducedness of lowest Hodge filtration}, $F_m\dual\K_X^\bullet$ is a line bundle $\omega_X(mX)|_S$ where $S$ is the $m$-liminal locus of $X$ with reduced scheme structure.
\end{proof}

Note that the $m$-liminal locus, phrased in terms of minimal exponents, is the set
$$
\left\{x\in X:\wt\alpha_x(X)=m+1\right\}.
$$
The last statement of Theorem \ref{thm:liminal centers} says that if $\wt\alpha(X)=m+1\in \Z$, then the above set has Du Bois singularities. Using the Thom-Sebastiani theorem, we prove an analogous statement even when the global minimal exponent is not an integer.

\begin{cor}
Let $X$ be a variety with hypersurface singularities. Then the (reduced) locus where the local and global minimal exponents agree,
$$
\left\{x\in X:\wt\alpha_x(X)=\wt\alpha(X)\right\},
$$
has Du Bois singularities.
\end{cor}

\begin{proof}
Since the assertion is local on $X$, we assume that $X$ is a hypersurface defined by a local equation $f=0$ on a smooth variety $Y$. Let $N$ be a sufficiently divisible positive integer such that $N\wt\alpha(X)\in \Z$. Then there exists a positive integer $e$ such that
$$
\wt\alpha(X)+\frac{e}{N}=m+1\in \Z.
$$
Consider the global function
$$
F:=f\oplus (y_1^N+\dots+y_e^N)\in \O_{Y\times\C^e}(Y\times \C^e)
$$
and $W\subset Y\times \C^e$ be the hypersurface $\{F=0\}$. Note that by the Jacobian criterion, we have
$$
\Sing(W)=\Sing(X)\times\{0\}.
$$
By Thom-Sebastiani Theorem \ref{thm:Thom-Sebastiani for minimal exponents}, the global minimal exponent $\wt\alpha(W)=m+1$, and
$$
\left\{w\in W:\wt\alpha_w(W)=m+1\right\}=\left\{x\in X:\wt\alpha_x(X)=\wt\alpha(X)\right\}\times \{0\}.
$$
The formal is Du Bois by Theorem \ref{thm:liminal centers}, so the latter is also Du Bois.
\end{proof}

\subsection{Cores of Calabi-Yau type Hodge structures}

From now on, we focus on understanding the core of the middle cohomology and Hodge-Du Bois numbers of a Calabi-Yau type hypersurface. Recall from Introduction that a mixed Hodge structure $H=(V_\Q,F^\bullet,W_\bullet)$ is of Calabi-Yau type if $F^m V_\C=V_\C$ and $\dim \gr_F^mV_\C=1$ for some $m$. For instance, the middle cohomology of a smooth Calabi-Yau type hypersurface is a pure Hodge structure of Calabi-Yau type. Another instance occurs when the first nonzero Hodge filtration of a mixed Hodge module is a structure sheaf of a reduced variety:

\begin{lem}
\label{lem:MHS of Calabi-Yau type from MHM}
Let $X$ be projective and $\M\in {\rm MHM}(X)$. Let $m$ be the index of the first nonzero Hodge filtration of $\dual\M$. If
$$
F_m\dual\M=\gr^F_m\DR(\dual\M)\isom \O_S
$$
for a reduced connected closed subscheme $S\subset X$, then the hypercohomology $\HH^0(\M)$ is a mixed Hodge structure of Calabi-Yau type.
\end{lem}

\begin{proof}
Denote by $a_X:X\to \mathrm{pt}$ the constant map to a point. By definition, 
$$
\HH^0(\M)=H^0({a_X}_*\M).
$$
Since $X$ is projective, we have the duality
$$
\HH^0(\M)\isom {\rm Hom}(\HH^0(\dual \M),\Q^H)
$$
by $\dual\circ{a_X}_*={a_X}_*\circ\dual$ \cite[4.3.5]{Saito90}. Additionally, we have 
$$
\gr^F_{<m}\HH^0(\dual \M)=0\quad \mathrm{and}\quad \gr^F_m\HH^0(\dual\M)=H^0(X,\O_S)=\C
$$
from Saito's Hodge-to-de Rham spectral sequence. Consequently, $\HH^0(\M)$ is of Calabi-Yau type.
\end{proof}

We will see later in Proposition \ref{prop:lowest Hodge filtration of dual K} that $F_m\dual\K_X^\bullet\isom\O_S$ for a $m$-liminal locus $S$ of a Calabi-Yau type hypersurface $X\subset \P^n$ of degree $d$. Although $S$ is not necessarily connected, we may restrict $\K_X^\bullet$ to each irreducible component of $S$ and apply Lemma \ref{lem:MHS of Calabi-Yau type from MHM}. To begin with, we denote two inverse image functors by $\Gamma_S(\,\cdot\,)$ and $\,\cdot\,|_S$ for a reduced closed subscheme $S\subset X$.

\begin{defn}
\label{defn:two pullbacks of MHM}
Let $\M\in {\rm MHM}(X)$ be a mixed Hodge module on a variety $X$ and $\iota:S\hookrightarrow X$ a closed embedding. We define
$$
\Gamma_S(\M):=\mathcal H^0(\iota^!\M)\quad\mathrm{and}\quad \M|_S:=\mathcal H^0(\iota^*\M),
$$
the mixed Hodge modules supported on $S$.
\end{defn}

As a D-module on a smooth variety, $\Gamma_S(\M)$ is a submodule of $\M$ consisting of sections with support in $S$. The adjunction morphism $\Gamma_S(\M)\to \M$ is injective, and by duality, $\M\to \M|_S$ is surjective. Note that we have an isomorphism
\begin{equation}
\label{eqn:dual restriction=restriction dual}
\dual(\M|_S)\isom \Gamma_S(\dual \M).    
\end{equation}
Next, we study liminal sources and the associated cores in the situation where $F_m\dual\K$ is a structure sheaf of a reduced variety.

\begin{prop}
\label{prop:core of liminal source of MHM}
Let $X$ be a projective variety and $\K\in {\rm MHM}(X)$ a mixed Hodge module. Let $m$ be the index of the first nonzero Hodge filtration of $\dual \K$. Assume $F_m\dual\K\isom \O_S$ for a reduced closed subscheme $S\subset X$. For each connected component $S_0\subset S$, the mixed Hodge modules $\Gamma_{S_0}(\dual\K)$ and $(\dual\K)|_{S_0}$ satisfies
$$
F_m(\Gamma_{S_0}(\dual\K))\isom F_m((\dual\K)|_{S_0})\isom\O_{S_0}.
$$
Moreover, for any liminal source $\dual \M$ of $\dual \K$ supported on a minimal liminal center $Z\subset S_0$, we have isomorphisms
$$
F_m\dual\M\isom \O_Z\quad\mathrm{and}\quad \Core(\HH^0(\K|_{S_0}))\isom\Core(\HH^0(\M)).
$$
\end{prop}

\begin{proof}
Denote by $j_0:X\sm S_0\to X$ the open embedding. Then the cokernel of the inclusion
$$
\Gamma_{S_0}(\dual\K)\to \dual\K
$$
is a submodule of $\mathcal H^0({j_0}_*j_0^*\dual\K)$ (see e.g. \cite[4.4.1]{Saito90}). Since any section of $\mathcal H^0({j_0}_*j_0^*\dual\K)$ considered as a D-module on the smooth ambient projective space is not supported in $S_0$, the induced map of the first nonzero Hodge filtration
$$
F_m\Gamma_{S_0}(\dual\K)\to F_m\dual\K\isom\O_S
$$
is an isomorphism near $S_0$. Hence, $F_m\Gamma_{S_0}(\dual\K)\isom \O_{S_0}$. The isomorphism $F_m((\dual\K)|_{S_0})\isom\O_{S_0}$ also follows from the dual argument for the surjection
$$
\O_S\isom F_m\dual\K\to F_m((\dual\K)|_{S_0}).
$$

Consider the inclusion
$$
\Gamma_{S}(\dual\K)\to \dual\K.
$$
As above, this induces an isomorphism of the first nonzero Hodge filtrations. The set of liminal sources of $\Gamma_{S}(\dual\K)$ is exactly the set of liminal sources of $\dual \K$. Hence, a liminal source $\dual\M$ of $\dual\K$, as in the statement of the proposition, is a liminal source of $\Gamma_{S_0}(\dual\K)$. By Lemma \ref{lem:MHS of Calabi-Yau type from MHM} and \eqref{eqn:dual restriction=restriction dual}, the mixed Hodge structure $\HH^0(\K|_{S_0})$ is of Calabi-Yau type.

Without loss of generality, we may now assume $S=S_0$ is connected and $\dual\K=\Gamma_{S}(\dual\K)$, that is, $\K$ is supported on $S$.
To prove the last two isomorphisms in the statement, we proceed by induction on $s$, the number of simple factors of $\K$. When $s=1$, the statements are vacuous. As in the proof of Proposition \ref{prop:liminal sources and centers of MHM}, consider $\N\in {\rm MHM}(X)$ a simple factor of $\K$ that fits into a short exact sequence
$$
0\to \dual\N\to \dual\K\to \dual\K'\to 0.
$$
Passing to the Hodge filtration at level $m$, we have
$$
0\to F_m\dual\N\to \O_S\to\O_T\to 0,
$$
for a closed subvariety $T\subset S$. If $S=T$, then the induction hypothesis applies to $\K'$, and
$$
\Core(\HH^0(\K))\isom\Core(\HH^0(\K'))
$$
by Lemma \ref{lem:isomorphism of cores of MHS} below; note that the morphism $\K'\to\K$ induces $\HH^0(\K')\to \HH^0(\K)$, and this induces an isomorphism of top Hodge pieces. This completes the proof.

If $S\neq T$, then the strict support of $\dual\N$ is an irreducible component of $S$ that intersects with $T$. By Proposition \ref{prop:liminal sources and centers of MHM}, this implies that the strict support of $\dual\N$ is not a minimal liminal center, and thus, $\M\neq \N$. Hence, $\dual\M$ is a liminal source of $\dual \K'$. 

When $T$ is connected, we apply the induction hypothesis to $\K'$. The map
$$
\gr_F^m\HH^0(\K')\to\gr_F^m\HH^0(\K)
$$
is dual to the map $H^0(\O_S)\to H^0(\O_T)$, and thus an isomorphism. In particular, 
$$
\Core(\HH^0(\K))\isom\Core(\HH^0(\K'))
$$
by Lemma \ref{lem:isomorphism of cores of MHS} below, and this completes the proof.

When $T$ is not connected, suppose $Z\subset T_0$ for a connected component $T_0\subset T$. Then $\dual\M$ is a liminal source of the further quotient $\dual \K'/\Gamma_{T-T_0}(\dual \K')$. Applying the induction hypothesis to the dual of this further quotient, we likewise complete the proof.
\end{proof}

In the course of the proof, we used the following general result about morphisms of mixed Hodge structures of Calabi-Yau type and their cores:

\begin{lem}
\label{lem:isomorphism of cores of MHS}
Let $H\to H'$ be a morphism of mixed Hodge structures $H=(V_\Q,F^\bullet,W_\bullet)$ and $H=(V'_\Q,F^\bullet,W_\bullet)$ of Calabi-Yau type, such that the induced map
$$
\gr^m_FV_\C\to\gr^m_FV_\C' 
$$
is an isomorphism of one-dimensional top Hodge pieces. Then $\Core(H)\isom\Core(H')$.
\end{lem}

\begin{proof}
This is immediate from the strictness of morphisms of mixed Hodge structures (see e.g. \cite[Corollary 3.6]{PS08}).
\end{proof}

For the RHM-defect object of a Calabi-Yau type $m$-Du Bois hypersurface, we verify that the first nonzero Hodge filtration of its dual is the structure sheaf of the $m$-liminal locus.

\begin{prop}
\label{prop:lowest Hodge filtration of dual K}
Let $X\subset \P^n$ be an $m$-Du Bois hypersurface of degree $d$, with $\frac{n+1}{d}=m+1$. Then there exists a surjective composition map
$$
\O_{X}\to R\Hom_{\O_X}(\Omega_X^m,\w_X)[m]\to F_m\dual\K_X^\bullet,
$$
and $F_m\dual\K_X^\bullet\isom \O_S$ where $S$ is the $m$-liminal locus of $X$.
\end{prop}

\begin{proof}
We first construct the composition map. Applying the functor $\gr^F_m\DR(\,\cdot\,)$ to the morphism $\dual(\Q_X^H[n-1])\to \dual\K_X^\bullet$, we obtain
$$
R\Hom_{\O_X}(\Omega_X^m,\w_X)[m]\to F_m\dual\K_X^\bullet.
$$
By Lemma \ref{lem:Koszul resolution of differentials}, $R\Hom_{\O_X}(\Omega_X^m,\w_X)[m]$ is represented by the complex of vector bundles
$$
\Hom_{\O_X}(\Omega_{\P^n}^m|_X,\w_X)\to \dots\to\Hom_{\O_X}(\O_{\P^n}(-mX)|_X,\w_X).
$$
with the right-most term supported on degree $0$. By adjunction, we have
$$
\Hom_{\O_X}(\O_{\P^n}(-mX)|_X,\w_X)\isom\omega_{\P^n}((m+1)X)|_X\isom\O_X(-n-1+(m+1)d),
$$
which is isomorphic to $\O_X$. Hence, we have the composition map.

Recall from the proof of Theorem \ref{thm:liminal centers} that we have a surjection
$$
\mathcal Ext^m_{\O_X}(\Omega_X^m,\w_X)\twoheadrightarrow F_m\dual\K_X^\bullet.
$$
Therefore, the composition $\O_X\to F_m\dual\K_X^\bullet$ is a surjection. The set-theoretic support of $F_m\dual\K_X^\bullet$ is the $m$-liminal locus $S$ of $X$. Using Lemma \ref{lem:reducedness of lowest Hodge filtration}, we obtain $F_m\dual\K_X^\bullet\isom \O_S$.
\end{proof}

Combining Propositions \ref{prop:core of liminal source of MHM} and \ref{prop:lowest Hodge filtration of dual K}, we deduce the following

\begin{cor}
\label{cor:core of liminal center}
In the setting of Proposition \ref{prop:lowest Hodge filtration of dual K}, the natural map
$$
\gr_F^m\HH^0(\K_X^\bullet)\to \gr_F^mH^{n-1}(X,\C)
$$
of Hodge pieces induced by $\K_X^\bullet\to \Q_X^H[n-1]$ is surjective, and for any $m$-liminal source $\M$ supported on a minimal $m$-liminal center of $X$, we have
$$
\Core(H^{n-1}(X,\Q))\isom\Core(\HH^0(\M)).
$$
\end{cor}

\begin{proof}
Recall that $\dim \gr_F^mH^{n-1}(X,\C)=1$. This follows from either the constancy of Hodge-Du Bois numbers in families with $m$-Du Bois singularities \cite{FL24} or a direct spectral sequence computation.

Taking (hyper)cohomologies of the composition in Proposition \ref{prop:lowest Hodge filtration of dual K}, we have
$$
H^0(\O_X)\to \mathrm{Ext}^m(\Omega_X^m,\w_X)\to H^0(\O_S).
$$
The second map is the dual of the map $\gr_F^m\HH^0(\K_X^\bullet)\to \gr_F^mH^{n-1}(X,\C)$. Since the above composition is a nonzero map, the natural map of our interest is nonzero, hence surjective.

For an $m$-liminal source $\M$ of $X$ supported on a minimal $m$-liminal center $Z\subset X$, assume that $Z\subset S_0$ where $S_0$ is a connected component of the $m$-liminal locus $S$. Consider a composition
$$
\Gamma_{S_0}(\K_X^\bullet)\to \K_X^\bullet\to \Q_X^H[n-1].
$$
The natural map
$$
\gr_F^m\HH^0(\Gamma_{S_0}(\K_X^\bullet))\to \gr_F^mH^{n-1}(X,\C)
$$
is an isomorphism, since Proposition \ref{prop:core of liminal source of MHM} implies that the dual of this map is the composition
$$
H^0(\O_X)\isom\mathrm{Ext}^m(\Omega_X^m,\w_X)\to H^0(\O_S)\to H^0(\O_{S_0})=\C
$$
via Saito's Hodge-to-de Rham spectral sequence. Therefore, by Lemma \ref{lem:isomorphism of cores of MHS}, we have
$$
\Core(H^{n-1}(X,\Q))\isom \Core(\HH^0(\Gamma_{S_0}(\K_X^\bullet))),
$$
and the latter is isomorphic to $\Core(\HH^0(\K_X^\bullet|_{S_0}))\isom\Core(\HH^0(\M))$ by Proposition \ref{prop:core of liminal source of MHM}.
\end{proof}

This corollary establishes the first equality of the cores in Theorem \ref{thm:core of liminal sources}; the rest follows from Theorem \ref{thm:cohomologically insignificant} in the subsequent section on limit mixed Hodge structures.

\subsection{Limit mixed Hodge structures of degenerations}

A flat projective morphism $f:\X\to \Delta$ over the unit disk $\Delta$ is called a \emph{one-parameter degeneration}. If the generic fiber of $f$ is smooth, then we call $f$ a \emph{one-parameter smoothing} of the special fiber $X:=f^{-1}(0)$. Every one-parameter degeneration has the associated limit mixed Hodge structure $H^\bullet(\X_\infty,\Q)$ (see \cite{Schmid73, Steenbrink75} for one-parameter smoothing and \cite{SZ85, Saito90} in general) with the specialization map
$$
\sp^\bullet: H^\bullet(X,\Q)\to H^\bullet(\X_\infty,\Q)
$$
of mixed Hodge structures. Additionally, the limit mixed Hodge structure $H^\bullet(\X_\infty,\Q)$ has a quasi-unipotent monodromy operator $T$, and the associated nilpotent operator $N$ (of some sufficiently divisible power of $T$).

From now on, we assume that $f$ extends to a projective morphism between quasi-projective varieties, which allows flexible use of Saito’s theory of mixed Hodge modules \cite{Saito90}. This reduction is sufficient for the proofs of Theorem \ref{thm:core of liminal sources} and Corollary \ref{cor:maximal degeneration}; see Remark \ref{rmk:one parameter degeneration}.

By Shah \cite{Shah79} for surfaces, and later by Dolgachev \cite{Dolgachev81} and Steenbrink \cite{Steenbrink81} for higher-dimensional varieties, it was shown that any one-parameter smoothing $f:\X\to \Delta$, whose special fiber $X$ has Du Bois singularities, is \emph{cohomologically insignificant}, in the sense that
$$
\gr_F^0(\sp^\bullet): \gr_F^0H^\bullet(X,\C)\to \gr_F^0H^\bullet(\X_\infty,\C)
$$
is an isomorphism. For one-parameter degenerations whose generic fiber has higher rational singularities and special fiber has higher Du Bois singularities, we establish a form of \emph{higher cohomological insignificance}, with particular focus on degenerations of Calabi-Yau type hypersurfaces.

\begin{thm}
\label{thm:cohomologically insignificant}
Let $f:\X\to \Delta$ be a projective one-parameter degeneration with special fiber $X$. If $X$ has $m$-Du Bois lci singularities and the generic fiber has $m$-rational lci singularities, then
$$
\gr_F^p(\sp^k): \gr_F^pH^k(X,\C)\to \gr_F^pH^k(\X_\infty,\C)
$$
is an isomorphism for all $p\le m$ and $k\in \Z$. In particular, if $X$ is an $m$-Du Bois degree $d$ hypersurface in $\P^n$ with $\frac{n+1}{d}=m+1$, then
$$
\Core\left(H^{n-1}(X,\Q)\right)\isom\Core\left(H^{n-1}(\X_\infty,\Q)\right).
$$
\end{thm}

When $f$ is a one-parameter smoothing, Acu\~ na-Kerr \cite[Theorem 3]{AK25} proved the statement by applying results from \cite{FL24} after passing to a compactification of $\X$ with $m$-rational singularities. When $f$ is not a one-parameter smoothing, such a compactification need not exist. We prove a generalized version under condition $D_m$ that applies to families with non-lci singularities.

\begin{prop}
\label{prop:limit MHS via Dm}
Let $f:\X\to \Delta$ be a projective one-parameter degeneration with a reduced special fiber $X$ of pure dimension $r$. If $X$ satisfies $D_{m-1}$ and $\X$ satisfies $D_m$, then the natural morphism
$$
\Q_X^H[r]\to \psi_{f,1}(\Q_\X^H[r+1])
$$
induces isomorphisms
$$
\gr^F_p\DR(\Q_X^H[r])\to \gr^F_p\DR(\psi_{f,1}(\Q_\X^H[r+1])) \quad \text{for all}\quad p\ge -m.
$$
\end{prop}

Following \cite[Section 2.2]{Saito90}, the nearby and vanishing cycle functors
$$
\psi_f: {\rm MHM}(\X)\to {\rm MHM}(X),\quad \phi_{f,1}:{\rm MHM}(\X)\to {\rm MHM}(X)
$$
are defined for mixed Hodge modules, which extend to the functors on $D^b{\rm MHM}(\X)$ and underlie the usual nearby and vanishing functors for regular holonomic D-modules and perverse sheaves \cite[Theorem 0.1]{Saito90}. The limit mixed Hodge structure is isomorphic to the hypercohomology of $\psi_{f}(\Q_\X^H[r+1])$:
$$
\HH^{k-r}(\psi_{f}(\Q_\X^H[r+1]))\isom H^{k}(\X_{\infty},\Q).
$$
The functor $\psi_{f,1}$ is the unipotent summand of $\psi_f$ with respect to the monodromy operator.

For convenience, for a mixed Hodge module $\M$, we denote by
$$
p(\M):=\min\{p: F_p\M\neq 0\},
$$
the index of the first nonzero Hodge filtration of $\M$.

\begin{lem}
\label{lem:grDR vanishing for nearby and vanishing}
Let $f:\X\to \Delta$ be a projective one-parameter degeneration and $\M\in D^b{\rm MHM}(\X)$ satisfying
$$
\gr^F_p\DR(\M)=0\quad\text{for all}\quad p\ge k
$$
for some integer $k$. Then,
$$
\gr^F_p\DR(\psi_{f}(\M))=\gr^F_p\DR(\phi_{f,1}(\M))=0\quad\text{for all}\quad p\ge k.
$$
\end{lem}

\begin{proof}
We may assume $\M\in {\rm MHM}(\X)$. Indeed, the assumption is equivalent to
$$
\gr^F_p\DR(\dual\M)=0\quad\text{for all}\quad p\le -k
$$
by Proposition \ref{prop:duality}, which is equivalent to $p(\mathcal H^i(\dual \M))\ge -k+1$ for all $i$. Since $\dual\mathcal H^{-i}(\M)\isom \mathcal H^i(\dual\M)$, this reduces to proving the assertion for each cohomology module of $\M$.

It suffices to prove $p(\dual \psi_{f}(\M))\ge -k+1$ and $p(\dual\phi_{f,1}(\M))\ge -k+1$. By \cite[Proposition 2.6]{Saito90}, they are equivalent to
$$
p(\psi_{f}(\dual\M))\ge -k+2\quad\text{and}\quad p(\phi_{f,1}(\dual\M))\ge -k+1.
$$
They are immediate from $p(\dual\M)\ge -k+1$, via the definition \cite[(2.2.6)]{Saito90} of the Hodge filtration on the underling D-module of $\psi_{f}(\dual\M)$ and $\phi_{f,1}(\dual\M)$.
\end{proof}

Denote by $\iota:X\hookrightarrow \X$ the closed embedding. By \cite[Corollary 2.24]{Saito90}, we have the distinguished triangle
$$
\iota^*\M[-1]\to \psi_{f,1}(\M)\xrightarrow{\rm can} \phi_{f,1}(\M)\xrightarrow{+1}.
$$
As a consequence, under the same assumption of Lemma \ref{lem:grDR vanishing for nearby and vanishing}, we have the vanishing
\begin{equation}
\label{eqn:pullback grDR vanishing}
\gr^F_p\DR(\iota^*\M)=0\quad\text{for all}\quad p\ge k .
\end{equation}

\begin{proof}[Proof of Proposition \ref{prop:limit MHS via Dm}]
Consider the distinguished triangle \eqref{eqn:Q to IC triangle}:
$$
\K_\X^\bullet\to \Q_\X^H[r+1]\xrightarrow{\gamma_\X}\IC_\X^H\xrightarrow{+1}.
$$
Applying the vanishing cycle functor, we have
$$
\phi_{f,1}(\K_\X^\bullet)\to \phi_{f,1}(\Q_\X^H[r+1])\to\phi_{f,1}(\IC_\X^H)\xrightarrow{+1}.
$$
Condition $D_m$ on $\X$ implies the vanishing
$$
\gr^F_p\DR(\K_\X^\bullet)=0\quad\text{for all}\quad p\ge -m,
$$
which implies
$$
\gr^F_p\DR(\phi_{f,1}(\K_\X^\bullet))=0\quad\text{for all}\quad p\ge -m
$$
by Lemma \ref{lem:grDR vanishing for nearby and vanishing}. Note that the assertion of the proposition is equivalent to the vanishing
$$
\gr^F_p\DR(\phi_{f,1}(\Q_\X^H[r+1]))=0\quad\text{for all}\quad p\ge -m.
$$
Hence, it suffices to prove
$$
\gr^F_p\DR(\phi_{f,1}(\IC_\X^H))=0\quad\text{for all}\quad p\ge -m,
$$
or equivalently, $p(\dual \phi_{f,1}(\IC_\X^H))\ge m+1$ by Proposition \ref{prop:duality}.

By \cite[Corollary 2.24]{Saito90}, we have the distinguished triangles
\begin{gather*}
\phi_{f,1}(\dual\IC_\X^H)\xrightarrow{\rm Var} \psi_{f,1}(\dual\IC_\X^H)(-1)\to \iota^!(\dual\IC_\X^H)[1]\xrightarrow{+1},\\
\iota^*(\dual\IC_\X^H)[-1]\to\psi_{f,1}(\dual\IC_\X^H)\xrightarrow{\rm can} \phi_{f,1}(\dual\IC_\X^H)\xrightarrow{+1}.
\end{gather*}
Since $\dual\IC_\X^H$ does not admit a nontrivial subobject or quotient object supported inside $f^{-1}(0)$, $\rm can$ is surjective and $\rm Var$ is injective (see \cite[Lemme 5.1.4]{Saito88}). In particular,
$$
\phi_{f,1}(\dual\IC_\X^H)={\rm image}(N:\psi_{f,1}(\dual\IC_\X^H)\to \psi_{f,1}(\dual\IC_\X^H)(-1))
$$
where the nilpotent operator $N:=\rm Var\circ can$. Note that $\dual\IC_\X^H$ is a pure Hodge module of weight $-r-1$, so we have the isomorphisms
$$
N^j:\gr^W_{-r-2+j}\psi_{f,1}(\dual\IC_\X^H)\xrightarrow{\sim}\left(\gr^W_{-r-2-j}\psi_{f,1}(\dual\IC_\X^H)\right)(-j) \quad\text{for all}\quad j>0.
$$
(See \cite[Section 5.1.6]{Saito88}.)

Let $\M$ be any simple subquotient of $\dual \phi_{f,1}(\IC_\X^H)\isom\phi_{f,1}(\dual\IC_\X^H)$. This can be considered as a subquotient of $\psi_{f,1}(\dual\IC_\X^H)(-1)$. Then the above isomorphism of $N^j$ implies that $\M(-k)$ is a simple subquotient of
$$
\left(\iota^!(\dual\IC_\X^H)[1]\right)/W_{-r}\left(\iota^!(\dual\IC_\X^H)[1]\right).
$$
for some $k>0$. Indeed, we may consider the Lefschetz decomposition of
$$
\bigoplus_{j\in \Z}\gr^W_{-r+j}\left(\psi_{f,1}(\dual\IC_\X^H)(-1)\right)
$$
with respect to the nilpotent operator $N$. Note that by the Tate twist $(-1)$, the decomposition is centered at $-r$. Consequently, there should exist a primitive element $\M(-k)$ of weight $\ge -r+1$, not in the image of $N$. Therefore, this is a subquotient of
$$
\G:=\left(\iota^!(\dual\IC_\X^H)[1]\right)/W_{-r}\left(\iota^!(\dual\IC_\X^H)[1]\right).
$$
Since we need to prove that $p(\M)\ge m+1$, the above argument reduces to proving that $p(\G)\ge m$.

By $\iota^!\circ\dual\isom\dual\circ\iota^*$ in \cite[Section 4.4]{Saito90}, we have
$$
\dual \G\isom W_{r-1}\left(\iota^*\IC_\X^H[-1]\right).
$$
Moreover, $\iota^*\IC_\X^H[-1]$ is a mixed Hodge module of weight $\le r$ (\cite[(4.5.2)]{Saito90}). Hence, there is a natural morphism
$$
\alpha:\iota^*\IC_\X^H[-1]\to \IC_X^H
$$
which restricts to the identity over the smooth locus of $X$. Indeed, $\IC_X^H$ appears as a direct summand of $\gr^W_r\iota^*\IC_\X^H[-1]$, and the complementary summand is supported on a lower-dimensional subvariety. In particular, we obtain an inclusion
$$
\dual\G \subset \ker (\alpha).
$$
Again by duality, it suffices to show that $p(\dual(\ker(\alpha)))\ge m$, or equivalently,
\begin{equation}
\label{eqn:grDR vanishing for alpha}
\gr^F_p\DR(\ker(\alpha))=0\quad\text{for all}\quad p\ge -m+1.    
\end{equation}

Consider the commuting diagram
\begin{equation*}
\xymatrix{
{\Q_X^H[r]}\ar[drr]_{\gamma_X}\ar[rr]^-{\iota^*\gamma_{\X}[-1]}&{}& {\iota^*\IC_\X^H[-1]}\ar[d]^{\alpha}\\
{}&{}&{\IC_X^H}
}
\end{equation*}
where $\gamma_X=\alpha\circ(\iota^*\gamma_{\X}[-1])$, since $\alpha\circ(\iota^*\gamma_{\X}[-1])$ restricts to the identity over the smooth locus of $X$ (see also \cite[Section 4.5]{Saito90}). From the distinguished triangle
$$
\iota^*\K_\X^\bullet[-1]\to \Q_X^H[r]\xrightarrow{\iota^*\gamma_{\X}[-1]}\iota^*\IC_\X^H[-1]\xrightarrow{+1},
$$
we have the isomorphism
$$
\gr^F_p\DR(\iota^*\gamma_{\X}[-1]):\gr^F_p\DR(\Q_X^H[r])\xrightarrow{\sim} \gr^F_p\DR(\iota^*\IC_\X^H[-1])\quad\text{for all}\quad p\ge -m,
$$
by \eqref{eqn:pullback grDR vanishing} and condition $D_m$ for $\X$. Additionally, condition $D_{m-1}$ for $X$ implies the isomorphism
$$
\gr^F_p\DR(\gamma_X):\gr^F_p\DR(\Q_X^H[r])\xrightarrow{\sim} \gr^F_p\DR(\IC_X^H)\quad\text{for all}\quad p\ge -m+1.
$$
Therefore, the above commuting diagram yields the isomorphism
$$
\gr^F_p\DR(\alpha):\gr^F_p\DR(\iota^*\IC_\X^H[-1])\xrightarrow{\sim} \gr^F_p\DR(\IC_X^H)\quad\text{for all}\quad p\ge -m+1,
$$
which implies \eqref{eqn:grDR vanishing for alpha}. This completes the proof.
\end{proof}

Together with Chen's inversion of adjunction for higher singularities \cite[Theorem 1.2]{Chen25}, Proposition \ref{prop:limit MHS via Dm} implies Theorem \ref{thm:cohomologically insignificant}.

\begin{proof}[Proof of Theorem \ref{thm:cohomologically insignificant}]
Let $r=\dim X$. From the commutativity of the projective pushforward and the nearby cycle functor \cite[Theorem 2.14]{Saito90}, we have
$$
H^\bullet(\X_\infty,\Q)_1\isom\HH^{\bullet-r}(\psi_{f,1}(\Q_\X^H[r+1]))
$$
where $H^\bullet(\X_\infty,\Q)_1\subset H^\bullet(\X_\infty,\Q)$ is the unipotent eigenspace for the monodromy operator $T$.

Via cyclic base change, we may assume that $T$ is unipotent:
$$
H^\bullet(\X_\infty,\Q)\isom\HH^{\bullet-r}(\psi_{f,1}(\Q_\X^H[r+1])).
$$
From the assumption, $X$ has $m$-Du Bois singularities and $\X\sm X$ has $m$-rational singularities. Therefore, by Chen's inversion of adjunction \cite[Theorem 1.2]{Chen25}, $\X$ has $m$-rational singularities. In particular, $X$ satisfies $D_{m-1}$ and $\X$ satisfies $D_m$. Applying Proposition \ref{prop:limit MHS via Dm} and Saito's Hodge-to-de Rham spectral sequence, we obtain that
$$
\gr_F^p(\sp^k): \gr_F^pH^k(X,\C)\to \gr_F^pH^k(\X_\infty,\C)
$$
is an isomorphism for all $p\le m$.

In the Calabi-Yau type hypersurface case, we conclude from Lemma \ref{lem:isomorphism of cores of MHS} that
$$
\Core\left(H^{n-1}(X,\Q)\right)\isom\Core\left(H^{n-1}(\X_\infty,\Q)\right).
$$
\end{proof}

\begin{rmk}[$m=0$]
When $f:\X\to \Delta$ is projective and $\X$ has rational singularities, Proposition \ref{prop:limit MHS via Dm} implies that
$$
\gr_F^0(\sp^k): \gr_F^0H^k(X,\C)\to \gr_F^0H^k(\X_\infty,\C)_1
$$
is an isomorphism for all $k$. Here, $H^k(\X_\infty,\Q)_1$ is the unipotent monodromy eigenspace. Notably, this avoids any assumption on the singularities of the special fiber.
\end{rmk}

\begin{proof}[Proof of Corollary \ref{cor:maximal degeneration}]
Let $f:\X\to \Delta$ be a one-parameter smoothing of $X$. By definition, $f$ is a maximal degeneration if $N^{n-2m-1}\neq0$ where the nilpotent operator
$$
N:H^{n-1}(\X_\infty,\Q)\to H^{n-1}(\X_\infty,\Q)(-1).
$$
Since we have the isomorphisms
$$
N^j:\gr^W_{n-1+j}H^{n-1}(\X_\infty,\Q)\xrightarrow{\sim}\left(\gr^W_{n-1-j}H^{n-1}(\X_\infty,\Q)\right)(-j),
$$
the family $f$ is maximally degenerate if and only if we have the nonvanishing
$$
W_{2m}H^{n-1}(\X_\infty,\Q)\neq 0.
$$
However, by the invariance of Hodge numbers
$$
\dim\gr_F^pH^{n-1}(\X_\infty,\C)=\dim\gr_F^pH^{n-1}(X_t,\C)
$$
for the general fiber $X_t:=f^{-1}(t)$ and all $p$, we have
$$
\dim\gr_F^mH^{n-1}(\X_\infty,\C)=1\quad\text{and}\quad\dim\gr_F^{<m}H^{n-1}(\X_\infty,\C)=0.
$$
Therefore, $W_{2m}H^{n-1}(\X_\infty,\Q)\neq 0$ if and only if
$$
\dim\gr_F^mW_{2m}H^{n-1}(\X_\infty,\C)=1.
$$
In other words, $W_{2m}H^{n-1}(\X_\infty,\Q)\isom\Q^H(-m)$, where $\Q^H$ is the trivial Hodge structure of weight $0$. Moreover, $\Q^H(-m)$ should be the core of $H^{n-1}(\X_\infty,\Q)$. Hence, statement (3) is equivalent to
$$
\Core(H^{n-1}(\X_\infty,\Q))=\Q^H(-m).
$$

Applying Theorem \ref{thm:core of liminal sources}, we have $(2)\Rightarrow(1)\Leftrightarrow(3)$. It remains to prove that for a $m$-liminal source $\IC^H_Z(\mathbb V)$ supported on a minimal $m$-liminal center $Z\subset X$, an isomorphism
$$
\Core(\HH^0(Z,\IC^H_Z(\mathbb V)))\isom\Q^H(-m)
$$
implies that $\IC^H_Z(\mathbb V)\isom\Q_{\{x\}}^H(-m)$ for some $x\in X$. The above isomorphism says that $\IC^H_Z(\mathbb V)$ is of weight $2m$.

Recall that $\mathbb V$ is a polarizable variation of Hodge structure, with $\gr_F^m\mathbb V_\C\neq 0$ and $\gr_F^{< m}\mathbb V_\C=0$; see Definition \ref{defn:liminal sources and centers} and the following discussions. In particular, the weight of $\mathbb V$ is at least $2m$. Since the weight of $\IC^H_Z(\mathbb V)$ is $2m$, which is equal to the sum of the weight of $\mathbb V$ and the dimension of $Z$, the subvariety $Z$ is a closed point $x\in X$ and $\mathbb V=\Q^H(-m)$ supported at $x$. This completes the proof.
\end{proof}

\begin{rmk}[Reduction to a family over an algebraic curve]
\label{rmk:one parameter degeneration}
A one-parameter degeneration $f:\X\to \Delta$ is induced by a holomorphic map $\Delta\to \rm{Hilb}^{\rm an}$ to the analytification of the Hilbert scheme $\rm{Hilb}$ of $X:=f^{-1}(0)$. Let $u: \mathcal U\to\rm{Hilb}$ be the universal family. The sheaves $R^ku_*\Q_{\mathcal U}$ are constructible and, on a suitable algebraic stratification of $\rm{Hilb}$, underlie (admissible) variations of mixed Hodge structure; on each stratum, they are induced by the $k$-th cohomology of the fibers. Assume $\Delta\sm\{0\}$ maps into a stratum $S\subset {\rm Hilb}$, and write $h:\Delta\to \bar S$ for the induced map to its closure. After finitely many blow-ups centered in $\bar S\sm S$, $h$ lifts across the origin. Thus we obtain a resolution $\nu: \wt S\to \bar S$ whose boundary $E:=\nu^{-1}(\bar S\sm S)$ is a simple normal crossing divisor, and the map $\Delta\sm\{0\}\to \wt S$ extends to $\tilde h:\Delta\to \wt S$.

By the theory of admissible variations of mixed Hodge structure (see \cite{CKS86,Kashiwara86}), the graded pieces of the limit mixed Hodge structure
$$
\bigoplus_w\gr^W_w H^k(\X_\infty,\Q)
$$
depend only on the base point $\tilde h(0)\in \wt S$ and the local monodromy
$$
\pi_1(\Delta\sm\{0\})\to \pi_1^{\rm loc}(\wt S\sm E, \tilde h(0)).
$$
In a sufficiently small neighborhood of $\tilde h(0)$, the complement $\wt S\sm E$ is a product of punctured disks. Choosing an algebraic curve $C\to \wt S$ with the same local monodromy at $c\in C$ mapping to $\tilde h(0)$, we may replace $\Delta$ by the pointed algebraic curve $(C,c)$ for the proofs of Theorem \ref{thm:core of liminal sources} and Corollary \ref{cor:maximal degeneration}.
\end{rmk}

\subsection{Hodge Du-Bois numbers and liminal loci}

A projective variety $X$ with $m$-rational singularities exhibits fundamental symmetries of Hodge-Du Bois numbers:
$$
\underline h^{p,q} (X) =\underline h^{q,p} (X) =\underline h^{\dim X-p,\dim X-q} (X) =\underline h^{\dim X-q,\dim X-p} (X)
$$
for all $0\le p\le m$ and $0\le q\le \dim X$. This was established for the first two equalities in \cite{FL24, SVV23} and the third in \cite{PP25a}. For Calabi-Yau type hypersurfaces $X\subset \P^n$ of degree $d$ with with $m$-Du Bois singularities (and $m+1=\frac{n+1}{d}$), the Hodge numbers $\h^{p,q}(X)$ for $0\le p\le m$ and $0\le q\le n-1$ are equal to those of smooth hypersurfaces by \cite{FL24}. In particular, when $p=m$, we have
\begin{equation}
\label{eqn:p=m Hodge numbers}
\h^{m,m}(X)=\h^{m,n-1-m}(X)=1,\quad \h^{m,q}(X)=0 \;\;\;\mathrm{for~ all}\;\; q\neq m \;\;\mathrm{or}\;\; n-1-m.
\end{equation}
When $X$ has $m$-rational singularities and $p=n-1-m$, the symmetries above give
$$
\h^{n-1-m,m}(X)=\h^{n-1-m,n-1-m}(X)=1,\quad \h^{n-1-m,q}(X)=0 \;\;\;\mathrm{for~ all}\;\; q\neq m \;\;\mathrm{or}\;\; n-1-m.
$$
However, when $X$ is $m$-liminal, that is $m$-Du Bois but not $m$-rational, the symmetries always break. Notably, $\h^{n-1-m,m}(X)=1$ when $X$ is $m$-rational and $\h^{n-1-m,m}(X)=0$ when $X$ is $m$-liminal. Theorem \ref{thm:Hodge numbers} says more: the Hodge-Du Bois numbers of $X$ for $p=n-1-m$ and the Hodge-Du Bois numbers of the liminal locus for $p=0$ are related by an explicit formula.

\begin{proof}[Proof of Theorem \ref{thm:Hodge numbers}]
Recall that $m$-Du Bois hypersurface singularities are $(m-1)$-rational singularities. Hence, $X$ satisfies condition $D_{m-1}$. By \cite[Proposition 7.4]{PP25a}, we have
$$
\DB_X^{n-1-m}=I\DB_X^{n-1-m}
$$
which implies the equalities
$$
\h^{n-1-m,i}(X)=I\h^{n-1-m,i}(X)=I\h^{m,n-1-i}(X).
$$
The second equality is the Poincaré duality for intersection cohomology. From \eqref{eqn:Q to IC triangle}, we have the long exact sequence
$$
\dots\to\gr_F^m\HH^{j}(\K_X^\bullet)\to \gr_F^m H^{n-1+j}(X,\C)\to \gr_F^m \mathit{IH}^{n-1+j}(X,\C)\to \gr_F^m\HH^{j+1}(\K_X^\bullet)\to\cdots.
$$
By duality, we have
$$
\gr_F^m\HH^{j}(\K_X^\bullet)\isom\left(\gr_F^{-m}\HH^{-j}(\dual \K_X^\bullet)\right)^*\isom H^{-j}(\O_S)^*
$$
where the second isomorphism is a consequence of Proposition \ref{prop:lowest Hodge filtration of dual K}.

By \cite[Lemma 2.2]{MOPW}, we have $\dim S\le n-2-2m$ and
$$
H^{-j}(\O_S)\neq 0 \quad\text{only if}\quad -n+2+2m\le j\le 0.
$$
Additionally, Corollary \ref{cor:core of liminal center} says
\begin{equation}
\label{eqn:surjection of m-Hodge piece}
\gr_F^m\HH^{0}(\K_X^\bullet)\to \gr_F^m H^{n-1}(X,\C)    
\end{equation}
is surjective. Making the change of variables $i=m-j$, the above long exact sequence, together with \eqref{eqn:p=m Hodge numbers} yields the claimed identities for the Hodge-Du Bois numbers in Theorem \ref{thm:Hodge numbers}.

When the core of $H^{n-1}(X,\Q)$ has weight $w\le n-3$, we prove that the surjection \eqref{eqn:surjection of m-Hodge piece} is an isomorphism. Observe that the $m$-th Hodge piece of $\HH^{0}(\K_X^\bullet)$ is pure of weight $w$:
$$
\gr_F^m\gr^W_w\HH^{0}(\K_X^\bullet)=\gr_F^m\HH^{0}(\K_X^\bullet).
$$
This follows from the identity $\gr_F^m\HH^{0}(\K_X^\bullet)=\gr_F^m\HH^{0}(\K_X^\bullet|_S)$, together with Theorem \ref{thm:core of liminal sources} and Proposition \ref{prop:core of liminal source of MHM} applied to each irreducible component of $S$. Hence, from the above long exact sequence, we have
$$
\dots\to\gr_F^m\gr^W_w\mathit{IH}^{n-2}(X,\C)\to\gr_F^m\gr^W_w\HH^{j}(\K_X^\bullet)\to \gr_F^m\gr^W_w H^{n-1+j}(X,\C)\to \cdots
$$
and the vanishing $\gr^W_w\mathit{IH}^{n-2}(X,\C)=0$ implies that \eqref{eqn:surjection of m-Hodge piece} is injective (see e.g. \cite[Corollary 3.8]{PS08} for the strictness of Hodge and weight filtrations). In conclusion, $\dim \gr_F^m\HH^{0}(\K_X^\bullet)=h^{0,0}(S)=1$, and thus, $S$ is connected.
\end{proof}

For a quartic K3 surface $X\subset \P^3$ with Du Bois (non-rational) singularities, Theorem \ref{thm:Hodge numbers} determines the Hodge-Du Bois diamond completely, but two entries $\h^{1,1}(X)$ and $\h^{1,2}(X)$. The same proof works for Gorenstein semi-log canonical K3 surfaces, and the Hodge-Du Bois diamond is illustrated in the following example:

\begin{ex}[{\bf Gorenstein semi-log canonical K3 surfaces}]
\label{ex:K3 surface}
When $X$ is a K3 surface with Gorenstein Du Bois (non-rational) singularities, we have the following Hodge-Du Bois diamond (see \cite[Section 4]{PP25a}) of $X$ by Theorem \ref{thm:Hodge numbers}:
$$
\renewcommand{\arraystretch}{1.15}
\newcommand{\hd}[1]{\makebox[3.2em][c]{$#1$}}
\begin{array}{ccccc}
 & & \hd{h^1(\mathcal O_S)+1} & & \\
 & \hd{h^0(\mathcal O_S)-1} & & \hd{\h^{1,2}} & \\
 \hd{0} & & \hd{\h^{1,1}} & & \hd{1} \\
 & \hd{0} & & \hd{0} & \\
 & & \hd{1} & &
\end{array}
$$
Here, $S=\mathrm{nklt}(X)$ is the non-klt locus of $X$. Note that the sum along each horizontal row equals the corresponding Betti number. For any one-parameter smoothing of $X$, the degeneration is of Type II or Type III: it is Type II precisely when the core of $H^2(X,\Q)$ has weight 1, and Type III precisely when it has weight 0. This recovers the fact that $S$ is connected if Type III.

Moreover, if $X$ is log canonical, then $H^3(X,\Q)$ is pure (see \cite[Theorem 6.33]{PS08}). In particular, the Betti number $h^3(X)=2|S|-2$, where $|S|$ is the number of non-klt points on $X$. This only leaves $\h^{1,1}(X)$ undetermined, which can be recovered from the topological Euler characteristic of $X$.

Theorem \ref{thm:core of liminal sources} implies that all elliptic curves appearing in the resolution of $X$ are isogenous. In fact, they should be isomorphic by classical theory, but our method -- formulated for $\Q$-Hodge structures, without integer lattices -- proves only isogeny at the moment.
\end{ex}

This paper is framed for hypersurfaces, but the same methods yield the corresponding statement on Hodge-Du Bois numbers for Gorenstein semi-log canonical (strict) Calabi-Yau varieties, up to small modifications; we do not record the details.

\subsection{Thom-Sebastiani for liminal sources and examples}
\label{sec:Thom-Sebastiani for liminal sources}

We state a Thom-Sebastiani type theorem that enables explicit computation of $m$-liminal sources in many cases, including the example from the introduction. The result is extracted from Maxim-Saito-Schürmann’s Thom-Sebastiani theorem for filtered D-modules \cite{MSS20}; for completeness, we present a proof adapted to our setting. We then work out a range of examples.

\begin{thm}
\label{thm:Thom-Sebastiani for liminal sources}
Let $Y_1,Y_2$ be smooth varieties, and let $f_i\in \O_{Y_i}(Y_i)$ be nonzero regular functions with hypersurfaces
$$
X_i:=\{f_i=0\}\subset Y_i\quad (i=1,2).
$$
Assume $X_i$ is $m_i$-liminal. Let $\IC^H_{Z_i}(\mathbb V_i)$ be an $m_i$-liminal source of $X_i$ supported on a proper $m_i$-liminal center $Z_i\subsetneq X_i$. Then $X:=\{f_1\oplus f_2=0\}\subset Y_1\times Y_2$ is $(m_1+m_2+1)$-liminal, and
$$
\IC^H_{Z_1\times Z_2}(\mathbb V)(-1)
$$
is an $(m_1+m_2+1)$-liminal source of $X$, where $\mathbb V$ is the core of the external product $\mathbb V_1\boxtimes \mathbb V_2$, defined generically on $Z_1\times Z_2$.

Conversely, any nontrivial $(m_1+m_2+1)$-liminal source of $X$ (i.e. one with support $\neq X$) is obtained this way.
\end{thm}

Here, $\mathbb V_1\boxtimes\mathbb V_2$ is a Calabi-Yau type variation of pure Hodge structure, and its core $\mathbb V$ is the simple summand containing the top Hodge piece. Recall that a variation of (polarizable) pure Hodge structure is semisimple.

\begin{proof}
Denote by $Y:=Y_1\times Y_2$, $\dim Y=n$, and $\dim Y_i=n_i$ $(i=1,2)$. By \cite[Theorem 2]{MSS20}, we have
$$
\phi_{f_1\oplus f_2,1}(\O_Y,F_\bullet)=\left(\phi_{f_1,1}(\O_{Y_1},F_\bullet)\boxtimes \phi_{f_2,1}(\O_{Y_2},F_\bullet)\right)\oplus R
$$
where
$$
R:=\bigoplus_{\alpha_1+\alpha_2=1, \alpha_1,\alpha_2>0}\phi_{f_1,e(\alpha_1)}(\O_{Y_1},F_\bullet)\boxtimes \phi_{f_2,e(\alpha_2)}(\O_{Y_2},F_{\bullet+1})
$$
Following the notations in \emph{loc. cit.}, $(\O_{Y_i},F_\bullet)$ is the underlying filtered left $D_{Y_i}$-module of $\Q_{Y_i}^H[n_i]$, $e(\alpha):=e^{2\pi\sqrt{-1}\alpha}$, and $\phi_{f_i,e(\alpha_i)}$ is the $e(\alpha_i)$-eigenspace of the vanishing cycle functor $\phi_{f_i}$.
Note that the term
$$
T:=\phi_{f_1,1}(\O_{Y_1},F_\bullet)\boxtimes \phi_{f_2,1}(\O_{Y_2},F_\bullet)
$$
is the underlying filtered left $\D_Y$-module of $\phi_{f_1,1}\Q_{Y_1}^H[n_1]\boxtimes \phi_{f_2,1}\Q_{Y_2}^H[n_2]$, and thus, is naturally equipped with the weight filtration and the $\Q$-perverse sheaf. In particular, \emph{loc. cit.} implies that
$$
\phi_{f_1\oplus f_2,1}\Q^H_{Y}[n]=\left(\phi_{f_1,1}\Q_{Y_1}^H[n_1]\boxtimes \phi_{f_2,1}\Q_{Y_2}^H[n_2]\right)\oplus \mathcal R,
$$
where $\mathcal R\in {\rm MHM}(Y)$ with the underlying filtered left $\D_Y$-module $R$.

Thom-Sebastiani Theorem \ref{thm:Thom-Sebastiani for minimal exponents} for minimal exponents implies that $X$ is $(m_1+m_2+1)$-liminal. By \cite[Corollary 2.24]{Saito90}, we have the short exact sequence
$$
0\to\Q_X^H[n-1]\to \psi_{f,1}(\Q_Y^H[n])\xrightarrow{\rm can} \phi_{f,1}(\Q_Y^H[n])\to 0.
$$
where $f:=f_1\oplus f_2$. As in the proof of Proposition \ref{prop:limit MHS via Dm},
$$
\phi_{f,1}(\Q_Y^H[n])\isom{\rm image}(N:\psi_{f,1}(\Q_Y^H[n])\to \psi_{f,1}(\Q_Y^H[n])(-1))
$$
where the nilpotent operator $N:=\rm Var\circ can$. Recall that we have the isomorphism
$$
N^j:\gr^W_{n-1+j}\psi_{f,1}(\Q_Y^H[n])\xrightarrow{\sim}\left(\gr^W_{n-1-j}\psi_{f,1}(\Q_Y^H[n])\right)(-j) \quad\text{for all}\quad j>0.
$$
This implies that if $\M$ is a nontrivial $(m_1+m_2+1)$-liminal source of $X$, then $\M(-1)$ is a simple subquotient of $\phi_{f,1}(\Q_Y^H[n])$. By duality, $(\dual\M)(1)$ is a simple subquotient of
$$
\phi_{f,1}(\dual(\Q_Y^H[n]))\isom\phi_{f,1}(\Q_Y^H[n])(n).
$$
Recall the polarization $\dual(\Q_Y^H[n])\isom \Q_Y^H[n](n)$, and $\dual\phi_{f,1}=\phi_{f,1}\dual$ \cite[Proposition 2.6]{Saito90}. In other words, $(\dual\M)(1-n)$ is a simple subquotient of $\phi_{f,1}(\Q_Y^H[n])$.

By Definition \ref{defn:liminal sources and centers}, the index of the first nonzero filtration of $\dual\M$ is $m_1+m_2+1$. As a filtered left $\D_Y$-module, the index of $(\dual\M)(1-n)$ is $m_1+m_2+2$. On the other hand, \cite[(6)]{Saito17} (see also \cite[Theorem A]{MY23}) implies that the index of the first nonzero Hodge filtration of $T$ (resp. $R$) as a left $\D_Y$-module is $m_1+m_2+2$ (resp. $\ge m_1+m_2+3$). Therefore, $(\dual\M)(1-n)$ must be a subquotient of
$$
\phi_{f_1,1}\Q_{Y_1}^H[n_1]\boxtimes \phi_{f_2,1}\Q_{Y_2}^H[n_2].
$$
Again, by duality, $\M(-1)$ is a subquotient of $\phi_{f_1,1}\Q_{Y_1}^H[n_1]\boxtimes \phi_{f_2,1}\Q_{Y_2}^H[n_2]$.

From the property of the nilpotent operator, every simple subquotient of $\phi_{f_i,1}\Q_{Y_i}^H[n_i]$ is isomorphic to $\mathcal G_i(-k_i)$ for a simple subquotient $\G_i$ of $\K_{X_i}^\bullet$ and a positive integer $k_i$ ($i=1,2$). Therefore, by Lemma \ref{lem:box product of IC extension} below, $\M(-1)$ is a simple summand of the (semisimple) Hodge module
$$
\G_1(-k_1)\boxtimes \G_2(-k_2).
$$
Let $\G_i:=\IC_{Z_i}^H(\mathbb V_i)$ for $i=1,2$, then
$$
\left(\G_1(-k_1)\boxtimes \G_2(-k_2)\right)(1)=\IC_{Z_1\times Z_2}^H(\mathbb V_1\boxtimes \mathbb V_2)(1-k_1-k_2).
$$
Denote by $\mathbb V$, the summand of $\mathbb V_1\boxtimes \mathbb V_2(2-k_1-k_2)$ such that $\M=\IC_{Z_1\times Z_2}^H(\mathbb V)(-1)$. For $\M$ to satisfy the definition of $(m_1+m_2+1)$-liminal source, we need
$$
\gr_F^{m_1+m_2}\mathbb V\neq 0\quad\mathrm{and}\quad\gr_F^{<m_1+m_2}\mathbb V= 0.
$$
See \eqref{eqn:Hodge piece nonvanishing for VHS}. Recall that $\gr_F^{<m_i}\mathbb V_i=0$. Therefore, we must have $k_1=k_2=1$ and $\gr_F^{m_i}\mathbb V_i\neq 0$ for $i=1,2$. Therefore, $\IC_{Z_i}^H(\mathbb V_i)$ is an $m_i$-liminal source of $X_i$ for $i=1,2$, and
$$
\M=\IC_{Z_1\times Z_2}^H(\mathbb V)(-1)
$$
where $\mathbb V$ is the core of $\mathbb V_1\boxtimes \mathbb V_2$.

Conversely, any such $\M$ is an $(m_1+m_2+1)$-liminal source of $X$. Indeed, the above argument implies that $\M(-1)$ appears as a simple subquotient of $\phi_{f,1}(\Q_Y^H[n])$. Hence, $\M(k-1)$ is a simple subquotient of $\K^\bullet_X$ for some $k>0$. By \eqref{eqn:Hodge piece nonvanishing for VHS}, $k=1$ and $\M$ should be an $(m_1+m_2+1)$-liminal source.
\end{proof}

In the above proof, we used the following lemma on the external product of minimal extensions.

\begin{lem}
\label{lem:box product of IC extension}
Let $Z_1$ and $Z_2$ be irreducible varieties. Let $\IC_{Z_i}^H(\mathbb V_i)$ be the pure Hodge module associated to a polarizable variation $(\mathbb V_i, F^\bullet)$ of $\Q$-Hodge structure defined on a smooth Zariski dense open subset of $Z_i$ for $i=1,2$. Then we have an isomorphism of pure Hodge modules
$$
\IC_{Z_1\times Z_2}^H(\mathbb V_1\boxtimes \mathbb V_2)\isom\IC_{Z_1}^H(\mathbb V_1)\boxtimes \IC_{Z_2}^H(\mathbb V_2)
$$
where $\mathbb V_1\boxtimes \mathbb V_2$ is the polarizable variation of $\Q$-Hodge structure defined by the external product on a smooth Zariski dense open subset of $Z_1\times Z_2$.
\end{lem}

\begin{proof}
Recall from \cite[(2.17.4)]{Saito90}, the definition of the functor
$$
\cdot\;\boxtimes\;\cdot:{\rm MHM}(Z_1)\times{\rm MHM}(Z_2)\to {\rm MHM}(Z_1\times Z_2).
$$
By \cite[(3.8.5)]{Saito90}, we have canonical isomorphisms
$$
(j_!j^{-1}\M)\boxtimes \N\isom j_!j^{-1}(\M\boxtimes\N), \quad (j_*j^{-1}\M)\boxtimes \N\isom j_*j^{-1}(\M\boxtimes\N)
$$
for an open embedding $j:U_1\hookrightarrow Z_1$ such that $Z_1\sm U_1$ is a locally principal divisor. Since there exists a smooth open subvariety $U_i\subset Z_i$ such that $Z_i\sm U_i$ is a locally principal divisor and $\mathbb V_i$ is a variation of Hodge structure on $U_i$ ($i=1,2$), the minimal extension of $\mathbb V_1\boxtimes\mathbb V_2$ from $U_1\times U_2$ to $Z_1\times Z_2$ is naturally isomorphic to
$$
\IC_{Z_1}^H(\mathbb V_1)\boxtimes \IC_{Z_2}^H(\mathbb V_2)
$$
as desired.
\end{proof}

We list $m$-minimal sources for normal crossing singularities and for affine cones over smooth Calabi-Yau type hypersurfaces.

\begin{ex}[{\bf Normal crossing singularities}]
A normal crossing singularity
$$
X:=\{x_1\dots x_n=0\}\subset \C^n
$$
is $0$-liminal, and admits a stratification consisting of coordinate planes. Note that
$$
\gr^W_w\Q_X^H[n-1]=\bigoplus_{L\subset \C^n} \Q^H_L[w] \quad \text{for all} \quad 0\le w\le n-1,
$$
where the direct sum runs over every $w$-dimensional coordinate planes $L$ in $\C^n$. This can be easily obtained by induction on the dimension of $X$. In particular, every $\Q^H_L[w]$ is a $0$-liminal source of $X$, and $\Q^H_{\{0\}}$ is the $0$-liminal source with the minimal $0$-liminal center.
\end{ex}

\begin{ex}[{\bf Cones over smooth Calabi-Yau type hypersurfaces}]
An affine cone over a smooth hypersurface $X\subset \P^n$ of degree $d$ with $\frac{n+1}{d}=m+1\in\Z$,
$$
{\rm Cone}(X)\subset \C^{n+1},
$$
is $m$-liminal (see Theorem \ref{thm:minimal exponent of cone}). The cone point $\{0\}$ is the only singular point, and thus the only nontrivial $m$-liminal center ($\neq {\rm Cone}(X)$). As in the proof of Theorem \ref{thm:higher singularities of cone}(2), denote by $\mu:\wt C\to{\rm Cone}(X)$ the blow-up along the cone point. Then, we have Saito's Decomposition Theorem
$$
\mu_*\Q_{\wt C}[n]\isom \IC_{{\rm Cone}(X)}^H\oplus\M^\bullet
$$
where $\M^\bullet$ is supported on $\{0\}$. Applying the pullback $\iota^*$ where $\iota:\{0\}\to {\rm Cone}(X)$, we obtain
$$
\mu_*\Q_X^H[n]\isom \iota^*\IC_{{\rm Cone}(X)}^H\oplus \M^\bullet
$$
by the proper base change theorem \cite[(4.4.3)]{Saito90}. Since $\K_{{\rm Cone}(X)}^\bullet$ is also supported on $\{0\}$, it is obvious from the pullback $\iota^*$ of \eqref{eqn:Q to IC triangle} for $\rm{Cone}(X)$ that
$$
\mathcal H^{-1}(\iota^*\IC_{{\rm Cone}(X)}^H)\isom \K_{{\rm Cone}(X)}^\bullet.
$$
Hence, the nontrivial $m$-liminal source, which is a subquotient of $\K_{{\rm Cone}(X)}^\bullet$, is the Hodge module
$$
\Core\left(H^{n-1}(X,\Q)\right)_{\{0\}}
$$
associated to the pure Hodge structure $\Core(H^{n-1}(X,\Q))$ supported at the cone point.
\end{ex}

Combined with Theorem \ref{thm:Thom-Sebastiani for liminal sources} and results obtained throughout this paper, we determine GIT (semi)stability, compute the core of the middle cohomology, and identify the nilpotency index of the limit mixed Hodge structure for various Calabi-Yau type hypersurfaces.

\begin{ex}[{\bf Sum of independent normal crossing monomials}]
\label{ex:sum of normal crossing}
Fix $m\ge 0$ and $d\ge 3$. Let $\{x_{i,j}\}_{0\le i\le m, 1\le j\le d}$ denote homogeneous coordinates on $\P^{(m+1)d-1}$. Define
$$
X=\left\{\sum_{i=0}^m\left(\prod_{j=1}^d x_{i,j}\right)=0\right\}\subset \P^{(m+1)d-1}.
$$
On an affine chart $\{x_{0,1}\neq 0\}$, the local equation of $X$ is
$$
x_{0,2}\dots x_{0,d}+\sum_{i=1}^m\left(\prod_{j=1}^d x_{i,j}\right)=0.
$$
First of all, $X$ is $m$-liminal by Theorem \ref{thm:Thom-Sebastiani for minimal exponents}. Note that every monomial is a normal crossing singularity, which does not share a variable with any other monomial. Additionally, for each normal crossing singularity, $\Q^H_{\{0\}}$ is a $0$-liminal source. Therefore, applying Theorem \ref{thm:Thom-Sebastiani for liminal sources},
$$
\Q^H_{\{0\}}\boxtimes\dots\boxtimes\Q^H_{\{0\}}(-m)=\Q^H_{\{0\}}(-m)
$$
is a $m$-liminal source of $X$. Hence, Corollary \ref{cor:maximal degeneration} applies.

In conclusion, $X$ is GIT semistable, every one-parameter smoothing is a maximal degeneration, and
$$
\Core\left(H^{(m+1)d-2}(X,\Q)\right)=\Q^H(-m).
$$
\end{ex}

\begin{ex}[{\bf Cubic sevenfolds}]
A cubic sevenfold $X\subset \P^8$ is a Calabi-Yau type hypersurface. When $X$ is smooth, the Hodge numbers of $H^7(X,\Q)$ are:
$$
0\quad0\quad1\quad84\quad84\quad1\quad0\quad0,
$$
with $h^{7,0}(X)$ on the left and $h^{0,7}(X)$ on the right.

\noindent
(1) Let $X=\left\{f(x_0,\dots,x_5)+x_6x_7x_8=0\right\}\subset \P^8$, such that $Y:=\{f=0\}\subset \P^5$ is a smooth cubic fourfold. By Theorem \ref{thm:Thom-Sebastiani for minimal exponents}, $X$ is $2$-liminal, and $X$ is GIT semistable. On an affine chart $\{x_8\neq 0\}$, the local equation of $X$ is
$$
f(x_0,\dots,x_5)+x_6x_7=0.
$$
Applying Theorem \ref{thm:Thom-Sebastiani for liminal sources}, we obtain a $2$-liminal source
$$
\Core\left(H^4(Y,\Q)\right)(-1)_{\{[0:\dots:0:1]\}}
$$
associated to $\Core\left(H^4(Y,\Q)\right)(-1)$ supported on $[0:\dots:0:1]$. Applying Theorem \ref{thm:core of liminal sources}, we conclude that
$$
\Core\left(H^7(X,\Q)\right)=\Core\left(H^4(Y,\Q)\right)(-1),
$$
and the nilpotent operator $N$ of the limit mixed Hodge structure of any one-parameter smoothing satisfies $N\neq 0$, $N^2=0$.

\noindent
(2) Let $X=\left\{g(x_0,x_1,x_2)+x_3x_4x_5+x_6x_7x_8=0\right\}\subset \P^8$, such that $C:=\{g=0\}\subset \P^2$ is a smooth cubic plane curve. As above, $X$ is $2$-liminal, and $X$ is GIT semistable. We obtain a $2$-liminal source
$$
\Core\left(H^1(C,\Q)\right)(-2)_{\{[0:\dots:0:1]\}}.
$$
In conclusion, we have
$$
\Core\left(H^7(X,\Q)\right)=\Core\left(H^1(C,\Q)\right)(-2),
$$
and the nilpotent operator $N$ of any limit mixed Hodge structure satisfies $N^2\neq 0$, $N^3=0$.

\noindent
(3) Let $X=\left\{x_0x_1x_2+x_3x_4x_5+x_6x_7x_8=0\right\}\subset \P^8$. This is subsumed in Example \ref{ex:sum of normal crossing}: $X$ is $2$-liminal, GIT semistable, every one-parameter smoothing is a maximal degeneration, and
$$
\Core\left(H^7(X,\Q)\right)=\Q^H(-2).
$$
\end{ex}

\bibliographystyle{alpha}
\bibliography{refs}

\end{document}